\documentclass[10pt,letterpaper]{amsart}
\usepackage{latexsym, amscd, amsfonts, eucal, mathrsfs, amsmath, amssymb, amsthm, tikz-cd, xr, stmaryrd, amsxtra, mathtools, MnSymbol,bbm,mathrsfs}
\usepackage[alphabetic]{amsrefs}
\usepackage{enumitem}
\usepackage[hidelinks]{hyperref}
\usepackage{chngcntr}
\usepackage{verbatim}

\counterwithin{equation}{section}
\newtheorem{theorem}[equation]{Theorem}
\newtheorem{lemma}[equation]{Lemma}

\newtheorem{proposition}[equation]{Proposition}
\newtheorem{corollary}[equation]{Corollary}

\newtheorem*{claim}{Claim}

\theoremstyle{definition}
\newtheorem{sub}[equation]{}

\newtheorem{prop-con}[equation]{Proposition--Construction}
\newtheorem{notation}[equation]{Notation}

\newtheorem{construction}[equation]{Construction}
\newtheorem*{observation}{Observation}

\newtheorem{definition}[equation]{Definition}
\newtheorem{definition/lemma}[equation]{Definition--Lemma}
\newtheorem{example}[equation]{Example}

\theoremstyle{remark}
\newtheorem{remark}[equation]{Remark}
\title[Cycles relations in the affine grassmannian and applications]{Cycles relations in the affine grassmannian and applications to Breuil--M\'ezard for $G$-crystalline representations}

\author[R. Bartlett]{Robin Bartlett}
\address{University of Glasgow, United Kingdom}
\email{robin.bartlett.math@gmail.com}
\begin{document}
	
	\begin{abstract}
		For a split reductive group $G$ we realise identities in the Grothendieck group of $\widehat{G}$-representations in terms of cycle relations between certain closed subschemes inside the affine grassmannian. These closed subschemes are obtained as a degeneration of $e$-fold products of flag varieties and, under a bound on the Hodge type, we relate the geometry of these degenerations to that of moduli spaces of $G$-valued crystalline representations of $\operatorname{Gal}(\overline{K}/K)$ for $K/\mathbb{Q}_p$ a finite extension with ramification degree $e$. By transferring the aforementioned cycle relations to these moduli spaces we deduce one direction of the Breuil--M\'ezard conjecture for $G$-valued crystalline representations with small Hodge type.
	\end{abstract}
	\maketitle
	\tableofcontents
	
	\section{Introduction}
	
	The goal of this paper is to prove new results towards the Breuil--M\'ezard conjecture for crystalline representations valued in a connected split reductive group $G$. This open conjecture is a combinatorial shadow of the expected $p$-adic Langlands correspondence and describes multiplicities of irreducible components inside moduli spaces of $p$-adic Galois representations. We refer to the introduction of \cite{B21} or \cite[\S1.7]{EG19} for more details. To achieve our goal we describe new structures in the affine grassmannian which exhibit Breuil--M\'ezard phenomena, and relate these to moduli of Galois representations. When $G = \operatorname{GL}_2$ these results were proven in \cite{B21}, and what we do here extends these techniques to general $G$.
	
	There are two main theorems we prove. The first is purely algebro-geometric and describes an analogue of the Breuil--M\'ezard conjecture for certain closed subschemes inside the affine grassmannian. For this we fix a split connected reductive group $G$, together with a  choice of maximal torus and Borel $T \subset B$ and write $\widehat{G}$ for the dual group. We let $\operatorname{Gr}_{G,\mathbb{F}}$ denote the associated affine grassmannian over a field $\mathbb{F}$ and, for an integer $e\geq 1$ and any $e$-tuple of dominant cocharacters $\mu = (\mu_1,\ldots,\mu_e)$ of $G$, we define closed subschemes $M_{\mu,\mathbb{F}} \subset \operatorname{Gr}_{G,\mathbb{F}}$ as degenerations of an $e$-fold product of flag varieties $G/P_{\mu_1} \times \ldots \times G/P_{\mu_e}$ (see the first bullet point after Theorem~\ref{thmA-p1} for more details). We then show that the geometry of these $M_{\mu,\mathbb{F}}$ as $\mu$ varies can be described in terms of the representation theory of $e$-fold tensor products of the highest weight $\widehat{G}$-representations $W(\mu_i)$. More precisely, we prove:
	
	\begin{theorem}\label{thmA-p1}
		Assume $G$ admits a twisting element $\rho$, i.e. a cocharacter pairing to $1$ with all simple roots of $G$. Then, for any tuple $\mu = (\mu_1,\ldots,\mu_e)$ of strictly dominant cocharacters of $G$ (i.e. $\mu_i-\rho$ is dominant) satisfying
		$$
		\sum_{i=1}^e \langle \alpha^\vee, \mu_i \rangle \leq \operatorname{char}\mathbb{F} + e-1
		$$
		for all roots $\alpha^\vee$ of $G$ (when $\operatorname{char}\mathbb{F} =0$ this condition is not needed) one has identities of $e\operatorname{dim}G/B$-dimensional cycles (i.e. in the free abelian group on $e\operatorname{dim}G/B$-dimensional integral closed subschemes)
		$$
		[M_{\mu,\mathbb{F}}] = \sum_{\lambda} m_\lambda [M_{(\lambda+\rho,\rho,\ldots,\rho),\mathbb{F}}]
		$$
		where $m_\lambda$ denotes the multiplicity of $W(\lambda)$ inside $\bigotimes_{i=1}^e W(\mu_i-\rho)$. Furthermore, each $M_{(\lambda+\rho,\rho,\ldots,\rho),\mathbb{F}}$ appearing in this sum is irreducible and generically reduced. 
	\end{theorem}
	
	A twisting element $\rho$ need not always exist (e.g. if $G= \operatorname{SL}_2$) but will whenever $\widehat{G}$ is semi-simple and simply connected, or has simply connected derived subgroup (e.g. if $G = \operatorname{GL}_n$). Twisting elements can also always be found after replacing $\widehat{G}$ by a $\mathbb{G}_m$-extension \cite[\S5]{BG14}. 
	
	The following three points identify the crucial inputs into the proof of Theorem~\ref{thmA-p1}:
	\begin{itemize}
		\item In order to construct the degenerations $M_{\mu,\mathbb{F}}$ we choose a discrete valuation ring $\mathcal{O}$ with residue field $\mathbb{F}$ and an $e$-tuple of pairwise distinct $\pi_1,\ldots,\pi_e$ in the maximal ideal of $\mathcal{O}$. Viewing $G$ as a group over $\operatorname{Spec}\mathcal{O}$ we use $\pi_1,\ldots,\pi_e$ to extend $\operatorname{Gr}_{G,\mathbb{F}}$ to and ind-$\mathcal{O}$-scheme $\operatorname{Gr}_G$ (as a specialisation of the Beilinson--Drinfeld grassmannian over $\mathbb{A}^e_{\mathbb{Z}}$) with generic fibre an $e$-fold product of $\operatorname{Gr}_{G,\operatorname{Frac}\mathcal{O}}$'s. Any tuple $\mu =(\mu_1,\ldots,\mu_e)$ then determines a closed immersion
		$$
		G/P_{\mu_1} \times \ldots \times G/P_{\mu_e} \hookrightarrow \operatorname{Gr}_G \otimes_{\mathcal{O}} \operatorname{Frac}\mathcal{O}
		$$
		for $P_{\mu_i}$ the associated parabolic subgroup. We set $M_{\mu}$ equal the closure of this embedding inside $\operatorname{Gr}_G$ and $M_{\mu,\mathbb{F}}$ the fibre over $\operatorname{Spec}\mathbb{F}$.
		\item Next we establish cycle identities 
		\begin{equation}\label{eq-nonrepcycleidentity}
			[M_{(\mu_1+\rho,\ldots,\mu_e+\rho)} \otimes_{\mathcal{O}} \mathbb{F}] = \sum_{\lambda} n_\lambda [M_{(\lambda_+\rho,\rho,\ldots,\rho)} \otimes_{\mathcal{O}}\mathbb{F}]
		\end{equation}
		for $n_{\lambda} \in \mathbb{Z}_{\geq 0}$, a priori, with no representation theoretic interpretation. This is essentially a topological calculation and is achieved by giving an explicit moduli description of closed subschemes $\operatorname{Gr}_{G,\leq \mu}^\nabla \subset \operatorname{Gr}_G$ approximating $M_{\mu,\mathbb{F}}$ in the sense that 
		$$
		M_{\mu,\mathbb{F}} \subset \operatorname{Gr}^\nabla_{\leq \mu}
		$$
		and that the top dimensional irreducible components of $\operatorname{Gr}^\nabla_{\leq \mu} \otimes_{\mathcal{O}} \mathbb{F}$ generically identify with the $M_{(\lambda+\rho,\rho,\ldots,\rho),\mathbb{F}}$ for dominant $\lambda$ with $\lambda+e\rho \leq \mu_1+\ldots+\mu_e$. The moduli description of $\operatorname{Gr}^\nabla_{\leq \mu}$ is primarily Lie theoretic, and can be interpreted as an infinitesimal version of being fixed by the loop rotation, or as an incarnation of Griffiths transversality for Breuil--Kisin modules. It is in these calculations that the restriction on $\operatorname{char}\mathbb{F}$ and the strict dominance of the $\mu_i$ play a crucial role.
		\item To finish the proof it remains to identify the $n_\lambda$'s with the representation theoretic multiplicities $m_\lambda$. For this we consider the $G$-equivariant ample line bundle $\mathcal{L}$ on $\operatorname{Gr}_G$ obtained by pulling back the determinant bundle along the adjoint representation. The restriction of $\mathcal{L}$ to $M_{(\mu_1+\rho,\ldots,\mu_e+\rho)} \otimes_{\mathcal{O}} \operatorname{Frac}\mathcal{O}$ can be expressed explicitly as a product of equivariant line bundles on flag varieties. Using ampleness of $\mathcal{L}$ and flatness of $M_{(\mu_1+\rho,\ldots,\mu_e+\rho)}$ over $\mathcal{O}$ we are therefore able to identify, for sufficiently large $n$,
		$$
		H^0(M_{(\mu_1+\rho,\ldots,\mu_e+\rho)} \otimes_{\mathcal{O}} \mathbb{F}, \mathcal{L}^{\otimes n}) = \bigotimes_{i=1}^e W(np(\mu_i +\rho))
		$$
		as $G$-representations over $\mathbb{F}$. Here $p(\eta) = \sum_{\alpha^\vee} \langle \alpha^\vee,\eta \rangle \alpha^\vee$ is the homomorphism from cocharacters to characters induced by the Killing form. In Section~\ref{sec-equivariantsheaves} we show that if $X$ is a $T$-equivariant scheme admitting an equivariant ample line bundle $\mathcal{L}_X$ then any identity of cycles between $T$-equivariant closed subschemes induces an asymptotic formula between the global sections of high powers of $\mathcal{L}_X$ inside the Grothendieck group of $T$-representations. Applying this to \eqref{eq-nonrepcycleidentity} and $\mathcal{L}$ produces a formula involving the $n_\lambda$ which asymptotically relates $[\bigotimes_{i=1}^e W(np(\mu_i +\rho))]$ and $[W(np(\lambda+\rho)) \otimes W(np(\rho))^{\otimes e-1}]$, in the sense that an appropriate difference is polynomial in $n$ of degree $< e\operatorname{dim}G/B$. In Sections~\ref{sec-proofofcyclethm} and~\ref{sec-reptheory} we show, using manipulations with the Weyl character formula, that such asymptotic relations can only occur if $m_\lambda = n_\lambda$, proving Theorem~\ref{thmA-p1}.
	\end{itemize}
	Our second main theorem specialises to the case of residue characteristic $p$, and proves new instances of one direction of the Breuil--M\'ezard conjecture (namely that Galois multiplicities are $\leq$ automorphic multiplicities). It is deduced from Theorem~\ref{thmA-p1} by relating the geometry of the $M_\mu$'s to that of moduli spaces of crystalline representations.
	
	\begin{theorem}\label{thmA-p2}
		Continue to assume $G$ admits a twisting element $\rho$ and let $K/\mathbb{Q}_p$ be a finite extension with ramification degree $e$. Let $R_{\overline{\rho}}^{\square,\operatorname{cr},\mu}$ denote the framed deformation ring of a fixed $\overline{\rho}:G_K \rightarrow G(\overline{\mathbb{F}}_p)$ classifying $G$-valued crystalline representations with Hodge type $\mu =(\mu_{\kappa})_{\kappa:K \rightarrow \overline{\mathbb{Q}}_p}$. Suppose further that each $\mu_\kappa$ is strictly dominant and, for each $\kappa_0:k\rightarrow \overline{\mathbb{F}}_p$,
		$$
		\sum_{\kappa|_k = \kappa_0} \langle \alpha^\vee, \mu_\kappa \rangle \leq p
		$$
		for all roots $\alpha^\vee$ of $G$. Then, as $\operatorname{dim}G + [K:\mathbb{Q}_p]\operatorname{dim}G/B$-dimensional cycles,
		$$
		[R_{\overline{\rho}}^{\square,\operatorname{cr},\mu} \otimes_{\overline{\mathbb{Z}}_p} \overline{\mathbb{F}}_p] \leq \sum_{\lambda} m_\lambda [R_{\overline{\rho}}^{\square,\operatorname{cr},(\lambda+\rho,\rho,\ldots,\rho)}  \otimes_{\overline{\mathbb{Z}}_p} \overline{\mathbb{F}}_p]	$$
		where again $m_\lambda$ denotes the multiplicity of the Weyl module $W(\lambda)$ of highest weight $\lambda$ inside $\bigotimes_{i=1}^e W(\mu_i-\rho)$.
	\end{theorem}
	
	When $G = \operatorname{GL}_n$ we can identify the $\mu_{\kappa}$ with $n$-tuples of integers $(\mu_{\kappa,1},\ldots,\mu_{\kappa,n})$ so that $\rho = (n-1,n-2,\ldots,1,0)$. Then the bound on the $\mu_\kappa$'s is equivalent to asking that
	$$
	\sum_{\kappa|_k = \kappa_0} (\mu_{\kappa,1}-\mu_{\kappa,n}) \leq p
	$$
	In particular, we see that the theorem, roughly speaking, accesses Hodge types contained in the interval $[0,p/e]$. Note that since we also ask each $\mu_\kappa$ to be strictly dominant we have $\mu_{\kappa,1} - \mu_{\kappa,n} \geq n-1$ and so $\mu$ as in Theorem~\ref{thmA-p2} will only exist when $e(n-1) \leq p$.
	
	The crux of Theorem~\ref{thmA-p2}'s proof lies in connecting crystalline representations to the $M_{\mu,\mathbb{F}}$. This passage is achieved via the intermediary of Breuil--Kisin modules associated to crystalline representations. To explain this we assume, for notational simplicity, that $K$ is totally ramified over $\mathbb{Q}_p$ and let $\mathcal{O}$ be the ring of integers in a finite extension containing the Galois closure of $K$ and with residue field $\mathbb{F}$. Let $\mathfrak{M}$ denote the Breuil--Kisin module associated to a crystalline representation valued in $G(\mathcal{O})$. This is a $G$-torsor on $\operatorname{Spec}\mathcal{O}[[u]]$ equipped with an isomorphism $\varphi^*\mathfrak{M}[\frac{1}{E(u)}] \cong \mathfrak{M}[\frac{1}{E(u)}]$, where $\varphi$ is given by $u \mapsto u^p$ and $E(u)$ is the minimal polynomial of a fixed choice of uniformiser $\pi \in K$. Any trivialisation $\iota$ of $\mathfrak{M}$ over $\operatorname{Spec}\mathcal{O}[[u]]$ produces an $\mathcal{O}$-valued point $\Psi(\mathfrak{M},\iota) \in\operatorname{Gr}_G$ describing the relative position of $\mathfrak{M}$ and $\varphi^*\mathfrak{M}$. We prove that if the Hodge type $\mu$ of the crystalline representation satisfies the bound from Theorem~\ref{thmA-p2} then $\Psi(\mathfrak{M},\iota) \otimes_{\mathcal{O}} \mathbb{F} \in M_{\mu,\mathbb{F}}$. The following describes the central idea:
	\begin{itemize}
		\item Consideration of Kisin's original construction of $\mathfrak{M}$ from \cite{Kis06} shows that, without any bound on $\mu$, one has 
		$$
		X \cdot \Psi(\mathfrak{M},\iota)[\tfrac{1}{p}] \in M_{\mu}[\tfrac{1}{p}]
		$$
		for $X\in  G(\mathcal{O}^{\operatorname{rig}}[\frac{1}{\lambda}])$ the automorphism inducing a Frobenius equivariant identification $\mathfrak{M} \otimes \mathcal{O}^{\operatorname{rig}}[\frac{1}{\lambda}] \cong D \otimes \mathcal{O}^{\operatorname{rig}}[\frac{1}{\lambda}]$. Here $D$ is the filtered $\varphi$-module associated to the crystalline representation, $\mathcal{O}^{\operatorname{rig}}$ is the ring of power series convergent on the open unit disk over $\operatorname{Frac}\mathcal{O}$, and $\lambda = \prod_{n \geq 0} \varphi^n(\frac{E(u)}{E(0)})$.
		\item While $X$ will almost never be defined integrally, calculations of \cite{Liu15} and \cite{GLS} bound the order of $\frac{1}{p}$ in the coefficients of $X$ (more precisely, they bound the coefficients of the monodromy operator from which $X$ can be recovered). Using this we are able to show that a truncation $X^{\operatorname{trun}}$ of $X$ modulo sufficient high powers of $E(u)$ (depending upon $\mu$)  is such that $X^{\operatorname{trun}}$ is integral, $\equiv 1$ modulo the maximal ideal of $\mathcal{O}$,  and still satisfies 
		$$
		X^{\operatorname{trun}} \cdot \Psi(\mathfrak{M},\iota)[\tfrac{1}{p}] \in M_{\mu}[\tfrac{1}{p}]
		$$
		Thus $\Psi(\mathfrak{M},\iota) \otimes_{\mathcal{O}} \mathbb{F} \in M_{\mu,\mathbb{F}}$ as desired.
	\end{itemize}
	
	More generally this construction works whenever the crystalline representation is valued in $G(A)$ for $A$ any finite flat $\mathcal{O}$-algebra for which the associated Breuil--Kisin module is a $G$-torsor on $\operatorname{Spec}A[[u]]$ (which is not automatic). As a consequence, if one considers the standard diagram (whose construction goes back to \cite{KisFF} and \cite{PR09})
	$$
	\begin{tikzcd}
		& & \widetilde{Y} \ar[dr,"\Psi"] \ar[dl,"\Gamma"]&   \\
		X & \ar[l,"\Theta"] Y  & & \operatorname{Gr}_G  
	\end{tikzcd}
	$$
	in which $X$ denotes an appropriate moduli space of crystalline Galois representations, $Y$ a moduli space of Breuil--Kisin modules associated to crystalline Galois representations, and $\widetilde{Y}$ a rigidification of $Y$ classifying an additional choice of trivialisation, then the restriction  of $\Psi$ to the fibre over $\operatorname{Spec}\mathbb{F}$ of the closed locus $\widetilde{Y}^\mu \subset \widetilde{Y}$ of Breuil--Kisin modules of Hodge type $\mu$ factors through $M_{\mu,\mathbb{F}}$. An additional (but much simpler) argument shows that, under the bound on $\mu$, the morphism $\Psi$ is smooth over $M_{\mu,\mathbb{F}}$.  As a result the cycle identities in Theorem~\ref{thmA-p1} can be pulled back along $\Psi$, descended along $\Gamma$, and then pushed forward along the proper morphism $\Theta$. Since we only know that the preimage of $M_{\mu,\mathbb{F}}$ contains $\widetilde{Y}^\mu \otimes \mathbb{F}$ this process produces an identity of cycles 
	$$
	[X_0^\mu] = \sum_{\lambda} n_\lambda [X_0^{(\lambda+\rho,\rho,\ldots,\rho)}]
	$$
	inside $X \otimes \mathbb{F}$ with $[X^\mu \otimes \mathbb{F}] \leq [X_0^\mu]$ for $X^\mu \subset X$ the locus of crystalline representations of Hodge type $\mu$. However, by the last part of Theorem~\ref{thmA-p1}, and the fact that over each $M_{\mu,\mathbb{F}}$ the morphism $\Psi$ is smooth with irreducible fibres, we can additionally show that $X_0^{(\lambda+\rho,\rho,\ldots,\rho)}$ is irreducible and generically reduced. Thus $[X_0^{(\lambda+\rho,\rho,\ldots,\rho)}] = [X^{(\lambda+\rho,\rho,\ldots,\rho)} \otimes \mathbb{F}]$. This gives the inequality in Theorem~\ref{thmA-p2}. The most natural choice for $X$ would be the moduli stack of Galois representations, constructed in \cite{EG19} when $G = \operatorname{GL}_n$. Since the case of more general groups has yet to be written up (though this is likely to be addressed by work of Lin, see for example \cite{Lin23}) we take, in the body of the text, $X$ equal to the formal spectrum of a Galois deformation ring.
	
	The methods of this paper do not appear to give any way to prove an equality in Theorem~\ref{thmA-p2}. This would come down to showing that $[X^\mu \otimes \mathbb{F}] \leq [X_0^\mu]$ is an equality which ultimately, is a question about producing crystalline lifts with Hodge type $\mu$ of torsion Breuil--Kisin modules whose relative position in $\operatorname{Gr}_G$ is contained in $M_{\mu,\mathbb{F}}$. On the other hand, at least when $G = \operatorname{GL}_n$, equality could be obtained by proving the full support of patched modules for e.g. at Hodge types of the form $(\lambda+\rho,\ldots,\rho)$. We hope to do this in future work.
	\subsection*{Additional remarks}
	\begin{itemize}
		\item We don't know if Theorem~\ref{thmA-p1} remains true without the bound on $\operatorname{char}\mathbb{F}$, or whether this bound, if necessary, is at all sharp. It is, however, worth observing that the bound in Theorem~\ref{thmA-p1} is somewhat natural because of its relation to the irreducibility of the $W(\lambda)$ appearing in $\bigotimes_{i=1}^e W(\mu_i -\rho)$. More precisely, recall from \cite[5.6]{Janbook} that, when $\mathbb{F}$ has positive characteristic, $W(\lambda)$ is simple when viewed as an algebraic representation over $\mathbb{F}$ whenever $0 \leq \langle \alpha^\vee,\lambda+\rho \rangle \leq p$.  Since $W(\lambda)$ appears in $\bigotimes_{i=1}^e W(\mu_i -\rho)$ only if $\lambda \leq \sum_{i=1}^e (\mu_i - \rho)$ each such $W(\lambda)$ will be simple if
		$$
		\sum_{i=1}^e \langle \alpha^\vee, \mu_i \rangle \leq \operatorname{char}\mathbb{F} + (e-1)\operatorname{max}_{\alpha^\vee} \langle \alpha^\vee,\rho \rangle
		$$
		If $\operatorname{max}_{\alpha^\vee} \langle \alpha^\vee,\rho \rangle = 1$ (e.g. if $G = \operatorname{GL}_2)$ then this is exactly the bound in Theorem~\ref{thmA-p1} and so one might hypothesise that Theorem~\ref{thmA-p1} remains true at least under this stronger bound. On the other hand, the irreducibility of the $W(\lambda)$ does not appear to play any direct role in our methods, making the significance of these observations questionable.
		\item In contrast, the stronger bound on the $\mu$ in Theorem~\ref{thmA-p2} is far more unnatural. It arises from certain estimates in $p$-adic Hodge theory which could quite possibly be improved, at least so that they agree with the bound in Theorem~\ref{thmA-p1}.
		\item The requirement in Theorem~\ref{thmA-p2} that the $\mu_\kappa$ be strictly dominant arises only from its appearance in Theorem~\ref{thmA-p1}. In particular, if a version of Theorem~\ref{thmA-p1} could be proven for not necessarily strictly dominant cocharacters then the same arguments would allow any such cycle identities to be transferred to an inequality of cycles of crystalline representations with irregular Hodge types.
		\item While we suppress it from the notation we are not able to show that the $M_{\mu,\mathbb{F}}$ do not depend upon the choice of $\mathcal{O}$ and the $\pi_1,\ldots,\pi_e$. Indeed, a more natural construction of the $M_{\mu,\mathbb{F}}$ would involve taking $M_\mu$ as the closure inside the Beilinson--Drinfeld grassmannian over $\mathbb{A}_{\mathbb{Z}}^e$ of an embedding an $e$-fold product of flag varieties over the locus of pairwise distinct tuples in $\mathbb{A}_{\mathbb{Z}}^e$. Then one could define $M_{\mu,\mathbb{F}}$ as the fibre over $0 \in \mathbb{A}^e_{\mathbb{F}}$, which would be independent of any choices. The problem is that, with this definition, we would need to know that $M_{\mu}$ is flat around $0 \in \mathbb{A}_{\mathbb{Z}}^e$ and this is probably a difficult question.
	\end{itemize}
	\subsection*{Connections to previous work}
	We conclude by saying a little about how this paper relates to previous work. Concrete results so far towards Breuil--M\'ezard fall into two broad categories and (with a few exceptions that we mention shortly) all consider the case of $\operatorname{GL}_n$. The first category considers the situation where $G = \operatorname{GL}_2$ and $K =\mathbb{Q}_p$. Here the conjecture is now essentially known, see \cite{KisFM,Pas15,HT15,Sand,Tun21}. These results rely on the existence of a form of the $p$-adic Langlands correspondence, and therefore have little direct relation to our work.
	
	The second category treats the conjecture in either higher dimensions and/or with $K$ a finite extension of $\mathbb{Q}_p$, but at the cost of making (as we do) very strong assumptions on the size of the Hodge types appearing. For example, \cite{GK14} proves the conjecture for any $K /\mathbb{Q}_p$ and two dimensional potentially crystalline representation of $G_K$ with Hodge type $(0,1)$, while \cite{LLBBM} proves the conjecture for $K/\mathbb{Q}_p$ unramified and $n$-dimensional (tamely) potentially crystalline representation of $G_K$ whose Hodge type is bounded by an inexplicit formula in terms of $p$ and the tame type (which, at the very least, requires the Hodge types to be $\leq p$)
	
	While the assumptions of both \cite{GK14} and \cite{LLBBM} are entirely perpendicular to ours (in situations where the assumptions overlap the statement of Breuil--M\'ezard is vacuous) their methods are much closer in sprit to those of our paper. Indeed, both use moduli spaces of Breuil--Kisin modules to control moduli spaces of Galois representations, and describe the former moduli spaces in terms of closed subschemes inside an affine grassmannian. This is particularly true of the closed subschemes appearing in \cite{LLBBM} which are defined as a degeneration of a single flag variety (recall they consider $e=1$) in an affine flag variety (i.e. a twisting of $\operatorname{Gr}_G$ which accounts for the tame type). Clearly, combining this definition with our construction of $M_{\mu,\mathbb{F}}$ describes candidate closed subschemes modelling the geometry of Breuil--Kisin modules associated to potentially crystalline representations of any finite extension of $\mathbb{Q}_p$ (at least for small Hodge types). On the other hand, there are significant points of departure from our methods and those of \cite{GK14} and \cite{LLBBM}. While we use the control of moduli of Breuil--Kisin modules to directly analyse the special fibres of moduli of Galois representations, in loc. cit. they are used as a means to prove modularity lifting theorems, which are in turn used to control the moduli spaces of local Galois representations using patched modules. 
	
	Finally, while the majority of work towards Breuil--M\'ezard has focused on the case of $\operatorname{GL}_n$, there has also been considerations of other groups. In \cite{GG15} and \cite{Dotto18} the conjecture is considered for the group of units in a central division algebra, and in the latter it is shown that these conjectures follow from the conjecture for $\operatorname{GL}_n$. In \cite{DR22} the conjecture for $\operatorname{PGL}_n$ is also shown to follow from the case of $\operatorname{GL}_n$. We also mention \cite{Lev15} and \cite{BB20} which use methods similar to ours to describe some deformation rings on crystalline representations valued in split reductive $G$.
	
	\subsection*{Acknowledgements}
		I would like to thank Yifei Zhao, for many helpful conversations, and the anonymous referee, whose careful reading of the paper caught a number of errors.

	\part{Cycles identities in the affine grassmannian}
	\section{Notation}\label{sec-notation1}
	
	\begin{sub}\label{sub-roots}
		Let $\mathcal{O}$ be a discrete valuation ring with residue field $\mathbb{F}$ and fraction field $E$. Let $G$ be a split connected reductive group over $\operatorname{Spec} \mathcal{O}$ with connected fibres together with a choice of maximal torus $T$. Let $X_*(T) = \operatorname{Hom}(\mathbb{G}_m,T)$ and $X^*(T) = \operatorname{Hom}(T,\mathbb{G}_m)$ and write $\langle ~,~\rangle$ for the natural pairing $X^*(T) \times X_*(T) \rightarrow \mathbb{Z}$. Let $R^\vee \subset X^*(T)$ be the roots $(G,T)$ and for $\alpha^\vee \in R^\vee$ write $\alpha \in X_*(T)$ for the corresponding coroot. Let $W$ be the Weyl group of $(G,T)$. Choose a set of positive roots $R_+^\vee \subset R^\vee$, with associated Borel $B$. Let $R_+$ denote the corresponding set of positive coroots. Recall
		$$
		\lambda^\vee  \leq \mu^\vee \Leftrightarrow \mu^\vee - \lambda^\vee \in \sum_{\alpha \in R_+^\vee} \mathbb{Z}_{\geq 0} \alpha^\vee
		$$
		and that $\lambda^\vee \in X^*(T)$ is dominant if $\langle \lambda^\vee, \alpha\rangle \geq 0$ for all positive coroots $\alpha \in R_+$. We say $\lambda^\vee$ is strictly dominant if $\langle \lambda^\vee, \alpha \rangle \geq 1$. We likewise make sense of $\leq$ on $X_*(T)$ as well as dominant and strictly dominant $\lambda \in X_*(T)$.
	\end{sub}
	
	\begin{definition}\label{def-twistingelement}
		An element $\rho \in X_*(T)$ is called a twisting element if $\langle \alpha^\vee,\rho \rangle =1$ for all simple roots $\alpha^\vee$. Similarly $\rho^\vee \in X^*(T)$ is a twisting element if $\langle \rho^\vee,\alpha \rangle =1$ for all simple coroots $\alpha$.
	\end{definition}
	Notice that if $\rho \in X_*(T)$ is a twisting element then $\rho - \frac{1}{2}\sum_{\alpha \in R_+} \alpha$ is $W$-invariant. Also $\lambda \in X_*(T)$ is strictly dominant if and only if $\lambda -\rho$ is dominant.
	
	\section{Torsors}
	
	\begin{sub}
		Throughout this paper we  view $G$-torsors from the following two equivalent viewpoints:
		\begin{itemize}
			\item A $G$-torsor $\mathcal{E}$ on $\operatorname{Spec}A$ is an $A$-scheme equipped with an action of $G$ so that fppf (equivalently etale) locally on $A$ one has $\mathcal{E} \cong G \times_{\mathcal{O}} \operatorname{Spec}A$.
			\item A fibre functor (a faithful exact tensor functor) from the category of representations of $G$ on finite free $\mathcal{O}$-modules into the category of projective $A$-modules.
		\end{itemize}
		Any $G$-torsor in the first sense induces a fibre functor which sends a representation $\chi: G \rightarrow \operatorname{GL}(V)$ onto the contracted product
		$$
		\mathcal{E}^\chi := \mathcal{E} \times^\chi V = \mathcal{E} \times V / \sim
		$$
		That this construction produces an equivalence of categories is proved in e.g. \cite[4.8]{Broshi}. We always write $\mathcal{E}^0$ for the trivial $G$-torsor and a trivialisation of a $G$-torsor on $\operatorname{Spec}A$ is an isomorphism $\mathcal{E} \cong \mathcal{E}^0$ over $\operatorname{Spec}A$.
	\end{sub}

	\section{Affine grassmannians}\label{sec-affinegr}
	Fix an integer $e \geq 1$ and pairwise distinct $\pi_1,\ldots,\pi_e$ in the maximal ideal of $\mathcal{O}$. For any $\mathcal{O}$-algebra $A$ we write $E(u) = \prod_{i=1}^e (u-\pi_i) \in A[u]$.
	\begin{definition}\label{sub-Grdef}
		Let $\operatorname{Gr}_G$ denote the projective ind-scheme over $\mathcal{O}$ whose $A$-points, for any $\mathcal{O}$-algebra $A$, classify isomorphism classes of pairs $(\mathcal{E},\iota)$ where
		\begin{itemize}
			\item $\mathcal{E}$ is a $G$-torsor over $\operatorname{Spec}A[u]$,
			\item $\iota$ is a trivialisation of $\mathcal{E}$ over the open subscheme $\operatorname{Spec}A[u,E(u)^{-1}]$, i.e. an isomorphism
			$$
			\mathcal{E}|_{\operatorname{Spec}A[u,E(u)^{-1}]} \cong \mathcal{E}^0|_{\operatorname{Spec}A[u,E(u)^{-1}]}
			$$
			where $\mathcal{E}^0$ denotes the trivial $G$-torsor.
		\end{itemize}
		We also consider variants $\operatorname{Gr}_{G,i}$ of $\operatorname{Gr}_G$ for $i =1,\ldots,e$ which are again projective ind-schemes over $\mathcal{O}$, and whose $A$-points classify isomorphism classes of pairs $(\mathcal{E},\iota)$ with $\mathcal{E}$ a $G$-torsor on $\operatorname{Spec}A[u]$ and $\iota$ is a trivialisation of $\mathcal{E}$ over the open subscheme $\operatorname{Spec}A[u,(u-\pi_i)^{-1}]$. Notice that for each $i$ there are natural closed immersions $\operatorname{Gr}_{G,i} \rightarrow \operatorname{Gr}_G$. Each of $\operatorname{Gr}_G$ and $\operatorname{Gr}_{G,i}$ are also functorial in $G$.
	\end{definition}
	\begin{remark}\label{rem-representability}
		When $G = \operatorname{GL}_n$ the above functor is a colimit over $a \geq 0$ of the functors sending an $\mathcal{O}$-algebra $A$ onto the set of rank $n$ projective $A[u]$ submodules 
		$$
		E(u)^{a} A[u]^n \subset \mathcal{E} \subset E(u)^{-a} A[u]^n
		$$
		Since a submodule $\mathcal{E} \subset E(u)^a A[u]^n$ is $A[u]$-projective of rank $n$ if and only if $E(u)^aA[u]^n /\mathcal{E}$ is $A$-projective (see \cite[Lemma 1.1.5]{Zhu17}) each subfunctor is represented by a subfunctor of the grassmannian classifying projective $A$-submodules of $E(u)^{-a}A[u]^n/E(u)^aA[u]^n$, which shows the ind-representability of $\operatorname{Gr}_{\operatorname{GL}_n}$. 
		
		For general $G$ one chooses a faithful representation into $\operatorname{GL}_n$ and, using \cite[1.2.6]{Zhu17}, identifies $\operatorname{Gr}_G$ as a closed sub-indscheme of $\operatorname{Gr}_{\operatorname{GL}_n}$.
	\end{remark}
	\begin{lemma}\label{lem-BL}
		For any $\mathcal{O}$-algebra $A$ set $\widehat{A[u]}_{E(u)}$ equal the $E(u)$-adic completion of $A[u]$. Then the $A$-valued points of $\operatorname{Gr}_G$ functorially identify with isomorphism classes of $G$-torsors on $\operatorname{Spec}\widehat{A[u]}_{E(u)}$ together with a trivialisation after inverting $E(u)$. Similarly for $A$-valued points of $\operatorname{Gr}_{G,i}$ with $E(u)$ replaced by $(u-\pi_i)$.
	\end{lemma}
	\begin{proof}
		This follows from the Beauville--Laszlo gluing lemma \cite{BL}.
	\end{proof}
	\begin{sub}\label{sub-productdecomp}
		If $A$ is an $E$-algebra and $n_i \in \mathbb{Z}_{\geq 0}$ then, since $E[u]$ is principal ideal domain, the product of the quotient maps describes an isomorphism
		$$
		\frac{A[u]}{\prod_{i=1}^e (u-\pi_i)^{n_i}}  \cong \prod_{i=1}^e \frac{A[u]}{(u-\pi_i)^{n_i}}
		$$
		In particular, this gives an isomorphism  $\widehat{A[u]}_{E(u)} \cong \prod_{i=1}^e \widehat{A[u]}_{(u-\pi_i)}$ where the completions are respectively taken against the ideals generated by $E(u)$ and $(u-\pi_i)$. As a consequence, we obtain:
	\end{sub}\begin{corollary}\label{lem-generic}
		There is an isomorphism 
		$$
		\operatorname{Gr}_G \otimes_{\mathcal{O}} E \cong \left(\operatorname{Gr}_{G,1} \times_{\mathcal{O}} \ldots \times_{\mathcal{O}} \operatorname{Gr}_{G,e}\right) \otimes_{\mathcal{O}} E
		$$
		written $(\mathcal{E},\iota) \mapsto (\mathcal{E}_i,\iota_i)_{i=1,\ldots,e}$ on $A$-valued points, so that
		$$
		\mathcal{E} \otimes_{A[u]} \widehat{A[u]}_{E(u)} = \prod_{i=1}^e \mathcal{E}_i \otimes_{A[u]} \widehat{A[u]}_{u-\pi_i}
		$$
		with $\iota = \prod_i \iota_i$.
	\end{corollary}
	
	\begin{sub}\label{sub-opencovers}
		The isomorphism in Corollary~\ref{lem-generic} has an alternative description. For any $\mathcal{O}$-algebra $A$ consider the open subsets
		$$
		U_i = \operatorname{Spec}A[u,\prod_{j \neq i} (u-\pi_j)^{-1}], \qquad  V_i = \operatorname{Spec}A[u,(u-\pi_i)^{-1}] 
		$$
		of $\operatorname{Spec}A[u]$. Then $V_i = \bigcup_{j \neq i} U_j$, $V_i \cap U_i = \operatorname{Spec}A[u,E(u)^{-1}]$ and, if $A$ is an $E$-algebra, then 
		$$
		\operatorname{Spec}A[u] = \bigcup_i U_i = V_i \cup U_i
		$$
		Then $\mathcal{E}_i$ is the $G$-torsor obtained by glueing $\mathcal{E}|_{U_i}$ and $\mathcal{E}^0|_{V_i}$ along the restriction of $\iota$ to $U_i \cap V_i= \operatorname{Spec}A[u,E(u)^{-1}]$
	\end{sub}

	\begin{notation}\label{notation-Elambda}
		For $\lambda \in X_*(T)$ write $\mathcal{E}_{\lambda,i}$ for the $\mathcal{O}$-valued point 
		$$
		(\mathcal{E}^0,(u-\pi_i)^\lambda) \in \operatorname{Gr}_{G,i}
		$$
		where $\mathcal{E}^0$ denotes the trivial $G$-torsor on $\operatorname{Spec}\mathcal{O}[u]$ and $(u-\pi_i)^{\lambda}$ denotes the automorphism of $\mathcal{E}^0|_{\operatorname{Spec}\mathcal{O}[u,E(u)^{-1}]}$ induced by multiplication by $\lambda(u-\pi_i) \in G(\mathcal{O}[u,E(u)^{-1}])$.  We then define the locally closed subscheme 
		$$
		\operatorname{Gr}_{G,i,\lambda} \subset \operatorname{Gr}_{G,i}
		$$
		as the orbit of $\mathcal{E}_{\lambda,i}$ under the action of the group scheme $L^+G$ with $A$-valued points $G(A[u])$.
	\end{notation}
	\begin{lemma}\label{lem-closedflag}
		For each $\lambda \in X_*(T)$ and $i=1,\ldots,e$ the morphism $G \rightarrow \operatorname{Gr}_{G,i}$ given by $g \mapsto g \mathcal{E}_{\lambda,i}$ induces a closed immersion
		$$
		G/P_\lambda \rightarrow \operatorname{Gr}_{G,i}
		$$
		where 
		$$
		P_\lambda = \lbrace g \in G \mid \operatorname{lim}_{t \rightarrow 0} \lambda(t)^{-1} g \lambda(t) \text{ exists} \rbrace
		$$
		Equivalently, $P_\lambda$ is the parabolic subgroup of $G$ generated by $T$ and the roots subgroups $U_{\alpha^\vee}$ for $\alpha^\vee \in R^\vee$ with $\langle \alpha^\vee ,\lambda \rangle \leq 0$.
	\end{lemma}
	\begin{proof}
		The $A$-points of the stabiliser in $G$ of $\mathcal{E}_\lambda$ consists of those $g \in G(A)$ for which  $\lambda(u-\pi_i)^{-1} g \lambda(u-\pi_i) \in G(A[u])$. Therefore, this stabiliser is precisely $P_\lambda$ and we obtain a monomorphism $G/P_\lambda \rightarrow \operatorname{Gr}_{G,i}$. Since $P_\lambda$ is a parabolic subgroup of $G$ the quotient $G/P_\lambda$ is proper over $\mathcal{O}$. Thus $G/P_\lambda \rightarrow \operatorname{Gr}_{G,i}$ is also proper. Proper monomorphisms are closed immersions \cite[04XV]{stacks-project} so the lemma follows.
	\end{proof}
	\begin{definition}\label{def-mmu}
		For $\mu = (\mu_1,\ldots,\mu_e)$ with $\mu_i \in X_*(T)$ set $M_\mu \subset \operatorname{Gr}_G$ equal to the scheme theoretic image of the composite
		$$
		\left( G/P_{\mu_1} \times_{\mathcal{O}} \ldots \times_{\mathcal{O}} G/P_{\mu_e} \right) \otimes_{\mathcal{O}} E \rightarrow \left( \operatorname{Gr}_{G,1} \times_{\mathcal{O}} \ldots \times_{\mathcal{O}} \operatorname{Gr}_{G,e} \right) \otimes_{\mathcal{O}} E \cong \operatorname{Gr}_G \otimes_{\mathcal{O}} E \rightarrow \operatorname{Gr}_{G}
		$$
		(the isomorphism coming from Lemma~\ref{lem-generic}).
	\end{definition}
	\begin{sub}\label{sub-moduliinterpret}
		The $A$-valued points of $G/P_\lambda$ classify filtrations of type $\lambda$ on the trivial $G$-torsor over $\operatorname{Spec}A$, i.e exact tensor functors from the category of representations of $G$ on finite free $\mathcal{O}$-modules $V$ into the category of filtrations on $V$ by $A$-modules with projective graded pieces. From this point of view:
		\begin{itemize}
			\item The closed immersion $G/P_\lambda \rightarrow \operatorname{Gr}_{G,i}$ from Lemma~\ref{lem-closedflag} sends a filtration $\operatorname{Fil}^\bullet$ onto the $G$-torsor $\operatorname{Fil}^0( G \otimes_A \widehat{A[u]}_{(u-\pi)})$ over $\widehat{A[u]}_{(u-\pi)}$ which, when interpreted as an exact tensor functor on the category of representations of $G$, is given by
			$$
			V \mapsto \sum_{n \in \mathbb{Z}} \operatorname{Fil}^n(V) \otimes_A (u-\pi_i)^{-n} \widehat{A[u]}_{(u-\pi_i)}
			$$
			This $G$-torsor is interpreted as a point of $\operatorname{Gr}_{G,i}$ via the natural trivialisation coming from the identification of $\operatorname{Fil}^0( G \otimes_A \widehat{A[u]}_{(u-\pi_i)})$ after inverting $(u-\pi_i)$ with the base change of the trivial $G$-torsor underlying $\operatorname{Fil}^\bullet$ to $\widehat{A[u]}_{(u-\pi_i)}[\frac{1}{(u-\pi_i)}]$. To verify this assertion it suffices, by functoriality of $G/P_\lambda \rightarrow \operatorname{Gr}_{G,i}$ in $G$, to consider the case $G = \operatorname{GL}_n$.
			\item Similarly, the closed immersion	$\left( G/P_{\mu_1} \times_{\mathcal{O}} \ldots \times_{\mathcal{O}} G/P_{\mu_e} \right) \rightarrow \operatorname{Gr}_G$ from Definition~\ref{def-mmu} sends an $A$-valued point corresponding to an $e$-tuple of filtrations $(\operatorname{Fil}^\bullet_i)_i$ onto the $G$-torsor $\operatorname{Fil}^0( G \otimes_A \widehat{A[u]}_{E(u)})$ which, from the Tannakian viewpoint, we view as the functor on representations of $G$ given by
			$$
			V \mapsto \left( \sum_{n \in \mathbb{Z}} \operatorname{Fil}^n(V) \otimes_A (u-\pi_i)^{-n} \widehat{A[u]}_{(u-\pi_i)} \right)_{i=1,\ldots,e }
			$$
			The target is viewed as a projective $\widehat{A[u]}_{E(u)}$-module via the isomorphism from \ref{sub-productdecomp}.
		\end{itemize}
	\end{sub}

	\section{Various Schubert varieties}\label{sec-schubert}
	
	Here we introduce some variants on the usual notation of Schubert varieties inside the affine grassmannian.
	\begin{notation}\label{sub-orbit}
		For $\lambda \in X_*(T)$ write $\mathcal{E}_{\lambda,\mathbb{F}} \in \operatorname{Gr}_G(\mathbb{F})$ for the point corresponding to $(\mathcal{E}^0,u^\lambda)$. Thus 
		$$
		\mathcal{E}_{\lambda,\mathbb{F}} = \mathcal{E}_{\lambda,i} \otimes_{\mathcal{O}} \mathbb{F}
		$$
		with $\mathcal{E}_{\lambda,i}$ as defined in Notation~\ref{notation-Elambda} and for any $i=1,\ldots,e$. We also write $\operatorname{Gr}_{G,\lambda,\mathbb{F}} = \operatorname{Gr}_{G,\lambda,i} \otimes_{\mathcal{O}} \mathbb{F}$ for any $i=1,\ldots,e$. Equivalently, $\operatorname{Gr}_{G,\lambda,\mathbb{F}}$ is the $L^+G$-orbit of $\mathcal{E}_{\lambda,\mathbb{F}}$.
	\end{notation}
	
	For $\lambda \in X_*(T)$ the Schubert variety $\operatorname{Gr}_{G,\leq \lambda,\mathbb{F}} \subset \operatorname{Gr}_G \otimes_{\mathcal{O}} \mathbb{F}$ is usually defined as the closure of $\operatorname{Gr}_{G,\lambda,i} \otimes_{\mathcal{O}} \mathbb{F}$ (for any $i=1,\ldots,e)$. Then $\operatorname{Gr}_{G,\leq \lambda,\mathbb{F}}$ is reduced, irreducible, and can be expressed as
	$$
	\operatorname{Gr}_{G,\leq \lambda,\mathbb{F}} = \bigcup_{\lambda' \leq \lambda} \operatorname{Gr}_{G,\lambda',\mathbb{F}}
	$$
	We would like to use $\operatorname{Gr}_{G,\leq \mu_1+\ldots+\mu_e,\mathbb{F}}$ as an ambient space in which to study $M_\mu \otimes_{\mathcal{O}} \mathbb{F}$. However, the containment of $M_\mu \otimes_{\mathcal{O}} \mathbb{F}$ in $\operatorname{Gr}_{G,\leq \mu_1+\ldots+\mu_e,\mathbb{F}}$ is not immediate due to both varieties construction as a closure. This issue can be addressed using the identifications from \cite{PZ13,Levin} of $\operatorname{Gr}_{G,\leq \lambda,\mathbb{F}}$ with the special fibre of mixed characteristic Schubert varieties. However, we instead use give a more simple minded approach and replace $\operatorname{Gr}_{G,\leq \lambda,\mathbb{F}}$ with a  moduli construction which is close to (and conjectured to equal) $\operatorname{Gr}_{G,\leq \lambda,\mathbb{F}}$.
	
	\begin{definition}\label{def-schubert}
		Let $V$ be a free $\mathcal{O}$-module. Any $A$-valued point of $\operatorname{Gr}_{\operatorname{GL}(V)}$ corresponds to a projective $A[u]$-submodule of $V \otimes_{\mathcal{O}} A[u,E(u)^{-1}]$ and so, for any $e$-tuple of integers $(n_i)$, we may consider the closed subfunctor $Y_{\operatorname{GL}(V)}^{\geq (n_i)} \subset \operatorname{Gr}_{\operatorname{GL}(V)}$  consisting of those $\mathcal{E}$ for which 
		\begin{equation}\label{eq-schubert}
			\mathcal{E} \subset \prod_{i=1}^e (u-\pi_i)^{n_i} V \otimes_{\mathcal{O}} A[u]	
		\end{equation}
		If $\mu = (\mu_1,\ldots,\mu_e)$ with $\mu_i \in X_*(T)$ dominant then define
		$$
		Y_{G,\leq \mu} = \bigcap_\chi \left( \operatorname{Gr}_G \times_{\chi,\operatorname{Gr}_{\operatorname{GL}(V)}} Y_{\operatorname{GL}(V)}^{\geq (\langle w_0\chi^\vee,\mu_i\rangle)} \right)
		$$
		where $w_0 \in W$ is the longest element and the intersection runs over irreducible algebraic representations $\chi: G \rightarrow \operatorname{GL}(V)$ of highest weight $\chi^\vee$. Notice that each $Y_{\operatorname{GL}(V)}^{\geq (n_i)}$ is stable under the action of $L^+\operatorname{GL}(V)$ on $\operatorname{Gr}_{\operatorname{GL}(V)}$ and so $Y_{G,\leq \mu}$ is stable under the $L^+G$-action on $\operatorname{Gr}_G$.
	\end{definition}
	
	\begin{example}
		Suppose that $G = \operatorname{GL}_n$ so that $A$-valued points of $\operatorname{Gr}_G $ identify with projective $A[u]$-submodules $\mathcal{E} \subset A[u,E(u)^{-1}]^n$. If 
		$$
		\mu = (\mu_1,\ldots,\mu_e), \qquad  \mu_i = (\mu_{i,1} \geq \ldots \geq \mu_{i,n})
		$$
		and $\mathcal{E} \in Y_{G,\leq \mu} \otimes_{\mathcal{O}} \mathbb{F}$ then
		\begin{equation}\label{eq-wedge}
			\bigwedge^j(\mathcal{E}) \subset \prod_{i=1}^e (u-\pi_i)^{ \mu_{i,n} + \ldots + \mu_{i,n-j+1}} A[u]^{\binom{n}{j}}
		\end{equation}
		for all $j = 1,\ldots,n$. Indeed, if $\chi:G \rightarrow \operatorname{GL}(V)$ equals the $j$-th exterior power of the standard representation, then the induced morphism $\operatorname{Gr}_G \rightarrow \operatorname{Gr}_{\operatorname{GL}(V)}$ sends $\mathcal{E}$ onto $\bigwedge^j(\mathcal{E})$ and so
		$$
		\operatorname{Gr}_G \times_{\chi,\operatorname{Gr}_{\operatorname{GL}(V)}} Y_{\operatorname{GL}(V)}^{\geq (\langle w_0\chi^\vee,\mu_i\rangle)}
		$$
		is the closed subscheme consisting of $\mathcal{E}$ as in \eqref{eq-wedge}. This is because
		$$
		\chi^\vee = (\underbrace{1,\ldots,1}_{j\text{ ones}},0,\ldots,0)
		$$
		and so $\langle w_0 \chi^\vee, \mu_i \rangle = \mu_{i,n}+\ldots+\mu_{i,n-j+1}$. In fact, since every highest weight representation of $\operatorname{GL}_n$ is a quotient of tensor products of these exterior powers representations, conditions \eqref{eq-wedge} suffice to determine $Y_{G,\leq \mu}$.
		
	\end{example}
	The following lemma describes the basic properties of $Y_{G,\leq \mu}$ that we need.
	
	\begin{lemma}\label{lem-Yschuprop}
		Let $\mu_1,\ldots,\mu_e \in X_*(T)$ be dominant. Then $Y_{G,\leq \mu} \otimes_{\mathcal{O}} \mathbb{F}$ only depends upon $\mu_1+\ldots+\mu_e$, contains $\operatorname{Gr}_{G,\mu_1+\ldots+\mu_e,\mathbb{F}}$ as an open subset, and 
		$$
		(Y_{G,\leq \mu} \otimes_{\mathcal{O}} \mathbb{F})_{\operatorname{red}} = \bigcup_{\lambda \leq \mu_1+\ldots+\mu_e} \operatorname{Gr}_{G,\lambda,\mathbb{F}}
		$$
	\end{lemma}
	\begin{proof}
		That $Y_{G,\leq \mu} \otimes_{\mathcal{O}} \mathbb{F}$ only depends upon $\mu_1+\ldots+\mu_e$ is clear from the definition. It is also clear from the definition that $Y_{G,\leq \mu}$ contains $\mathcal{E}_{\lambda,\mathbb{F}} \in \operatorname{Gr}_{G}(\mathbb{F})$ if and only if $\lambda \leq \mu_1+\ldots+\mu_e$. Since $Y_{G,\leq \mu}$ is $L^+G$-stable it follows that 
		$$
		(Y_{G,\leq \mu} \otimes_{\mathcal{O}} \mathbb{F})_{\operatorname{red}} = \bigcup_{\lambda \leq \mu_1+\ldots+\mu_e} \operatorname{Gr}_{G,\lambda,\mathbb{F}}
		$$ 
		It remains to show $\operatorname{Gr}_{G, \mu_1+\ldots,\mu_e,\mathbb{F}}$ is open in $Y_{G,\leq \mu} \otimes_{\mathcal{O}} \mathbb{F}$. This will follow if $Y_{G,\leq \mu} \otimes_{\mathcal{O}}\mathbb{F}$ is reduced at $\mathcal{E}_{\mu_1+\ldots+\mu_e,\mathbb{F}}$, which can be achieved by a simple tangent space computation identifying the tangent space of $Y_{G,\leq \mu}\otimes_{\mathcal{O}} \mathbb{F}$ at $\mathcal{E}_{\mu_1+\ldots+\mu_e,\mathbb{F}}$ with the tangent space of $\operatorname{Gr}_{G,\mu_1+\ldots+\mu_e,\mathbb{F}}$. See \cite[\S3]{KMW18}.
	\end{proof}
	\begin{remark}
		The construction of $Y_{G,\leq \mu} \otimes_{\mathcal{O}} \mathbb{F}$ was proposed in \cite{FM99} (see also \cite{H12}) as a moduli interpretation of $\operatorname{Gr}_{G,\leq \mu_1+\ldots+\mu_e,\mathbb{F}}$. However, it is an open question whether $Y_{G,\leq \mu} \otimes_{\mathcal{O}} \mathbb{F} = \operatorname{Gr}_{G,\leq \mu_1+\ldots+\mu_e,\mathbb{F}}$ (equivalently, whether $Y_{G,\leq \mu} \otimes_{\mathcal{O}} \mathbb{F}$ is reduced). The equality is known when $G =\operatorname{SL}_n$ and $\mathbb{F}$ has characteristic zero, see \cite{KMW18}.
	\end{remark}
	
	\begin{lemma}\label{lem-schubertgeneric}
		Under the isomorphism $\operatorname{Gr}_G \otimes_{\mathcal{O}} E \cong \left( \operatorname{Gr}_{G,1} \times_{\mathcal{O}} \ldots \times_{\mathcal{O}} \operatorname{Gr}_{G,e} \right) \otimes_{\mathcal{O}} E$ from Lemma~\ref{lem-generic} we have
		$$
		Y_{G,\leq \mu} \otimes_{\mathcal{O}} E \cong \left( Y_{G,1,\leq \mu_1} \times_{\mathcal{O}} \ldots \times_{\mathcal{O}} Y_{G,e,\leq \mu_e} \right) \otimes_{\mathcal{O}} E
		$$
		where $Y_{G,i,\leq \mu_i} \subset \operatorname{Gr}_{G,i}$ is defined as a special case of Definition~\ref{def-schubert}.
	\end{lemma}
	\begin{proof}
		The lemma reduces to showing that, for any tuple $(n_i)$ of integers and any finite free $\mathcal{O}$-module $V$, 
		$$
		Y_{\operatorname{GL}(V)}^{\geq (n_i)} \otimes_{\mathcal{O}} E = \left( Y_{\operatorname{GL}(V),1}^{\geq n_1} \times_{\mathcal{O}} \ldots \times_{\mathcal{O}} Y_{\operatorname{GL}(V),e}^{\geq n_e} \right) \otimes_{\mathcal{O}} E
		$$ 
		under the identification from Lemma~\ref{lem-generic}. This is clear since if $A$ is an $E$-algebra then the isomorphism $\widehat{A[u]}_{E(u)} \cong \prod_{i=1}^e \widehat{A[u]}_{u-\pi_i}$ identifies the ideal generated by $\prod (u-\pi_i)^{n_i}$ with the product of the ideals generated by $(u-\pi_i)^{n_i}$.
	\end{proof}
	The reason we introduce Definition~\ref{def-schubert} is because it easily allows us to prove:
	
	\begin{proposition}\label{prop-schubertcontain}
		For $\mu = (\mu_1,\ldots,\mu_e)$ with $\mu_i \in X_*(T)$ dominant we have $M_\mu \subset Y_{G,\leq \mu}$.
	\end{proposition}
	\begin{proof}
		It suffices to show $M_\mu \otimes_{\mathcal{O}} E \subset Y_{G,\leq \mu}$ because $Y_{G,\leq \mu}$ is closed in $\operatorname{Gr}_G$. By Lemma~\ref{lem-schubertgeneric} we are reduced to showing that $G/P_{\mu_i} \rightarrow \operatorname{Gr}_{G,i}$ factors through $Y_{G,i,\leq \mu_i}$. Since $Y_{G,i,\leq \mu_i}$ is $G$-stable it is enough to show $\mathcal{E}_{\mu_i,i} \in Y_{G,i,\leq \mu_i}$, and this is clear.
	\end{proof}

	\section{Approximations via $\operatorname{Gr}_G^\nabla$}\label{sec-dlog}
	When $G = \operatorname{GL}_n$ the spaces $M_\mu \otimes_{\mathcal{O}} \mathbb{F}$ were accessed in \cite[\S7]{B21} by constructing a subfunctor $\operatorname{Gr}^\nabla_G \subset \operatorname{Gr}_G$ with $M_\mu  \subset \operatorname{Gr}_G^\nabla$. The subfunctor $\operatorname{Gr}_G^\nabla$ consists of $\mathcal{E} \in \operatorname{Gr}_G(A)$ which, when viewed as projective $A[u]$-submodules of $A[u,E(u)^{-1}]^n$, satisfy
	$$
	E(u)\nabla(\mathcal{E}) \subset \mathcal{E}
	$$
	for $\nabla$ the operator on $A[u,E(u)^{-1}]^n$ given coordinate-wise via $\frac{d}{du}$. If $\mathcal{E}$ is generated by $(e_1,\ldots,e_n)X$ for a matrix $X \in \operatorname{GL}_n(A[u,E(u)^{-1}])$ then this is equivalent to asking that
	$$
	E(u)X^{-1}\frac{d}{du}(X) \in \operatorname{Mat}(A[u])
	$$
	In this section we show how to extend this construction to general $G$.
	\begin{sub}
		The following construction works when $G = \operatorname{Spec}\mathcal{O}_G$ is any affine algebraic group over $\mathcal{O}$. Set $\mathfrak{g} = \operatorname{Lie}(G)$. In what follows we will interpret elements of $\mathfrak{g}$ as derivations $\mathcal{O}_G \rightarrow \mathcal{O}$ over $\mathcal{O}$ where $\mathcal{O}_G$ acts on $\mathcal{O}$ via the counit map $e:\mathcal{O}_G \rightarrow \mathcal{O}$. The logarithmic derivative can then be described as a map
		$$
		\operatorname{dlog}: G(B) \rightarrow \mathfrak{g} \otimes_{\mathbb{Z}} \Omega_{B/\mathcal{O}}
		$$
		for any ring $B$. To define this map identify $\Omega_{G/\mathcal{O}} =  \mathcal{L}(G)^\vee \otimes_{\mathcal{O}} \mathcal{O}_G$ where $\mathcal{L}(G)$ denotes the translation invariant derivations $\mathcal{O}_G \rightarrow \mathcal{O}_G$. Then
		$$
		\mathfrak{g} \otimes_{\mathcal{O}} \Omega_{G/\mathcal{O}} = \operatorname{Hom}_{\mathcal{O}}(\mathcal{L}(G),\mathfrak{g}) \otimes_{\mathcal{O}} \mathcal{O}_G
		$$
		and the map  $\mathcal{L}(G) \rightarrow \mathfrak{g}$ given by composition with the counit $e$ defines a canonical global section. For any $g \in G(B)$ define $\operatorname{dlog}(g)$ as the image of this section under $\mathfrak{g} \otimes_{\mathcal{O}} g^*\Omega_{G/\mathcal{O}} \rightarrow \mathfrak{g} \otimes_{\mathcal{O}} \Omega_{B/\mathcal{O}}$. If $A$ is any $\mathcal{O}$-algebra and $B = A[u,E(u)^{-1}]$ then we obtain an element
		$$
		\operatorname{dlog}_u(g) \in \mathfrak{g} \otimes_{\mathcal{O}} A[u,E(u)^{-1}]
		$$
		by evaluating $\operatorname{dlog}(g)$ at the derivation $\frac{d}{du}: A[u,E(u)^{-1}] \rightarrow A[u,E(u)^{-1}]$. This construction is functorial in $G$. 
	\end{sub}
	The following example motivates us calling this construction the logarithmic derivative.
	\begin{example}\label{ex-gln}
		Let $G = \operatorname{GL}_n$ with coordinates $T_{ij}$ and write $\iota: \mathcal{O}_G \rightarrow \mathcal{O}_G$ for the coinverse map. Write $\frac{d}{d T_{ij}}$ for the element of $\mathfrak{g}$ sending $T_{ij} \mapsto 1$ and zero on all other coordinates. We claim that the section
		\begin{equation}\label{eq-section}
			\sum_{ij} \frac{d}{dT_{ij}} \otimes \left( \sum_m \iota(T_{im}) d(T_{mj}) \right) \in \mathfrak{g} \otimes_{\mathcal{O}} \Omega_{G/\mathcal{O}}	
		\end{equation}
		coincides with the map $\mathcal{L}(G) \rightarrow \mathfrak{g}$ given by composition with the counit. If $\Delta: T_{ij} \mapsto \sum_l T_{il} \otimes T_{lj}$ is the comultiplication map then $\mathcal{L}(G)\rightarrow \mathfrak{g}$ has an inverse given by $d \mapsto (\operatorname{id} \otimes d) \circ \Delta$ (see for example \cite[12.24]{MilneAlggroups}). Therefore we can check the claim by evaluating $\sum_m \iota(T_{im}) d(T_{mj})$ at $(\operatorname{id} \otimes \frac{d}{d T_{lk}}) \circ \Delta$. Since 
		$$
		(\operatorname{id} \otimes \frac{d}{d T_{lk}}) \circ \Delta: T_{mj} \mapsto \begin{cases}
			T_{ml} & \text{ if $j=k$} \\
			0 & \text{ otherwise}
		\end{cases}
		$$
		this evaluation is equal to
		$$
		\begin{cases}
			\sum_m \iota(T_{im}) T_{ml} & \text{ if $j=k$} \\
			0 & \text{ if $j \neq k$}
		\end{cases}$$
		Since composing $(\iota \otimes \operatorname{id}) \circ \Delta$ with multiplication $\mathcal{O}_G \otimes \mathcal{O}_G \rightarrow \mathcal{O}_G$ equals the counit $e$ it follows that the above evaluation is $1$ if $ij = lk$ and zero otherwise. This verifies our claim. If $g = (g_{ij}) \in G(B)$ has inverse $g^{-1} = (h_{ij})$ then the image of \eqref{eq-section} in $\mathfrak{g} \otimes \Omega_{B/\mathcal{O}}$ is
		$$
		\operatorname{dlog}(g) = \sum_{ij} \frac{d}{d T_{ij}} \otimes \left( \sum_m h_{im} d(g_{mj}) \right)
		$$
		If $B = A[u,E(u)^{-1}]$ and we identify $\mathfrak{g} = \operatorname{Mat}_{n\times n}(\mathcal{O})$ via $\frac{d}{d T_{ij}}$ then evaluating $\operatorname{dlog}(g)$ at $\frac{d}{du}$ yields $\operatorname{dlog}_u(g) = g^{-1}\frac{d}{du}(g)$.
	\end{example}
	\begin{remark}\label{rem-Tannakiandlogu}
		If $G$ is a flat and finite type over $\mathcal{O}$ then $\operatorname{dlog}_u$ can alternatively be constructed using the Tannakian viewpoint. As explained in e.g. \cite{Lev15}, an element of $\mathfrak{g} \otimes_{\mathcal{O}} A[u]$ is equivalent to a collection of endomorphisms $X_V$ for all representations $G \rightarrow \operatorname{GL}(V)$ of $G$ on finite free $\mathcal{O}$-modules which are compatible with exact sequences and satisfy $X_{V_1\otimes V_2} = X_{V_1} \otimes 1 + 1 \otimes X_{V_2}$. Example~\ref{ex-gln} shows that $\operatorname{dlog}_u(g)$ corresponds to the rule sending a representation $\rho$ onto
		$$
		\rho(g)^{-1} \frac{d}{du}(\rho(g)) \in \operatorname{End}(V)
		$$
		(the compatibility of this rule with exact sequences and tensor products being an easy computation).
	\end{remark}
	\begin{example}\label{ex-Ga}
		If $G  = \mathbb{G}_a$ with coordinate $T$ then write $\frac{d}{dT}$ for the element of $\mathfrak{g}$ sending $T \mapsto 1$. In this case the section
		$$
		\frac{d}{dT} \otimes d(T) \in \mathfrak{g} \otimes_{\mathcal{O}} \Omega_{G/\mathcal{O}}
		$$
		coincides with the map $\mathcal{L}(G) \rightarrow \mathfrak{g}$ given by composition with the counit and so if $g \in \mathbb{G}_a(B)$ then
		$$
		\operatorname{dlog}(g) = \frac{d}{dT} \otimes d(b)
		$$
		If $B = A[u,E(u)^{-1}]$ and we identify $\mathfrak{g} = \mathcal{O}$ via $\frac{d}{dT}$ it follows that  $\operatorname{dlog}_u(g) = \frac{d}{du}(g)$.
	\end{example}

	The next lemma contains what we will need to compute with $\operatorname{dlog}_u(-)$.
	\begin{lemma}\label{lem-dlog}
		\begin{enumerate}
			\item For $g,h \in G(A[u,E(u)^{-1}])$ we have 
			$$
			\operatorname{dlog}_u(gh) = \operatorname{Ad}(h^{-1}) \operatorname{dlog}_u(g) + \operatorname{dlog}_u(h)
			$$
			where $\operatorname{Ad}$ denotes the adjoint action of $G$ on $\mathfrak{g}$.
			\item $\operatorname{dlog}_u(g) = 0$ for $g \in G(A)$
			\item $u \operatorname{dlog}_u(u^\lambda) \in \operatorname{Lie}(T)$ for $\lambda \in X_*(T)$ and $\alpha^\vee(u \operatorname{dlog}_u(u^\lambda)) = \langle \alpha^\vee,\lambda \rangle$ for any root $\alpha^\vee$.
		\end{enumerate}
	\end{lemma}
	\begin{proof}
		Both (1) and (2) can be proven by choosing a faithful representation of $G$ to reduce to the case of $\operatorname{GL}_n$, where they are clear given the interpretation of $\operatorname{dlog}_u$ from Example~\ref{ex-gln}. (3) follows from the functorialty of $\operatorname{dlog}_u$ under the morphism $T \rightarrow G$.
	\end{proof}

	\begin{definition}
		Define $\operatorname{Gr}^{\nabla}_G \subset \operatorname{Gr}_G$ to be the subfunctor consisting of those $(\mathcal{E},\iota) \in \operatorname{Gr}_G(A)$ for which there exist an fppf cover $A\rightarrow A'$ trivialising $\mathcal{E}$ so that, if $\iota \times_{A[u]} A'[u]$ is given by left multiplication by $g \in G(A'[u,E(u)^{-1}])$, then
		$$
		E(u) \operatorname{dlog}_u(g) \in \mathfrak{g} \otimes_{\mathcal{O}} A[u]
		$$
		This is well defined because of other choice of trivialisation replaces $g$ by $gh$ for $h \in L^+G$ and part (1) of Lemma~\ref{lem-dlog} shows $E(u) \operatorname{dlog}_u(g) \in \mathfrak{g} \otimes_{\mathcal{O}} A[u]$ if and only if $E(u) \operatorname{dlog}_u(gh) \in \mathfrak{g} \otimes_{\mathcal{O}} A[u]$.
		This is a closed subfunctor since $A[u,E(u)^{-1}]/A[u]$ is a projective $A$-module. When $G = \operatorname{GL}_n$ this coincides with the subfunctor defined in \cite[7.4]{B21}. Lemma~\ref{lem-dlog} shows that $\operatorname{Gr}^{\nabla}_G$ is also $G$-stable.
	\end{definition}
	
	\begin{proposition}\label{prop-include}
		$M_\mu \subset \operatorname{Gr}^{\nabla}_G$ for any $\mu = (\mu_1,\ldots,\mu_e)$.
	\end{proposition}
	\begin{proof}
		Since $\operatorname{Gr}^{\nabla}_G$ is a closed subfunctor it suffices to show that $M_\mu \otimes_{\mathcal{O}} E \subset \operatorname{Gr}_G^\nabla$. For this observe that, under the identification $\operatorname{Gr}_G \otimes_{\mathcal{O}} E \cong \left( \operatorname{Gr}_{G,1} \times_{\mathcal{O}} \ldots \times_{\mathcal{O}} \operatorname{Gr}_{G,e} \right)  \otimes_{\mathcal{O}} E$, one has
		$$
		\operatorname{Gr}_G^\nabla \otimes_{\mathcal{O}} E = \left( \operatorname{Gr}_{G,1}^\nabla \times_{\mathcal{O}} \ldots \times_{\mathcal{O}} \operatorname{Gr}_{G,e}^\nabla \right) \otimes_{\mathcal{O}} E
		$$
		where $\operatorname{Gr}_{G,i}^\nabla $ is defined analogously to $\operatorname{Gr}_{G}^\nabla$ with the condition  $E(u) \operatorname{dlog}_u(g) \in \mathfrak{g} \otimes_{\mathcal{O}} A[u]$ replaced by $(u-\pi_i) \operatorname{dlog}_u(g) \in \mathfrak{g} \otimes_{\mathcal{O}} A[u]$. We are therefore reduced to showing that the closed immersions $G/P_\lambda \rightarrow \operatorname{Gr}_{G,i}$ induced by any $\lambda \in X_*(T)$ factor through $\operatorname{Gr}_{G,i}^\nabla$, and this follows from Lemma~\ref{lem-dlog}.
	\end{proof}
	
	\section{Computations in $\operatorname{Gr}_G^\nabla$}\label{sec-comutationsinnabla}

	In this section we show that the inclusion $M_\mu \subset \operatorname{Gr}_G^\nabla$ induces a reasonable topological description of $M_\mu \otimes \mathbb{F}$ provided $\mu$ is sufficiently small relative to the characteristic of $\mathbb{F}$. As with the previous section, this extend results from \cite[\S7]{B21} beyond $G = \operatorname{GL}_n$.
	
	\begin{proposition}\label{prop-dim}
		Suppose that $\lambda \in X_*(T)$ is dominant. If $\operatorname{char}\mathbb{F} >0$ then assume that
		$$
		\langle \alpha^\vee, \lambda \rangle \leq \operatorname{char}\mathbb{F} + e -1
		$$
		for every positive root $\alpha^\vee$. Then
		$$
		\operatorname{Gr}_{G,\lambda,\mathbb{F}} \times_{\operatorname{Gr}_G} \operatorname{Gr}_{G}^\nabla
		$$
		is smooth and irreducible of dimension
		$$
		\sum_{\alpha \in R^+}   \operatorname{min}\lbrace e, \langle \alpha,\lambda \rangle \rbrace
		$$
	\end{proposition}
	
	\begin{proof}
		As in the previous section we write $\mathfrak{g} = \operatorname{Lie}(G)$. Let $\mathfrak{t} = \operatorname{Lie}(T)$ and write 
		$$
		\mathfrak{g} = \mathfrak{t} \oplus \bigoplus_{\alpha^\vee\in R^\vee} \mathfrak{g}_{\alpha^\vee}
		$$
		for the root decomposition of $\mathfrak{g}$. For each $\alpha^\vee \in R^\vee$ we have the associated root homomorphism $x_{\alpha^\vee}: \mathbb{G}_a \rightarrow G$ which induces an identification $dx_{\alpha^\vee}: \mathbb{G}_a \xrightarrow{\sim} \mathfrak{g}_{\alpha^\vee}$. See for example \cite[1.2]{Janbook}.
		
		\subsubsection*{Step 1}
		We begin by recalling a standard open cover of $G/P_{\lambda}$. Let $U$ denote the image of the morphism 
		$$
		\prod_{\langle \alpha^\vee, \lambda \rangle >0} \mathbb{A}^1 \rightarrow G
		$$
		given by $(a_{\alpha^\vee})_{\alpha^\vee} \mapsto \prod_{\alpha^\vee} x_{\alpha^\vee}(a_{\alpha^\vee})$ (the product taken in an arbitrary, but fixed, order). Then $U$ is a closed subgroup of $G$ and the induced morphism $U \rightarrow G/P_\lambda$ is an open immersion. Furthermore, the $W$-translates of the image of $U$ form an open cover of $G/P_\lambda$. See \cite[1.10]{Janbook} for more details.
		
		\subsubsection*{Step 2}
		Let $U_{\lambda}$ denote the image of the morphism
		$$
		\prod_{\langle \alpha^\vee,\lambda \rangle >0} \mathbb{A}^{\langle \alpha^\vee,\lambda \rangle} \rightarrow L^+G
		$$
		given by $(a_{\alpha^\vee,0},a_{\alpha,1},\ldots,a_{\alpha^\vee, \langle \alpha^\vee,\lambda \rangle -1})_{\alpha^\vee} \mapsto \prod_{\alpha^\vee} x_{\alpha^\vee}( \sum_i a_{\alpha^\vee,i} u^i)$ (again the product is taken in an arbitrary, but fixed, order). Then the morphism $U_{\lambda} \rightarrow \operatorname{Gr}_{G,\lambda,\mathbb{F}}$ given by $g \mapsto gu^\lambda$ is an open immersion whose $W$-translates cover $\operatorname{Gr}_{G,\lambda,\mathbb{F}}$. To see this note first that $U_\lambda \rightarrow \operatorname{Gr}_{G,\lambda,\mathbb{F}}$ is a monomorphism. Secondly, note that $U_{\lambda} \rightarrow \operatorname{Gr}_{G,\lambda,\mathbb{F}}$ factors through the preimage of $U$ under the morphism $q:\operatorname{Gr}_{G,\lambda,\mathbb{F}} \rightarrow G/P_\lambda$ sending $gu^\lambda$ onto $g$ modulo $u$. Thirdly, note that $q^{-1}(U)$ is smooth and irreducible of dimension $\sum_{\alpha \in R^+ } \langle \alpha^\vee,\lambda \rangle = \langle 2\rho^\vee,\lambda \rangle$ (because the same is known to be true of $\operatorname{Gr}_{G,\lambda,\mathbb{F}}$, see for example \cite[2.1.5]{Zhu17}). Therefore $U_\lambda \rightarrow q^{-1}(U)$, being a monomorphism between integral schemes of the same dimension, is an isomorphism (because monomorphisms are unramified \cite[02GE]{stacks-project} and unramified morphisms are etale locally closed immersions \cite[4.11]{Liu}).
		
		\subsubsection*{Step 3}
		We are going to compute the closed subscheme $U_\lambda \times_{\operatorname{Gr}_G} \operatorname{Gr}_G^\nabla$. By definition $g \in U_\lambda(A)$ is contained in this closed subscheme if and only if
		$$
		u^e \operatorname{dlog}_u(gu^\lambda) \in \mathfrak{g} \otimes_{\mathcal{O}} A[u]
		$$
		Lemma~\ref{lem-dlog} shows this is equivalent to asking that 
		$$
		u^e\operatorname{Ad}(u^{-\lambda}) \operatorname{dlog}_u(g)  \in \mathfrak{g} \otimes_{\mathcal{O}} A[u]
		$$
		It will therefore be necessary to compute $\operatorname{dlog}_u(g)$ and we will do this using the following two observations:
		
		\begin{itemize}
			\item If $g = x_{\alpha^\vee}(a)$ for $a \in A[u]$ then $\operatorname{dlog}_u(g) = dx_{\alpha^\vee}(\frac{d}{du} a)$. This follows from Example~\ref{ex-Ga} and the functoriality of $\operatorname{dlog}_u$.
			\item If $\alpha^\vee +\beta^\vee \neq 0$ then 
			\begin{equation}\label{eq-adjointaction}
				\operatorname{Ad}(x_{\alpha^\vee}(a)) dx_{\beta^\vee}(b) = dx_{\beta^\vee}(b) + \sum_{i >0} c_{i1} d_{i\alpha^\vee + \beta^\vee}(a^i b)
			\end{equation}
			for some $c_{ij} \in \mathbb{Z}$ independent of $a$ and $b$. This can be seen by passing the formula
			\begin{equation*}\label{eq-commutator}
				x_{\alpha^\vee}(a) x_{\beta^\vee}(b) x_{\alpha^\vee}(a)^{-1} x_{\beta^\vee}(b)^{-1}= \prod_{i,j>0} x_{i\alpha^\vee + j\beta^\vee}(c_{ij} a^ib_j)	
			\end{equation*}
			found in e.g. \cite[1.2.(5)]{Janbook} to the Lie algebra.
		\end{itemize}
		\subsubsection*{Step 4}
		Lemma~\ref{lem-dlog} shows that $\operatorname{dlog}_u(g x_{\beta^\vee}(b)) = \operatorname{Ad}(x_{\beta^\vee}(-b)) \operatorname{dlog}_u(g) + \operatorname{dlog}_u(x_{\beta^\vee}(b))$. This, together with the two bullet points from Step 3, allows an inductive computation of $\operatorname{dlog}_u(g)$ for $g = \prod_{\langle \alpha^\vee,\lambda \rangle >0} x_{\alpha^\vee}(a_{\alpha^\vee})$ with $a_{\alpha^\vee} \in A[u]$. We see that $\operatorname{dlog}_u(g)$ can be expressed as a sum, over $\gamma^\vee$ with $\langle \gamma^\vee,\lambda \rangle >0$, of terms
		$$
		dx_{\gamma^\vee}\left( \frac{d}{du}a_{\gamma^\vee} + C_{\gamma^\vee} \right) 
		$$
		where $C_{\gamma^\vee}$ is a $\mathbb{Z}$-linear combination of products of the $a_{\alpha^\vee}$ and $\frac{d}{du}(a_{\alpha^\vee})$ for those roots $\alpha^\vee$ with $\langle \alpha^\vee,\lambda \rangle >0$ and $\langle \gamma^\vee - \alpha^\vee,\lambda \rangle >0$. We can therefore write $C_{\gamma^\vee} = C_{\gamma^\vee,0} + C_{\gamma^\vee,1} u + C_{\gamma^\vee,2} u^2 + \ldots$ with each $C_{\gamma^\vee,i} = C_{\gamma^\vee,i}(a_{\alpha^\vee,j})$ a polynomial in the coefficients of the $a_{\alpha^\vee}$ for $\alpha^\vee$ with $0 < \langle \alpha^\vee,\lambda \rangle  < \langle \gamma^\vee,\lambda \rangle$. These polynomials have $\mathbb{Z}$-coefficients and depend only on the order in which the product defining $g$ is taken. It follows that $u^e \operatorname{Ad}(u^{-\lambda})\operatorname{dlog}_u(g)$ can likewise be expressed as a sum of the terms
		\begin{equation}\label{eq-dlogU}
			u^{e - \langle \gamma^\vee,\lambda \rangle}dx_{\gamma^\vee}\left( \frac{d}{du}a_{\gamma^\vee} + C_{\gamma^\vee} \right) 		
		\end{equation}
		The assumption that $\langle \gamma^\vee,\lambda \rangle - e +1\leq \operatorname{char}\mathbb{F}$ means there exist unique polynomials $D_{\gamma^\vee,i} = D_{\gamma^\vee,i}(a_{\alpha^\vee,i})$ in the coefficients of the $a_{\alpha}$ for $\alpha^\vee $ with $0 < \langle \alpha^\vee,\lambda \rangle < \langle \gamma^\vee,\lambda \rangle$ so that if 
		$$
		D_{\gamma^\vee} = D_{\gamma^\vee,1} u + D_{\gamma^\vee,2} u^2 + \ldots + D_{\gamma^\vee,\langle \gamma^\vee,\lambda \rangle-e+1} u^{\langle \gamma^\vee,\lambda \rangle-e+1}
		$$ then $\frac{d}{du} D_{\gamma^\vee} \equiv C_{\gamma^\vee}$ modulo $u^{\langle \gamma^\vee,\lambda \rangle -e+1}$. Again these polynomials depend only on the order of the product defining $g$. Thus \eqref{eq-dlogU} is contained in $\mathfrak{g}_{\gamma^\vee} \otimes_{\mathcal{O}} A[u]$ if and only if 
		$$
		a_{\gamma^\vee} - D_{\gamma^\vee} \in A + u^{\langle \gamma^\vee,\lambda\rangle -e +1}A[u]
		$$
		It follows that there is an isomorphism
		$$
		\prod_{\langle \gamma^\vee,\lambda \rangle > 0}\mathbb{A}^{\operatorname{min}\lbrace e,\langle \gamma^\vee,\lambda \rangle\rbrace} \rightarrow U_\lambda \times_{\operatorname{Gr}_G} \operatorname{Gr}^\nabla_G
		$$
		sending $(a_{\gamma^\vee,i})_{\gamma^\vee}$ onto
		$$
		\prod_{\gamma^\vee} x_{\gamma^\vee}\left( a_0 + a_{\gamma^\vee,1} u^{\langle \gamma^\vee,\lambda \rangle -1} + a_{\gamma^\vee,2} u^{\langle \gamma^\vee,\lambda \rangle -2} +\ldots + a_{\gamma^\vee,\operatorname{min}\lbrace e,\langle \gamma^\vee,\lambda \rangle-\rbrace-1} u^{\operatorname{max}\lbrace 1,\langle \gamma^\vee,\lambda \rangle\rbrace -e+1} + D_{\gamma^\vee} \right) 
		$$
		This shows that $U_\lambda \times_{\operatorname{Gr}_G} \operatorname{Gr}_{G}^{\nabla}$ is smooth of the claimed dimension.
		
		\subsubsection*{Step 5}
		It remains to show that $\operatorname{Gr}_{G,\lambda,\mathbb{F}} \times_{\operatorname{Gr}_G} \operatorname{Gr}_{G}^\nabla$ is irreducible. For this recall the action of $\mathbb{G}_m$ on $\operatorname{Gr}_G \otimes_{\mathcal{O}} \mathbb{F}$ via loop rotations: if $t \in A^\times$ and $(\mathcal{E},\iota) \in \operatorname{Gr}_{G,A}$ then
		$$
		t \cdot (\mathcal{E},\iota) = (x_t^*\mathcal{E},x_t^*\iota)
		$$
		where $x_t$ is the automorphism of $\operatorname{Spec}A[u]$ given by $u \mapsto tu$. If $\mathcal{E}$ has a trivialisation with respect to which $\iota$ is given by $g(u) \in G(A[u,E(u)^{-1}])$ then $t\cdot(\mathcal{E},\iota)$ corresponds to the trivial $G$-torsor and $g(tu) \in G(A[u,E(u)])$. In particular, this demonstrates why this is only an action on $\operatorname{Gr}_G \otimes_{\mathcal{O}} \mathbb{F}$; if $A$ is not an $\mathbb{F}$-algebra then $g(ut)$ need not be in $G(A[u,E(u)])$.
		
		This action stabilises both $\operatorname{Gr}_{G,\lambda,\mathbb{F}}$ and $\operatorname{Gr}_G^{\nabla}$.  The former stabilisation is clear and the latter follows from the observation, using Remark 6.4 and the chain rule, that $\operatorname{dlog}_u(g(ut)) = t x_t^* \operatorname{dlog}_u(g)$. The smoothness of $\operatorname{Gr}_{G,\lambda,\mathbb{F}} \times_{\operatorname{Gr}_G} \operatorname{Gr}_{G}^\nabla$ then ensures it is an affine bundle over its $\mathbb{G}_m$-fixed points, see \cite[Theorem 13.47]{MilneAlggroups}. Since the fixed point locus in $\operatorname{Gr}_{G,\lambda,\mathbb{F}}$ is the $G$-orbit of $\mathcal{E}_{\lambda,\mathbb{F}}$, and since this is contained in $\operatorname{Gr}_G^\nabla$, we conclude that the fixed point locus of $\operatorname{Gr}_{G,\lambda,\mathbb{F}} \times_{\operatorname{Gr}_G} \operatorname{Gr}_{G}^\nabla$ is also this $G$-orbit. As this orbit is irreducible the same is true for $\operatorname{Gr}_{G,\lambda,\mathbb{F}} \times_{\operatorname{Gr}_G} \operatorname{Gr}_{G}^\nabla$.
	\end{proof}

	\section{Naive cycle identities}\label{sec-naivecycles}
	
	Here we use Proposition~\ref{prop-dim} to produce a basic description of the cycles associated to $M_{\mu} \otimes_{\mathcal{O}} \mathbb{F}$.
	\begin{definition}
		A $d$-dimensional cycle on any Noetherian (ind)-scheme $X$ is a $\mathbb{Z}$-linear combination of integral closed subschemes in $X$ of dimension $d$. The group of all such cycles is denoted $Z_d(X)$. If $\mathcal{F}$ is any coherent sheaf on $X$ we write
		$$
		[\mathcal{F}] = \sum_Z m(Z,\mathcal{F}) [Z]
		$$
		where the sum runs over $d$-dimensional integral closed subschemes $Z$ in $X$ and $m(Z,Y)$ denotes the $\mathcal{O}_{Z,\xi}$-dimension of $\mathcal{F}_\xi$ for $\xi \in Z$ the generic point. If $i:Y \subset X$ is a closed subscheme then we set $[Y] = [i_*\mathcal{O}_Y]$. If $X$ is a scheme then \cite[02S9]{stacks-project} shows that $Z_d(X)$ can alternatively be defined as the cokernel of the map
		$$
		K_0(\operatorname{Coh}_{\leq d-1}(X)) \rightarrow K_0(\operatorname{Coh}_{\leq d}(X))
		$$
		where $\operatorname{Coh}_{\leq d}(X)$ denotes the category of coherent sheaves on $X$ with support of dimension $\leq d$. Then $[\mathcal{F}]$ coincides with the image of the class of $\mathcal{F}$.
	\end{definition}
	
	\begin{definition}\label{def-Clambda}
		For $\lambda \in X_*(T)$ define $\mathcal{C}_\lambda \subset \operatorname{Gr}_G \otimes_{\mathcal{O}} \mathbb{F}$ as the closure of $\operatorname{Gr}_{G,\lambda,\mathbb{F}} \times_{\operatorname{Gr}_G} \operatorname{Gr}^\nabla$. Proposition~\ref{prop-dim} ensures that $\mathcal{C}_\lambda$ is an integral closed subscheme of dimension 
		$$
		\sum_{\text{positive }\alpha^\vee} \operatorname{min}\lbrace e,\langle \alpha^\vee,\lambda \rangle \rbrace
		$$
		provided $\langle \alpha^\vee,\lambda \rangle \leq \operatorname{char}\mathbb{F} +e-1$ for every positive root $\alpha^\vee$. 
	\end{definition}
	
	\begin{proposition}\label{prop-firstcycle}
		Assume that $G$ admits a twisting element $\rho \in X_*(T)$ and suppose that $\mu = (\mu_1,\ldots,\mu_e)$ with $\mu_i \in X_*(T)$ strictly dominant. If $\operatorname{char}\mathbb{F} >0$ assume also that
		$$
		\sum_i \langle \alpha^\vee,\mu_i \rangle \leq \operatorname{char}\mathbb{F} + e-1
		$$
		for every positive root $\alpha^\vee$. Then there exists $m_{\lambda} \in \mathbb{Z}$ so that as $e|R^+|$-dimensional cycles in $\operatorname{Gr}_{G}\otimes_{\mathcal{O}} \mathbb{F}$
		$$
		[M_\mu \otimes_{\mathcal{O}}\mathbb{F}] = \sum_{\lambda} m_{\lambda} [\mathcal{C}_{\lambda+e\rho}]
		$$
		with the sum running over dominant $\lambda \leq \mu_1+\ldots +\mu_e - e\rho$. Furthermore, $m_{\mu_1+\ldots+\mu_e - e\rho}=1$.
	\end{proposition}
	Later on we will give the $m_\lambda$ a representation theoretic interpretation (see Theorem~\ref{thm-cycles}).
	\begin{proof}
		Propositions~\ref{prop-include} and~\ref{prop-schubertcontain} ensures $M_\mu \otimes_{\mathcal{O}} \mathbb{F}$ factors through $\operatorname{Gr}_{G}^\nabla$ and $Y_{G,\leq \mu}$. Lemma~\ref{lem-Yschuprop} implies $(Y_{G,\leq \mu}\otimes_{\mathcal{O}}\mathbb{F})_{\operatorname{red}} = \bigcup_{\lambda \leq \mu_1+\ldots+\mu_e-e\rho} \operatorname{Gr}_{G,\lambda+e\rho,\mathbb{F}}$ and so
		\begin{equation}\label{eq-initialcontain}
			(M_\mu \otimes_{\mathcal{O}} \mathbb{F})_{\operatorname{red}} \subset \bigcup_{\lambda \leq \mu_1+\ldots+\mu_e-e\rho} \operatorname{Gr}_{G,\lambda+e\rho,\mathbb{F}} \times_{\operatorname{Gr}_G} \operatorname{Gr}_{G}^\nabla \subset \bigcup_{\lambda \leq \mu_1+\ldots+\mu_e- e\rho}  \mathcal{C}_{\lambda+e\rho}
		\end{equation}
		These unions run over $\lambda$ which are not necessarily dominant. To show that the containment still holds with the union running over dominant $\lambda$ we use the assumption that each $\mu_i$ is strictly dominant. This ensures that $\operatorname{dim}G/P_{\mu_i} = |R|$ and so $\operatorname{dim}M_\mu \otimes_{\mathcal{O}} \mathbb{F} = e|R^+|$. Thus
		$$
		\operatorname{dim}\mathcal{C}_{\lambda+e\rho} = \sum_{\text{positive }\alpha^\vee} \operatorname{min}\lbrace e,\langle \alpha^\vee,\lambda +e\rho \rangle \rbrace \leq \operatorname{dim}M_\mu \otimes_{\mathcal{O}} \mathbb{F}
		$$
		with equality if and only if $\langle \alpha^\vee,\lambda+e\rho \rangle \geq e$ for every positive $\alpha^\vee$. Notice that $\langle \alpha^\vee, \lambda + e\rho \rangle \geq e$ for every positive root $\alpha^\vee$ (equivalently every simple root) if and only if $\lambda$ is dominant, because $\langle \alpha^\vee,\rho \rangle  = 1$ whenever $\alpha^\vee$ is simple. Thus, \eqref{eq-initialcontain} can be refined to:
		$$
		(M_\mu \otimes_{\mathcal{O}} \mathbb{F})_{\operatorname{red}} \subset \bigcup  \mathcal{C}_{\lambda+e\rho}
		$$
		with the union running over dominant $\lambda \leq \mu_1+\ldots+\mu_e- e\rho$. In other words,
		$$
		[M_\mu \otimes_{\mathcal{O}}\mathbb{F}] = \sum_{\lambda'} m_{\lambda} [\mathcal{C}_{\lambda+e\rho}]
		$$
		as desired. To finish the proof we have to $m_{\mu_1+\ldots+\mu_e-e\rho} = 1$, and for this it suffices to show $\mathcal{C}_{\mu_1+\ldots+\mu_e} \subset M_\mu \otimes_{\mathcal{O}} \mathbb{F}$ and this closed immersion becomes an open immersion after restricting to an open subset of $\mathcal{C}_{\mu_1+\ldots+\mu_e}$. Recall from Lemma~\ref{lem-Yschuprop} that $\operatorname{Gr}_{G,\mu_1+\ldots+\mu_e,\mathbb{F}}$ is open in $Y_{G,\leq \mu} \otimes_{\mathcal{O}} \mathbb{F}$. Therefore,
		$$
		U := M_\mu \times_{\operatorname{Gr}_G} \operatorname{Gr}_{G, \mu_1+\ldots+\mu_e,\mathbb{F}}
		$$
		is open in $M_\mu \otimes_{\mathcal{O}} \mathbb{F}$. It is also non-empty because it is easy to see that $\mathcal{E}_{\mu_1+\ldots+\mu_e,\mathbb{F}} \in M_\mu$. Therefore, $U$ has dimension equal $\operatorname{dim}M_\mu \otimes_{\mathcal{O}} \mathbb{F}$. On the other hand $U$ is a closed subscheme of $\operatorname{Gr}_{G,\mu_1+\ldots+\mu_e,\mathbb{F}} \times_{\operatorname{Gr}_G} \operatorname{Gr}^\nabla$. We saw in Proposition~\ref{prop-dim} that $\operatorname{Gr}_{G, \mu_1+\ldots+\mu_e,\mathbb{F}} \times_{\operatorname{Gr}_G} \operatorname{Gr}^\nabla$ is smooth and irreducible of the same dimension. Thus, $U = \operatorname{Gr}_{G,\mu_1+\ldots+\mu_e,\mathbb{F}} \times_{\operatorname{Gr}_G} \operatorname{Gr}^\nabla$. As $\mathcal{C}_{\mu_1+\ldots+\mu_e}$ is the closure of $U$ we conclude that $\mathcal{C}_{\mu_1+\ldots+\mu_e} \subset M_\mu \otimes_{\mathcal{O}} \mathbb{F}$, and that this inclusion is an isomorphism over an open subset.
	\end{proof}
	
	\section{Irreducibility}\label{sec-irred}
	
	\begin{theorem}\label{thm-irred}
		Assume that $G$ contains a twisting element $\rho \in X_*(T)$ and suppose that $\lambda \in X_*(T)$ is dominant. If $\operatorname{char}\mathbb{F} >0$ then assume also that
		$$
		\langle \alpha^\vee,\lambda + e\rho \rangle \leq \operatorname{char}\mathbb{F} + e-1
		$$
		for every positive root $\alpha^\vee$. Then, as cycles
		$$
		[M_{(\lambda+\rho,\rho,\ldots,\rho)} \otimes_{\mathcal{O}} \mathbb{F}] =[\mathcal{C}_{\lambda+e\rho}]
		$$
		In other words, $M_{(\lambda+\rho,\rho,\ldots,\rho)} \otimes_{\mathcal{O}} \mathbb{F}$ is irreducible and generically reduced.
	\end{theorem}
	
	For this we need the following lemma. Here we use the notation $Y_{G, \leq \mu_2+\ldots+\mu_e,\mathbb{F}} \subset \operatorname{Gr}_{G} \otimes_{\mathcal{O}} \mathbb{F}$ to denote $Y_{G,\leq \eta} \otimes_{\mathcal{O}} \mathbb{F}$ for any tuple of cocharacters $\eta$ whose sum equals $\mu_2+\ldots+\mu_e$ (this special fibre is independent of $\eta$ by Lemma~\ref{lem-Yschuprop}).
	
	\begin{lemma}\label{lem-convolutionMmu}
		Suppose $\mu = (\mu_1,\ldots,\mu_e)$ with each $\mu_i \in X^*(T)$ dominant. Then any $\mathbb{F}$-valued point of $M_{\mu}(\mathbb{F})$ is contained in
		$$
		g_1 u^{\mu_1} Y_{G, \leq \mu_2+\ldots+\mu_e,\mathbb{F}} 
		$$
		for some $g_1 \in G$.
	\end{lemma}
	\begin{proof}
		Consider the convolution grassmannian $\operatorname{Gr}^{\operatorname{con}}_{G}$ whose $A$-points classify pairs $\mathcal{E}_{2},\mathcal{E}_1$ of $G$-torsors on $\operatorname{Spec}A[u]$ together with isomorphisms $\iota_1:\mathcal{E}_1|_{\operatorname{Spec}A[u,(u-\pi_1)^{-1}]} \xrightarrow{\sim} \mathcal{E}^0|_{\operatorname{Spec}A[u,(u-\pi_1)^{-1}]}$ and 
		$$
		\iota_2: \mathcal{E}_{2}|_{\operatorname{Spec}A[u,\prod_{i=2}^e (u-\pi_i)^{-1}]}\xrightarrow{\sim} \mathcal{E}_{1}|_{\operatorname{Spec}A[u,\prod_{i=2}^e (u-\pi_i)^{-1}]}
		$$
		Then $(\mathcal{E}_1,\mathcal{E}_2,\iota_1,\iota_2) \mapsto (\mathcal{E}_2, \iota_1 \circ \iota_2)$ induces a convolution morphism
		$$
		\operatorname{Gr}^{\operatorname{con}}_G \rightarrow \operatorname{Gr}_G
		$$
		which is well-known to be both proper, and an isomorphism after applying $\otimes_{\mathcal{O}} E$.  In particular, we can define $M^{\operatorname{con}}_\mu \subset \operatorname{Gr}^{\operatorname{con}}_G$ as the closure of $G/P_{\mu_1} \times_E \ldots \times_E G/P_{\mu_e} \subset \operatorname{Gr}^{\operatorname{con}}_G \otimes_{\mathcal{O}} E$, just as in Definition~\ref{def-mmu}. Then convolution restricts to a morphism $M^{\operatorname{con}}_\mu \rightarrow M_\mu$ which is proper and an isomorphism on generic fibres. In particular, this morphism is surjective on $\mathbb{F}$-valued points.
		
		Now write $\operatorname{Gr}_G^{(e-1)}$ for the affine grassmannian defined as in Definition~\ref{sub-Grdef} but for the $e-1$-tuple $(\pi_2,\ldots,\pi_e)$. As in Definition~\ref{def-schubert} we can consider the closed subscheme $Y_{G,\leq (\mu_2,\ldots,\mu_e)}$ whose special fibre coincides with $Y_{G,\leq \mu_2+ \ldots+\mu_e,\mathbb{F}}$. Then we can define a closed subscheme $Z$ of $\operatorname{Gr}^{\operatorname{con}}_G$ by imposing the conditions:
		\begin{itemize}
			\item The point of $\operatorname{Gr}_{G,1}$ defined by $(\mathcal{E}_1,\iota_1)$ lies in $G/P_{\mu_1}$.
			\item For any trivialisation $\mathcal{E}_1 \cong \mathcal{E}^0$ over $\operatorname{Spec}A[u]$ the point of $\operatorname{Gr}_{G}^{(e-1)}$ defined by $\mathcal{E}_1$ and the trivialisation $$
			\mathcal{E}_{2}|_{\operatorname{Spec}A[u,\prod_{i=2}^e (u-\pi_i)^{-1}]}\xrightarrow{\iota_2} \mathcal{E}_{1}|_{\operatorname{Spec}A[u,\prod_{i=2}^e (u-\pi_i)^{-1}]} \cong \mathcal{E}^0|_{\operatorname{Spec}A[u,\prod_{i=2}^e (u-\pi_i)^{-1}]} 
			$$
			lies in $Y_{G, \leq (\mu_2,\ldots,\mu_e)}$
		\end{itemize}
		Then $Z \otimes_{\mathcal{O}} E \cong G/P_{\mu_1} \times_E Y_{G,\leq \mu_2} \times_E \ldots \times_E Y_{G, \leq \mu_e}$ and so contains $M^{\operatorname{con}}_\mu \otimes_{\mathcal{O}} E$. As $Z \subset \operatorname{Gr}^{\operatorname{con}}$ is closed it follows $Z$ contains $M_\mu$.  It therefore suffices to show that any $\mathbb{F}$-valued point of $\operatorname{Gr}_G$ which is in the image of $Z$ under convolution lies in $g_1 u^{\mu_1} Y_{G, \leq \mu_2+\ldots+\mu_e,\mathbb{F}}$ for some $g_1 \in G$. But this is clear because if $(\mathcal{E}_1,\mathcal{E}_2,\iota_1,\iota_2) \in Z(\mathbb{F})$ and $\beta: \mathcal{E}_1 \cong \mathcal{E}^0$ is a trivialisation over $\operatorname{Spec}\mathbb{F}[u]$ so that $\iota_1  = g(u-\pi_1)^{\mu_1} \circ \beta$ then convolution sends $(\mathcal{E}_1,\mathcal{E}_2,\iota_1,\iota_2)$ onto 
		$$
		(\mathcal{E}_2,\iota_1\circ \iota_2) = (\mathcal{E}_2, g(u-\pi_1)^{\mu_1} \circ \beta \circ \iota_2) = g(u-\pi_1)^{\mu_1} \cdot (\mathcal{E}_2,\beta \circ \iota_2)
		$$
		and this last point is contained in $g(u-\pi_1)^{\mu_1} Y_{G, \leq (\mu_2,\ldots,\mu_e)} \otimes_{\mathcal{O}} \mathbb{F}$ as desired.
	\end{proof}
	\begin{proof}[Proof of Theorem~\ref{thm-irred} A]
		In view of Proposition~\ref{prop-firstcycle}, the theorem will follow if we can show that $\mathcal{C}_{\lambda'+e \rho} \not\subset M_{(\lambda+\rho,\rho,\ldots,\rho)} \otimes_{\mathcal{O}} \mathbb{F}$ for any dominant $\lambda' < \lambda$. We will do this by exhibiting an $\mathbb{F}$-valued point $g_0 \in  G$ so that 
		\begin{equation}\label{eq-calE}
			\mathcal{E} = (\mathcal{E}^0, u^{\lambda'+\rho} g_0 u^{(e-1)\rho}) \in \mathcal{C}_{\lambda'+e\rho}	
		\end{equation}
		but cannot be written as in Lemma~\ref{lem-convolutionMmu}. We refer to Example~\ref{ex-GL3example} for an illustration of some of the key steps in our argument in the case $G = \operatorname{GL}_3$.
		
		\subsubsection*{Step 1} We begin by constructing $g_0$ as above. Recall $U \subset B$ denotes the unipotent subgroup and write $H = u \operatorname{dlog}_u(u^{\lambda'+\rho})$ which, by Lemma~\ref{lem-dlog}, lies in $\operatorname{Lie}(T)$. We claim there exists $g_0 \in U$ so that 
		\begin{equation}\label{eq-adjointcondition}
			\operatorname{Ad}(g_0^{-1}) H = H + \sum_{\text{simple }\alpha^\vee} dx_{\alpha^\vee}(A_{\alpha^\vee})
		\end{equation}
		for some $A_{\alpha^\vee} \in \mathbb{F}^\times$ (recall here that $dx_{\alpha^\vee}: \mathbb{G}_a \rightarrow \mathfrak{g}_{\alpha^\vee}$ is the derivative of the root homomorphism $x_{\alpha^\vee}: \mathbb{G}_a \rightarrow G$). To see this, write $g_0^{-1} = \prod_{\beta^\vee >0} x_{\beta\vee}(a_{\beta^\vee})$ with the product running over all positive roots and taken in some fixed order. Then the description of the adjoint action from \eqref{eq-adjointaction} allows us to write
		$$
		\operatorname{Ad}(g_0^{-1}) H = H + \sum_{\beta^\vee >0} dx_{\beta_\vee}( \beta^\vee(H) a_{\beta^\vee} + P_{\beta^\vee})
		$$
		where $P_{\beta^\vee}$ is some polynomial in the $a_{\gamma^\vee}$ for roots $\gamma^\vee < \beta^\vee$. Lemma~\ref{lem-dlog} implies $\beta^\vee( H ) = \langle \beta^\vee,  \lambda'+ \rho\rangle$ and we claim this is non-zero in $\mathbb{F}$. Indeed, if $\beta^\vee_{\operatorname{max}}$ is the maximal root above $\beta^\vee$ then 
		$$
		0 < \beta^\vee( H ) \leq \langle \beta_{\operatorname{max}}^\vee,  \lambda'+ \rho\rangle \leq \langle \beta_{\operatorname{max}}^\vee,  \lambda+ \rho\rangle 
		$$
		where the first inequality follows from the dominance of $\lambda'$ and the last inequality uses that $\lambda'< \lambda$ and that $\langle \beta^\vee_{\operatorname{max}}, \alpha \rangle \geq 0$ for any simple coroot $\alpha$ (a consequence of the maximality of $\beta^\vee_{\operatorname{max}}$). The bound on $\lambda$ in the theorem therefore gives 
		$$
		0 < \beta^\vee(H) < \operatorname{char}\mathbb{F}
		$$
		as claimed. Consequently, \eqref{eq-adjointcondition} forces $a_{\beta^\vee}$ for non-simple $\beta^\vee$ to be expressed in terms of the $a_{\alpha^\vee}$ for $\alpha^\vee$ simple and implies $a_{\alpha^\vee} \neq 0$. In particular, there exists $g_0$ satisfying \eqref{eq-adjointcondition} as claimed.
		\subsubsection*{Step 2}
		Next we show $\mathcal{E} = (\mathcal{E}^0, u^{\lambda'+\rho} g_0 u^{(e-1)\rho}) \in \mathcal{C}_{\lambda'+e\rho}	$ for $g_0$ as in Step 1. Since $\lambda'+\rho$ is dominant we have $u^{\lambda'+\rho} g_0 u^{-(\lambda'+\rho)} \in L^+G$ for any $g_0  \in U$. Therefore, $\mathcal{E} \in \operatorname{Gr}_{G,\lambda'+e\rho,\mathbb{F}}$ and we just need to check when $\mathcal{E} \in \operatorname{Gr}^{\nabla}$. Since $g_0 \in G$ we have $\operatorname{dlog}_u(g_0) = 0$. Lemma~\ref{lem-dlog} therefore implies $\mathcal{E} \in \operatorname{Gr}^{\nabla}$ if and only if 
		$$
		\operatorname{Ad}(u^{-(e-1)\rho}) \left( \operatorname{Ad}(g_0^{-1}) H \right) \in u^{-(e-1)}\mathfrak{g}[u]
		$$
		As $g_0 \in U$ this is equivalent so asking that 
		$$
		\operatorname{Ad}(g_0^{-1}) H \in \operatorname{Lie}(T) \oplus \bigoplus_{\text{simple }\alpha^\vee } \mathfrak{g}_{\alpha^\vee}
		$$
		which is satisfied for $g_0$ as in Step 1.
		
		\subsubsection*{Step 3} Now suppose $\mathcal{E}$ satisfies the hypothesis of Lemma~\ref{lem-convolutionMmu}. We claim that
		$$
		\mathcal{E} \in u^{w(\lambda + \rho)} Y_{G,\leq (e-1)\rho, \mathbb{F}}
		$$
		for some $w \in W$. To see this consider the action of $\mathbb{G}_m \times T$ on $\operatorname{Gr}_{G,\mathbb{F}}$ where $T$ acts by left multiplication and $\mathbb{G}_m$ acts via the loop rotations described in Step 5 of the proof of Proposition~\ref{prop-dim}. Since $g_0 \in G$ the $1$-parameter subgroup $\mathbb{G}_m \rightarrow \mathbb{G}_m \times T$ given by $t \mapsto (t,t^{-(\lambda'+\rho)})$ induces an action fixing $\mathcal{E}$. Therefore, this $1$-parameter subgroup acts on the non-empty closed subscheme of $G/P_{\lambda+\rho} \subset \operatorname{Gr}_{G,\mathbb{F}}$ consisting of $g_1 \in G/P_{\lambda+\rho}$ with 
		$$
		\mathcal{E} \in g_1 u^{\lambda+\rho} Y_{G, \leq (e-1)\rho,\mathbb{F}}
		$$
		It follows that this closed subscheme contains  $\operatorname{lim}_{t \rightarrow 0} t \cdot g_1$. Since the loop rotation fixes $G/P_{\lambda+\rho}$ we have $t \cdot g_1 = t^{-(\lambda'+\rho)} g_1$.  The strict dominance of  $\lambda'+\rho$ therefore implies $\operatorname{lim}_{t \rightarrow 0} t \cdot g_1$ is a $T$-fixed point in $G/P_{\lambda+\rho}$. In other words, we can choose $g_1$ as above  representing an element in $W$, as claimed.
		
		\subsubsection*{Step 4} Continue to assume $\mathcal{E}$ satisfies the hypothesis of Lemma~\ref{lem-convolutionMmu}. Then Step 3 says
		$$
		u^{(-w(\lambda+\rho) + \lambda'+\rho)} g_0 u^{(e-1)\rho} \in Y_{G, \leq (e-1)\rho, \mathbb{F}}
		$$
		for some $w \in W$. In Step 5 we will show any $g_0$ as in Step 1 lies in $B^- w_0 B^-$ for $w_0 \in W$ the longest element and $B^- = P_{\lambda+\rho}$ the Borel opposite $B$. Here we explain how this will lead to a contradiction and finish the proof. First, consider the $\mathbb{G}_m$-action on $u^{(-w(\lambda+\rho) + \lambda'+\rho)} g_0 u^{(e-1)\rho}$ induced by a strictly dominant cocharacter. Since $g_0 \in U$ we have
		$$
		\operatorname{lim}_{t \rightarrow 0} t \cdot u^{(-w(\lambda+\rho) + \lambda'+\rho)} g_0 u^{(e-1)\rho} = u^{(-w(\lambda+\rho) + \lambda'+\rho)} u^{(e-1)\rho} \in Y_{G,\leq (e-1)\rho,\mathbb{F}} 
		$$
		Therefore $-w(\lambda+\rho) + \lambda'+\rho \leq 0$. On the other hand, if we consider the $\mathbb{G}_m$-action induced by a strictly anti-dominant cocharacter then the containment $g_0 \in B^- w_0 B^-$ implies
		$$
		\operatorname{lim}_{t \rightarrow 0} t \cdot u^{(-w(\lambda+\rho) + \lambda'+\rho)} g_0 u^{(e-1)\rho} = u^{(-w(\lambda+\rho) + \lambda'+\rho)} w_0 u^{(e-1)\rho} \in Y_{G,\leq (e-1)\rho,\mathbb{F}} 
		$$
		Therefore $w_0\left( -w(\lambda+\rho) + \lambda'+\rho \right) \leq 0$ and so $-w(\lambda+\rho) + \lambda'+\rho \geq 0$. Consequently, $w(\lambda+\rho) = \lambda'+\rho$ which is impossible since $\lambda' < \lambda$.
		
		\subsubsection*{Step 5} To conclude choose $s \in W$ so that $g_0^{-1} = b_1 s b_2$ with $b_1\in B^-$ and $b_2 \in U^-$ (the unipotent radical of $B^-$). We have to show $s = w_0$ and we will do this by using that, due to our choice of $g_0$ in Step 1, the application of first $\operatorname{Ad}(b_2)$, then $\operatorname{Ad}(s)$, and then $\operatorname{Ad}(b_1)$ sends $H$ onto an element of the form
		$$
		H + \sum_{\text{simple }\alpha^\vee} dx_{\alpha^\vee}(A_{\alpha^\vee})
		$$
		with $A_{\alpha^\vee} \in \mathbb{F}^\times$.  First, we can write
		$$
		\operatorname{Ad}(sb_2)H = s(H) + \sum_{\beta^\vee<0 } dx_{s(\beta^\vee)}(b_{\beta^\vee})
		$$
		for some $b_{\beta^\vee} \in \mathbb{F}$. Now, $\operatorname{Ad}(b_1)$ maps $dx_{\alpha^\vee}(1)$ onto a linear combination of  $dx_{\beta^\vee}$'s with $\beta^\vee \leq \alpha^\vee$ and the coefficient of $\alpha^\vee$ non-zero. Consequently, if $\gamma^\vee > 0 $ is a root maximal for the property that $b_{s^{-1}(\gamma^\vee)} \neq 0$  then the projection of 
		$$
		\operatorname{Ad}(b_1)\left( H + \sum_{\beta^\vee<0 } dx_{s(\beta^\vee)}(b_{\beta^\vee})\right)
		$$
		to the $\gamma^\vee$-root subspace will be a non-zero multiple of $d_{\gamma^\vee}(b_{s^{-1}(\gamma^\vee)})$. This forces $\gamma^\vee$ to be simple. On the other hand, if $\alpha^\vee$ is simple then $b_{s^{-1}(\alpha^\vee)} \neq 0$ because otherwise the projection of $\operatorname{Ad}(b_1)\left( H + \sum_{\beta^\vee<0 } dx_{s(\beta^\vee)}(b_{\beta^\vee})\right)$ onto the $\alpha^\vee$-root subspace will be zero. Therefore, $s^{-1}(\alpha^\vee) <0$ for each simple $\alpha^\vee >0$. This implies $s = w_0$, which finishes the proof.
	\end{proof}
	\begin{example}\label{ex-GL3example}
		Suppose $G = \operatorname{GL}_3$ and write $H = \operatorname{diag}(a,b,c)$ for integers $a> b > c$. Then $g_0 = \left( \begin{smallmatrix}
			1 & x & z \\ 0 & 1 & y \\ 0 & 0 & 1
		\end{smallmatrix} \right) \in U$ satisfies the conditions imposed in Step 1 precisely when $z = -xy \frac{c-b}{a-c}$ and $x,y \neq 0$. In this case it is easy to see directly that $g_0 \in B^- w_0 B^-$ as claimed in Step 4. However, we can also use run the arguments from Step 5 in this example to illustrate the idea. For example, if $g_0^{-1} = b_1 s b_2$ for $b_i \in B^-$ and $s = \left( \begin{smallmatrix}
			0 & 1 & 0 \\ 0 & 0 & 1 \\ 1 & 0 & 0
		\end{smallmatrix}  \right)$ then $\operatorname{Ad}(sb_2)$ would send $H$ onto a matrix of the form
		$$
		\begin{pmatrix}
			* & *_1 & *_2 \\ * & * & 0 \\ * & * & *
		\end{pmatrix}
		$$
		with the $*$'s possibly zero. We would then need $\operatorname{Ad}(b_1)$ to send this matrix onto something with non-zero entries directly above the diagonal and  top right entry zero. Considering how $\operatorname{Ad}(b_1)$ acts shows that $*_2 =0$, otherwise the  top right entry could never be killed. On the other hand, if $*_2 =0$ then it is impossible for an application of $\operatorname{Ad}(b_1)$ to create a non-zero entry in the $(2,3)$-th position.
	\end{example}
	\section{Equivariant sheaves and their cycles}\label{sec-equivariantsheaves}
	
	Our goal is now to give the coefficients appearing in Proposition~\ref{prop-firstcycle} a representation theoretic meaning. To do this we will relate these cycle identities to relations between global sections of line bundles on the $M_\mu$. 
	\begin{sub}\label{sub-TequiKtheory}
		Suppose $X$ is a finite type $\mathbb{F}$-scheme equipped with an action of the torus $T$. Then we write $K_0^T(X)$ for the Grothendieck group of the category of $T$-equivariant coherent sheaves on $X$. If $X$ is proper over $\mathbb{F}$ then the Euler characteristic $[\mathcal{F}] \mapsto \sum_{i \geq 0} (-1)^i [H^i(X,\mathcal{F})]$ defines a homomorphism
		$$
		\chi: K_0^T(X) \rightarrow R(T)
		$$ 
		where $R(T) = \mathbb{Z}[X^*(T)]$ denotes the Grothendieck group of algebraic $T$-representations (in which multiplication is given by the tensor product). For $\alpha^\vee \in X^*(T)$ we write $e(\alpha^\vee)$ for its class in $R(T)$ and for $V \in R(T)$ we write $V_{\alpha^\vee} \in \mathbb{Z}$ for the multiplicity of $e(\alpha^\vee)$ in $V$.
	\end{sub}
	
	\begin{definition}
		Suppose that $d \geq 0$.
		\begin{itemize}
			\item Define $K_0^T(X)_{\leq d}$ as the Grothendieck group of the category of $T$-equivariant coherent sheaves on $X$ with support of dimension $\leq d$.
			\item Let $V = (V_n)_{n \geq 0}$ be a sequence of elements in $R(T)$. We say $V$ is polynomial of degree $\leq d$ if there exists a polynomial $P(x) \in \mathbb{Q}[x]$ of degree $\leq d$ so that 
			$$
			\sum_{\alpha^\vee \in X^*(T)} |V_{n,\alpha^\vee}| \leq P(n)
			$$
			for all $n\geq 0$ (here $|\cdot|$ denotes the usual absolute value).
		\end{itemize}
		Notice that if each $V_n$ is effective (i.e. is the class of a $T$-representation $V_n$ in $R(T)$) then $V$ is polynomial if and only if the dimensions of $V_n$ are bounded by the value at $n$ of a polynomial. However the above definition also allows us to extend this notion to elements which are not necessarily effective. 
	\end{definition}
	\begin{remark}
		Note that there are homomorphisms $K_0^T(X)_{\leq d} \rightarrow K_0^T(X)$ which are not typically injective.
	\end{remark}
	\begin{lemma}\label{lemma-additive}
		If $V = (V_{n})_{n \geq 0}$ and $W = (W_{n})_{n \geq 0}$ are respectively polynomial of degree $\leq d_1$ and $\leq d_2$  then $(V_n+ W_n)_{n \geq 0}$ is also polynomial of degree $\leq \operatorname{max}\lbrace d_1,d_2 \rbrace$ and $(V_nW_n)_{n \geq 0}$ is polynomial of degree $\leq d_1+d_2$.
	\end{lemma}
	\begin{proof}
		The first assertion follows since $\sum_{\alpha^\vee} |V_{n,\alpha^\vee} + W_{n,\alpha^\vee}| \leq \sum_{\alpha^\vee} |V_{n,\alpha^\vee}| + \sum_{\alpha^\vee} |W_{n,\alpha^\vee}|$. For the second, use $\sum_{\alpha^\vee} |\sum_{\beta^\vee+\gamma^\vee = \alpha^\vee} V_{n,\beta^\vee} W_{n,\gamma^\vee}| \leq \sum_{\beta^\vee,\gamma^\vee} | V_{n,\beta^\vee} |  | W_{n,\gamma^\vee}| = \left( \sum_{\beta^\vee} | V_{n,\beta^\vee} |  \right) \left(\sum_{\gamma^\vee} | W_{n,\gamma^\vee} | \right) $.
		
	\end{proof}
	\begin{lemma}\label{lem-cohomology}
		Suppose that $X$ is a proper $\mathbb{F}$-scheme of finite type equipped with a $T$-equivariant ample line bundle $\mathcal{L}$ and $\mathcal{F} \in K_0^T(X)$ is contained in the image of $\operatorname{im}K_0^T(X)_{\leq d}$. Then
		$$
		\chi(\mathcal{F} \otimes [\mathcal{L}^{\otimes n}] ) \in R(T)
		$$
		is polynomial of degree $ \leq d$.
	\end{lemma}
	\begin{proof}
		Using Lemma~\ref{lemma-additive} we can assume that $\mathcal{F}$ is the class of a $T$-equivariant sheaf $\mathcal{G}$ on $X$ with support of dimension $\leq d$. Since $\mathcal{L}$ is ample $\chi([\mathcal{G} \otimes \mathcal{L}^{\otimes n}])$ equals the class of $H^0(X,\mathcal{G} \otimes \mathcal{L}^{\otimes n})$ in $R(T)$ for sufficiently large $n$. Since the dimension of $H^0(X,\mathcal{G} \otimes \mathcal{L}^{\otimes n})$ is for sufficiently large $n$ the value at $n$ of a polynomial $P(x)$ of degree equal the support of $\mathcal{G}$ it follows that for all $n \geq 0$
		$$
		\sum_{\alpha^\vee} |\chi(\mathcal{G} \otimes \mathcal{L}^{\otimes n})_{\alpha^\vee}| = \sum_{\alpha^\vee} \chi(\mathcal{G} \otimes \mathcal{L}^{\otimes n})_{\alpha^\vee} \leq P(n) + C
		$$ 
		for a constant $C>>0$.
	\end{proof}

	\begin{proposition}\label{prop-cycletorep}
		Let $\mathcal{F}$ be a $T$-equivariant coherent sheaf on $X$ with support of dimension $\leq d$ and let
		$$
		[\mathcal{F}] = \sum_Z n_Z [\mathcal{O}_Z] \in Z_d(X)
		$$
		be the associated $d$-dimensional cycle.
		Assume that $X$ is equipped with an ample $T$-equivariant line bundle $\mathcal{L}$ and that $n_Z >0$ implies $Z$ is $T$-stable. Then there are $q_Z \in \mathbb{Z}_{\geq 0}$ and $\theta_{Z,i}^\vee \in X^*(T)$ so that 
		$$
		\chi(\mathcal{F} \otimes \mathcal{L}^{\otimes n}) - \sum_Z \left(  \sum_{i=1}^{n_Z} e(\theta_{Z,i}^\vee) \chi(\mathcal{L}^{\otimes k - q_Z}|_{Z}) \right) \in R(T)
		$$
		is polynomial of degree $< d$.
	\end{proposition}
	\begin{proof}
		We induct on the number of $n_Z >0$. If this is zero then the class of $\mathcal{F}$ has support of dimension $<d$ and so its class in $K_0^T(X)$ is contained in $\operatorname{im}K_0^T(X)_{\leq d-1}$. Lemma~\ref{lem-cohomology} therefore implies $\chi(\mathcal{F} \otimes \mathcal{L}^{\otimes n})$ is polynomial of degree $<d$, and the proposition holds. Otherwise, write $\mathcal{I}_Z$ for the ideal sheaves corresponding to those $Z$ with $n_Z >0$. By assumption each such $Z$ is $T$-stable and so $\mathcal{I}_Z^N$ is a $T$-equivariant coherent sheaf for any $N\geq 1$. For $N$ sufficiently large the support of $\mathcal{I}_Z^N\mathcal{F}$ does not contain $Z$ (see \cite[0Y19]{stacks-project}) and so the inductive hypothesis holds for $\mathcal{I}_Z^N\mathcal{F}$. By applying Lemma~\ref{lem-cohomology} to the identity
		$$
		[\mathcal{F}] =  [\mathcal{I}_Z^N\mathcal{F}] + [\mathcal{F}/\mathcal{I}_Z^N\mathcal{F}] 
		$$
		in $K_0^T(X)$ we see that the proposition will hold if it holds for $\mathcal{F}/\mathcal{I}_Z^N\mathcal{F}$. 
		
		Since $\mathcal{F}/I_Z^N\mathcal{F}$ has support contained in $Z$ we are reduced to proving the proposition when the support of $\mathcal{F}$ is a single irreducible component. Thus we can assume $Z= X$ and write $[\mathcal{F}] = n[\mathcal{O}_X]$ in $Z_d(X)$. Since $\mathcal{L}$ is ample there is an integer $q \in \mathbb{Z}_{\geq 0}$ so that $\mathcal{F} \otimes \mathcal{L}^{\otimes q}$ is generated by global sections. This gives a $T$-equivariant surjection
		$$V \otimes \mathcal{L}^{\otimes -q} \rightarrow \mathcal{F}
		$$
		with $V = H^0(X,\mathcal{F} \otimes \mathcal{L}^{\otimes q})$.  Since $T$ is abelian we can choose a $T$-equivariant filtration 
		$$  \ldots\subset V_{j+1} \subset  V_j \subset V_{j-1} \subset \ldots V_0 =V
		$$
		with each graded piece of dimension one. 
		\begin{claim}
			There exists $j_1,\ldots,j_n \in \mathbb{Z}_{\geq 0}$ so that
			$$
			[\mathcal{F}]-  \sum_{i=1}^n [V_{j_i}/V_{j_i+1} \otimes \mathcal{L}^{-q}] \in \operatorname{im}K_0^T(X)_{\leq d-1}
			$$
		\end{claim}
		\begin{proof}[Proof of Claim]
			Set $\mathcal{F}_j$ equal the image in $\mathcal{F}$ of $V_j \otimes \mathcal{L}^{-q}$. Then $[\mathcal{F}] = \sum_{j \in \mathbb{Z}_{\geq 0}} [\mathcal{F}_j/\mathcal{F}_{j+1}]$ in $K_0^T(X)$. For each $j$ we also have a $T$-equivariant exact sequence 
			$$
			0 \rightarrow \mathcal{G}_j \rightarrow V_j/V_{j+1} \otimes \mathcal{L}^{-q} \rightarrow \mathcal{F}_j/\mathcal{F}_{j+1} \rightarrow 0
			$$
			of coherent sheaves.  Since $V_j/V_{j+1} \otimes \mathcal{L}^{-q}$ is locally free of rank one it follows that exactly one of $\mathcal{G}_j$ and $\mathcal{F}_j/\mathcal{F}_{j+1}$ has support of dimension $<d$ and exactly one has support equal to $X$. Thus 
			$$
			\begin{aligned}
				~[\mathcal{F}] - \sum_{\operatorname{supp}\mathcal{F}_j/\mathcal{F}_{j+1} =X} & [V_j/V_{j+1}\otimes \mathcal{L}^{\otimes -q}]  \\
				&= - \sum_{\operatorname{supp}\mathcal{F}_j/\mathcal{F}_{j+1} =X} [\mathcal{G}_j] + \sum_{\operatorname{supp}\mathcal{F}_j/\mathcal{F}_{j+1} \neq X} [\mathcal{F}_j/\mathcal{F}_{j+1}] \in \operatorname{im}K_0^T(X)_{\leq d-1}
			\end{aligned}
			$$
			To finish the proof of the claim we just need to check that $\operatorname{supp}\mathcal{F}_j/\mathcal{F}_{j+1} =X$ exactly $n$ times. This follows because in $Z_d(X)$ we have $ [V_j/V_{j+1}\otimes \mathcal{L}^{\otimes -q}]  = [X]$ and so $[\mathcal{F}] =  \sum_{\operatorname{supp}\mathcal{F}_j/\mathcal{F}_{j+1} =X}  [V_j/V_{j+1}\otimes \mathcal{L}^{\otimes -q}]  = \sum_{\operatorname{supp}\mathcal{F}_j/\mathcal{F}_{j+1} = X} [X] = n[X]$.
		\end{proof}
		
		Applying Lemma~\ref{lem-cohomology} to the identity in the claim gives the proposition because if $\theta_i^\vee \in X^*(T)$ is the character through which $T$ acts on $V_{j_i}/V_{j_i+1}$ then $\chi(V_j/V_{j+1}\otimes \mathcal{L}^{\otimes k-q}) = e(\theta_i^\vee)\chi(\mathcal{L}^{\otimes (k-q)})$.
	\end{proof}
	
	\section{Determinant line bundles}\label{sec-detlinebundle}
	
	\begin{sub}
		In order to apply Proposition~\ref{prop-cycletorep} to the identity of cycles established in Proposition~\ref{prop-firstcycle} we need to choose an equivariant line bundle on $\operatorname{Gr}_G$. To do this we consider the morphism
		$$
		\operatorname{Ad}:	\operatorname{Gr}_G \rightarrow \operatorname{Gr}_{\operatorname{GL}(\mathfrak{g})}
		$$
		induced by the adjoint representation of $G$. Then $\operatorname{Gr}_{\operatorname{GL}(\mathfrak{g})}$ is equipped with the ``determinantal'' line bundle $\mathcal{L}_{\operatorname{det}}$, defined by the property that its pull-back to $\operatorname{Spec}A$ along a morphism corresponding to $(\mathcal{E},\iota) \in \operatorname{Gr}_{\operatorname{GL}(\mathfrak{g})}(A)$ is given by the $A$-module
		$$
		\operatorname{det}_A(u^{-N} \mathfrak{g} \otimes_{\mathcal{O}} A[u]/ \iota(\mathcal{E})) \otimes_A \operatorname{det}_A(u^{-N}\mathfrak{g} \otimes_{\mathcal{O}} A[u]/ \mathfrak{g} \otimes_{\mathcal{O}} A[u])^{-1}
		$$
		for $N$ sufficiently large that $\iota(\mathcal{E}) \subset u^{-N} \mathfrak{g} \otimes_{\mathcal{O}} A[u]$. Note this is independent of the choice of $N$. Then $\mathcal{L}_{\operatorname{det}}$ is $\operatorname{GL}(\mathfrak{g})$-equivariant and is ample in the sense that its restriction to any closed subscheme in $\operatorname{Gr}_{\operatorname{GL}(\mathfrak{g})}$ is ample. Therefore
		$$
		\mathcal{L}_{\operatorname{ad}} := \operatorname{Ad}^* \mathcal{L}_{\operatorname{det}}
		$$
		is $G$-equivariant and also ample.
	\end{sub}
	\begin{lemma}\label{lem-linebundle}
		For $\lambda \in X_*(T)$ the group $T$ acts on the fibre of $\mathcal{L}_{\operatorname{ad}}$ over $\mathcal{E}_{\lambda,i}$ via the image of $\lambda$ under the homomorphism 
		$$
		p: X_*(T) \rightarrow X^*(T)
		$$
		given by $\lambda \mapsto \sum_{\alpha^\vee \in R^\vee} \langle \alpha^\vee,\lambda \rangle \alpha^\vee$.
	\end{lemma}
	\begin{proof}
		This fibre is the rank one $\mathcal{O}$-module
		\begin{equation}\label{eq-fibre}
			\underbrace{\operatorname{det}_{\mathcal{O}}\left(u^{-N} \mathfrak{g}[u] / \operatorname{Ad}(u-\pi_i)^\lambda \mathfrak{g}[u] \right)}_{\Lambda_1} \otimes \underbrace{\operatorname{det}_{\mathcal{O}}\left(u^{-N} \mathfrak{g}[u] / \mathfrak{g}[u] \right)^{-1}}_{\Lambda_2}
		\end{equation}
		for sufficiently large $N$. As an $\mathcal{O}$-module we have 
		$$
		u^{-N} \mathfrak{g}[u] / \operatorname{Ad}(u-\pi_i)^\lambda \mathfrak{g}[u] \cong \bigoplus_{\alpha^\vee \in R^\vee} \bigoplus_{n=-N}^{\langle \alpha^\vee,\lambda \rangle} (u-\pi_i)^n \mathfrak{g}_{\alpha^\vee}
		$$
		and $t \in T(\mathcal{O})$ acts on $(u-\pi_i)^n \mathfrak{g}_{\alpha^\vee}$ by $\alpha^\vee(t)$. Therefore, $t$ acts on	$\Lambda_1$ by $\prod_{\alpha^\vee \in R} \alpha^\vee(t)^{\langle \alpha^\vee,\lambda \rangle + N}$. Similarly $t$ acts on $\Lambda_2$ by $\prod_{\alpha^\vee \in R} \alpha^\vee(t)^{N}$. We conclude that $t$ acts on \eqref{eq-fibre} by
		$$
		\prod_{\alpha^\vee \in R} \alpha^\vee(t)^{\langle \alpha^\vee,\lambda \rangle} = p(\lambda)(t)
		$$
		as claimed.
	\end{proof}
	\begin{proposition}\label{prop-kunneth}
		If $\mu = (\mu_1,\ldots,\mu_e)$ with $\mu_i \in X_*(T)$ dominant then, after interpreting $H^0(M_\mu \otimes_{\mathcal{O}} \mathbb{F},\mathcal{L}_{\operatorname{ad}}^{\otimes n})$ and $\operatorname{Ind}_{P_{\mu_i}}^G(p(n\mu_i))$ as $T$-representations by restricting the $G$-action, one has
		$$
		[H^0(M_\mu \otimes_{\mathcal{O}} \mathbb{F},\mathcal{L}_{\operatorname{ad}}^{\otimes n})] =  [\bigotimes_{i=1}^e \operatorname{Ind}_{P_{\mu_i}}^G(p(n\mu_i))] \in R(T)
		$$ 
		for sufficiently large $n\geq 0$.
	\end{proposition}
	\begin{proof}
		
		Since $\mathcal{L}_{\operatorname{ad}}$ is an ample line bundle on the flat $\mathcal{O}$-scheme $M_\mu$ it follows that 
		$$
		H^0(M_\mu \otimes_{\mathcal{O}} \mathbb{F},\mathcal{L}_{\operatorname{ad}}^{\otimes n}) = H^0(M_\mu,\mathcal{L}_{\operatorname{ad}}^{\otimes n}) \otimes_{\mathcal{O}} \mathbb{F}
		$$
		for sufficiently large $n$. Therefore $[H^0(M_\mu \otimes_{\mathcal{O}} \mathbb{F},\mathcal{L}_{\operatorname{ad}}^{\otimes n}) ]$ is equal to the image of $[H^0(M_\mu \otimes_{\mathcal{O}} E,\mathcal{L}_{\operatorname{ad}}^{\otimes n}) ]$ under the specialisation map from \cite[10.9]{Janbook}. Since this map sends the class of $\bigotimes_{i=1}^e \operatorname{Ind}_{P_{\mu_i}}^G(p(n\mu_i))$ viewed as a representation on an $E$-vector space onto $\bigotimes_{i=1}^e \operatorname{Ind}_{P_{\mu_i}}^G(p(n\mu_i))$ viewed as a representation on an $\mathbb{F}$-vector space the proposition will follow if we can show
		$$
		[H^0(M_\mu \otimes_{\mathcal{O}} E,\mathcal{L}_{\operatorname{ad}}^{\otimes n})] =  [\bigotimes_{i=1}^e \operatorname{Ind}_{P_{\mu_i}}^G(p(n\mu_i))] \in R(T)
		$$
		Under the isomorphism $\operatorname{Gr}_G \otimes_{\mathcal{O}} E \cong  \left( \operatorname{Gr}_{G,1} \times_{\mathcal{O}} \ldots \times_{\mathcal{O}} \operatorname{Gr}_{G,e} \right) \otimes_{\mathcal{O}} E$ we have $\mathcal{L}_{\operatorname{ad}} = \bigotimes_{i=1}^e p_i^* \mathcal{L}_{\operatorname{ad},i}$ where $\mathcal{L}_{\operatorname{ad},i}$ is the restriction of $\mathcal{L}_{\operatorname{ad}}$ to $\operatorname{Gr}_{G,i}$ and $p_i$ is the $i$-th projection. Therefore
		$$
		\mathcal{L}_{\operatorname{ad}}|_{M_\mu \otimes_{\mathcal{O}} E} = \bigotimes_{i=1}^e p_i^*(\mathcal{L}_{\operatorname{ad,i}}|_{G/P_{\mu_i} \times_{\mathcal{O}} E})
		$$
		and so the Kunneth formula \cite[0BED]{stacks-project} identifies
		$$
		H^0(M_\mu \otimes_{\mathcal{O}} E,\mathcal{L}_{\operatorname{ad}}^{\otimes n}) = \bigotimes_{i=1}^e H^0(G/P_{\mu_i} \otimes_{\mathcal{O}} E, \mathcal{L}_{\operatorname{ad},i}^{\otimes n})
		$$
		as $G$-representations. To finish the proof we just have to show 
		\begin{equation}\label{eq-desiredisom}
			H^0(G/P_{\mu_i}, \mathcal{L}_{\operatorname{ad},i}^{\otimes n}) \cong \operatorname{Ind}_{P_{\mu_i}}^G(p(n\mu_i))
		\end{equation} 
		For this recall (see for example \cite[5.12]{Janbook}) that the  global sections of any $G$-equivariant line bundle on $G/P_{\mu_i}$ are $G$-equivariantly isomorphic to $\operatorname{Ind}_{P_{\mu_i}}^G(\eta)$ where $\eta \in X^*(T)$ is the character through which $T$ acts on the fibre over $1 \in G/P_{\mu_i}$. Since $1$ is mapped onto $\mathcal{E}_{\mu_i,i}$ under the closed immersion $G/P_{\mu_i} \rightarrow \operatorname{Gr}_{G,i}$ we deduce \eqref{eq-desiredisom} from Lemma~\ref{lem-linebundle}.
	\end{proof}
	
	\section{Main theorem}\label{sec-proofofcyclethm}
	
	For any dominant $\lambda^\vee \in X^*(T)$ write
	$$
	W(\lambda^\vee) = \operatorname{Ind}_{B^-}^{G}(\lambda^\vee)
	$$
	for $B^-$ the Borel opposite to $B$. Likewise, for $\lambda \in X_*(T)$ we make sense of $W(\lambda)$, now as a representation of $\widehat{G}$. The following is an alternative formulation of Theorem~\ref{thmA-p1}.
	
	\begin{theorem}\label{thm-cycles}
		Assume that $G$ admits a twisting element $\rho \in X_*(T)$ and let $\mu = (\mu_1,\ldots,\mu_e)$ with each $\mu_i \in X_*(T)$ strictly dominant. If $\operatorname{char}\mathbb{F}>0$ assume also that
		$$
		\sum_{i=1}^e \langle \alpha^\vee,\mu_i \rangle \leq \operatorname{char}\mathbb{F} + e-1
		$$
		for all positive roots $\alpha^\vee$. Then
		$$
		[M_\mu \otimes_{\mathcal{O}} \mathbb{F}] = \sum m_\lambda [M_{(\lambda+\rho,\ldots,\rho)}]
		$$
		as $e|R^+|$-dimensional cycles for $m_\lambda \in \mathbb{Z}_{\geq 0}$ determined by the identity 
		$$
		[	\bigotimes_{i=1}^e W(\mu_i-\rho)] = \sum_{\lambda}m_\lambda [ W(\lambda)]
		$$
		in the Grothendieck group of $\widehat{G}$-representations.
	\end{theorem}
	\begin{proof}
		We give the proof here using two representation-theoretic propositions from the next section (Propositions~\ref{prop-weylcharform} and~\ref{prop-kostant}). Proposition~\ref{prop-firstcycle} and Theorem~\ref{thm-irred} imply
		$$
		[M_\mu \otimes_{\mathcal{O}}\mathbb{F}] = \sum n_{\lambda} [M_{\widetilde{\lambda}}\otimes_{\mathcal{O}} \mathbb{F}]
		$$
		where the sum suns over dominant $\lambda \leq \mu_1+\ldots+\mu_e - e\rho$, $\widetilde{\lambda} = (\lambda+\rho,\rho,\ldots,\rho)$,  and $n_{\lambda} \in \mathbb{Z}_{\geq 0}$. We have to show that $n_\lambda =m_\lambda$. Applying Proposition~\ref{prop-cycletorep} to this identity with $\mathcal{L}$ equal to $\mathcal{L}_{\operatorname{ad}}$ gives $\theta_{\lambda,i}^\vee \in X^*(T)$ so that
		$$
		\chi(\mathcal{L}_{\operatorname{ad}}^{\otimes n}|_{M_\mu \otimes_{\mathcal{O}} \mathbb{F}}) - \sum_{\lambda} \sum_{i=1}^{n_\lambda} e(\theta_{\lambda,i}^\vee)  \chi(\mathcal{L}_{\operatorname{ad}}^{\otimes n}|_{M_{\widetilde{\lambda}} \otimes_{\mathcal{O}} \mathbb{F}})
		$$
		is polynomial of degree $< e|R_+^\vee|$. Since each $\mu_i$ is strictly dominant each $P_{\mu_i}$ equals the opposite Borel $B^-$ and so
		$$
		W(p(n\mu_i)) = \operatorname{Ind}_{P_{\mu_i}}^G(p(n\mu_i))
		$$
		Therefore Proposition~\ref{prop-kunneth} gives that
		\begin{equation*}\label{eq-polynomial<di}
			\prod_{i=1}^e W(p(n\mu_i)) - \sum_{\lambda} \sum_{i=1}^{n_\lambda} e(\theta_{\lambda,i}^\vee) W(p(n(\lambda+\rho))) W(p(n\rho))^{e-1}
		\end{equation*}
		is polynomial of degree $< e|R_+^\vee|$. In the next section we prove (see Proposition~\ref{prop-weylcharform}) that
		$$
		\prod_{i=1}^e W(p(n\mu_i)) - \sum_{\lambda} m_\lambda  W(p(n(\lambda+\rho))) W(p(n\rho))^{e-1}
		$$
		is polynomial of degree $< e|R_+^\vee|$, and so considering the difference gives that
		$$
		\sum_\lambda \left( m_\lambda - \sum_{i=1}^{n_\lambda} e(\theta_{\lambda,i}^\vee) \right)  W(p(n(\lambda+\rho))) W(p(n\rho))^{e-1}
		$$
		is also polynomial of degree $< e|R_+^\vee|$. In the next section we also prove (see Proposition~\ref{prop-kostant}) that if $X_{\lambda} \in R(T)$ are such that
		$$
		\sum_{\lambda} X_{\lambda}   W(p(n(\lambda+\rho))) W(p(n\rho))^{e-1}
		$$
		is polynomial of degree $<e|R_+^\vee|$ then $X_\lambda =0$ for each $\lambda$. Therefore
		$$
		m_\lambda - \sum_{i=1}^{n_\lambda} e(\theta_{\lambda,i}^\vee) =0 
		$$
		for each $\lambda$. This implies  that each $\theta^\vee_{\lambda,i} =1$ and that $n_\lambda = m_\lambda$ for each $\lambda$. The latter assertion, in particular, proves the theorem.
	\end{proof}

	\section{Some representation theory}\label{sec-reptheory}
	
	It remains to prove Proposition~\ref{prop-weylcharform} and Proposition~\ref{prop-kostant} which were used in the proof of Theorem~\ref{thm-irred}. For this set $\rho^\vee = \frac{1}{2}\sum_{\alpha^\vee \in R_+^\vee} \alpha^\vee$. Then the Weyl character formula asserts that for any dominant $\lambda^\vee \in X^*(T)$ 
	$$
	[W(\lambda^\vee)] = \frac{A(\lambda^\vee+\rho^\vee)}{A(\rho^\vee)}
	$$
	where $A(\lambda^\vee) := \sum_{w\ \in W} (-1)^{l(w)} e(w(\lambda^\vee))$ and this identity is occurring inside the ring $\mathbb{Z}[\frac{1}{2}X^*(T)]$. See, for example, \cite[5.10]{Janbook}.
	
	\begin{lemma}\label{lemma-qroot}
		If $\mu \in X^*(T)$ is strictly dominant then
		$$
		W(n\mu^\vee) - e(\rho^\vee)\frac{A(n\mu^\vee)}{A(\rho^\vee)} \in R(T)
		$$
		is a sequence of effective elements (i.e. $\mathbb{Z}_{\geq 0}$-linear combinations of the $e(\alpha^\vee)$) which is polynomial in $n$ of degree $<|R^+|$.
	\end{lemma}
	\begin{proof}	
		We first reduce to the case where $X^*(T)$ contains a twisting element $\rho^\vee_0$, in the sense of Definition~\ref{def-twistingelement}. The construction from \cite[\S5.3]{BG14} produces a central extension $1 \rightarrow \mathbb{G}_m \rightarrow \widetilde{G} \rightarrow G \rightarrow 1$ such that if $\widetilde{T} \subset \widetilde{G}$ is the preimage of $T$ then $X^*(\widetilde{T})$ contains such a twisting element. Being a central extension, the Weyl group of $\widetilde{G}$ relative to $\widetilde{T}$ equals $W$. Therefore, the inclusion 
		$$
		X^*(T) \rightarrow X^*(\widetilde{T})
		$$
		maps $A(\lambda^\vee)$ onto $A(\widetilde{\lambda}^\vee)$ for $\widetilde{\lambda}^\vee$ the character of $\widetilde{T}$ induced by $\lambda^\vee$. As a result the lemma holds for $G$ if it holds for $\widetilde{G}$
		
		We can therefore assume there exists a twisting element $\rho_0^\vee \in X^*(T)$. Then $\rho_0^\vee - \rho^\vee$ is $W$-invariant and so the Weyl character formula implies $[W(\mu^\vee)] = \frac{A((\mu^\vee+\rho_0^\vee))}{A(\rho_0^\vee)}$. Now, for any $\lambda^\vee \in X^*(T)$ write $\mathcal{L}(\lambda^\vee)$ for the $G$-equivariant line bundle on $G/B^-$ on which the action of $T$ on the fibre over the identity is given by $\lambda^\vee$. Then $\mathcal{L}(\rho^\vee_0)$ admits a unique global section on which $T$ acts by $\rho^{\vee}_0$, and this induces a $T$-equivariant injection
		$$
		\mathcal{O}_{G/B^-} \otimes \rho^\vee_0 \hookrightarrow \mathcal{L}_{\rho^\vee_0}
		$$
		Tensoring with $\mathcal{L}_{\mu^\vee - \rho^\vee_0}$ produces a $T$-equivariant injection $\mathcal{L}_{\mu^\vee - \rho^\vee_0} \otimes \rho^\vee_0 \hookrightarrow \mathcal{L}_{\mu^\vee}$ and taking global sections yields a $T$-equivariant injection of $W(\mu^\vee-\rho^\vee_0) \otimes\rho^\vee_0$ into $W(\mu^\vee)$. In particular,
		$$
		[W(n\mu^\vee) ]- e(\rho^\vee_0)[W(n\mu^\vee - \rho^\vee_0)] = [W(n\mu^\vee)] - e(\rho^\vee_0) \frac{A(n\mu^\vee)}{A(\rho_0^\vee)} = [W(n\mu^\vee)] - e(\rho^\vee )\frac{A(n\mu^\vee)}{A(\rho^\vee)}
		$$
		is an effective element of $R(T)$ for each $n \geq 0$. Since it is effective we can show the sequence of elements is polynomial of degree $<|R^+|$ by showing that the difference between the dimensions of $W(n\mu^\vee)$ and $W(n\mu^\vee - \rho^\vee_0)$ is a polynomial in $n$ of degree $<|R^+|$. But this follow from the Weyl dimension  formula $\operatorname{dim} W(\lambda^\vee) = \prod_{\alpha \in R_+} \frac{\langle \lambda^\vee+\rho^\vee_0,\alpha\rangle}{\langle \rho^\vee_0,\alpha \rangle}$ since it shows both $W(n\mu^\vee)$ and $W(n\mu^\vee - \rho^\vee_0)$ have dimension the value at $n$ of a degree $|R^+|$ polynomial with leading term $n^{|R^+|}\prod_{\alpha \in R_+} \frac{\langle \mu^\vee,\alpha \rangle }{\langle \rho^\vee_0,\alpha \rangle}$.
	\end{proof}
	\begin{proposition}\label{prop-weylcharform} Suppose that $X_*(T)$ contains a twisting element $\rho$ and consider strictly dominant $\mu_1,\ldots,\mu_e \in X_*(T)$. If
		$$
		[	\bigotimes_{i=1}^e W(\mu_i-\rho)] = \sum_{\lambda} m_\lambda[ W(\lambda)]
		$$
		in the Grothendieck group of $\widehat{G}$-representations then 
		$$
		\prod_{i=1}^e W(p(n\mu_i)) - \sum_{\lambda} m_\lambda W(p(n(\lambda+\rho))) W(p(n\rho))^{e-1}
		$$
		is polynomial of degree $<e|R_*^\vee|$.
	\end{proposition}
	\begin{proof}
		The Weyl character formula (applied to $\widehat{G}$) yields the identity 
		$$
		\prod_{i=1}^e \frac{A(\mu_i)}{A(\rho)} = \sum_{\lambda} m_\lambda \frac{A(\lambda+\rho)}{A(\rho)}
		$$
		in $R(\widehat{T})$. Multiplying by $A(\rho)^{e}$ gives 
		$$
		\prod_{i=1}^e A(\mu_i) = \sum_{\lambda} m_\lambda A(\lambda+\rho) A(\rho)^{e-1}
		$$
		The endomorphism of $R(\widehat{T}) = \mathbb{Z}[X_*(T)]$ induced by multiplication by $n$ on $X^*(T)$ is $W$-equivariant and so commutes with the formation of $A(\lambda)$. Applying this endomorphism to the previous identity gives
		$$
		\prod_{i=1}^e A(n\mu_i) = \sum_{\lambda} m_\lambda A(n(\lambda+\rho)) A(n\rho)^{e-1}
		$$
		The homomorphism $p: X_*(T) \rightarrow X^*(T)$ induces a homomorphism $R(T^\vee) \rightarrow R(T)$ which is again $W$-equivariant and so also commutes with the formation of $A(\lambda)$. Therefore, applying this homomorphism and multiplying by $\left(\frac{e(\rho^\vee)}{A(\rho^\vee)}\right)^e$ gives
		\begin{equation}\label{eq-finalmult}
			\prod_{i=1}^e\frac{ e(\rho^\vee)A(np(\mu_i)) }{A(\rho^\vee)}= \sum_{\lambda} m_\lambda \frac{e(\rho^\vee)A(np(\lambda+\rho))}{A(\rho^\vee)} \frac{e(\rho^\vee)A(np(\rho))}{A(\rho^\vee)}^{e-1}
		\end{equation}
		in $R(T)$. Write
		$$
		\begin{aligned}
			\prod_{i=1}^e\frac{ e(\rho^\vee)A(np(\mu_i)) }{A(\rho^\vee)} &= \prod_{i=1}^e \left( W(np(\mu_i)) - \left(W(np(\mu_i)) - \frac{ e(\rho^\vee)A(np(\mu_i)) }{A(\rho^\vee)} \right) \right) \\
			&= \prod_{i=1}^e \left( W(np(\mu_i)) \right) + C_{\mu,n}
		\end{aligned}
		$$
		for $C_{\mu,n} \in R(T)$.
		Lemma~\ref{lemma-qroot} ensures that $\left(W(np(\mu_i)) - \frac{ e(\rho^\vee)A(np(\mu_i)) }{A(\rho^\vee)} \right)$ is polynomial in $n$ of degree $<|R^+|$ and so, since $W(np(\mu_i))$ has dimension polynomial in $n$ of degree $|R_+^\vee|$, it follows that $C_\mu = (C_{\mu,n})_{n \geq 0}$ is polynomial of degree $<e|R_+^\vee|$. Similarly, each
		$$
		\frac{e(\rho^\vee)A(np(\lambda+\rho))}{A(\rho^\vee)} \frac{e(\rho^\vee)A(np(\rho))}{A(\rho^\vee)}^{e-1} = W(np(\lambda+\rho)) W(np(\rho))^{e-1} + C_{\lambda,n}
		$$
		with $C_\lambda = (C_{\lambda,n})_{n \geq 0}$ polynomial of degree $< e|R_+^\vee|$. Combining these observations with \eqref{eq-finalmult} gives that
		$$
		\prod_{i=1}^e W(p(n\mu_i)) - \sum_{\lambda} m_\lambda W(p(n(\lambda+\rho))) W(p(n\rho))^{e-1}
		$$
		is polynomial of degree $<e|R_+^\vee|$ as desired.
	\end{proof}
	
	\begin{proposition}\label{prop-kostant}
		Fix a strictly dominant $\mu^\vee \in X^*(T)$ and consider a finite collection of non-zero $C_{\lambda^\vee} \in R(T)$, indexed by strictly dominant $\lambda^\vee \in X^*(T)$. Suppose that
		$$
		\sum_{\lambda^\vee} C_{\lambda^\vee} W(n\lambda^\vee) W(n\mu^\vee)^{e-1}
		$$
		is polynomial of degree $<e|R_+^\vee|$. Then $C_{\lambda^\vee} =0$ for all $\lambda^\vee$.
	\end{proposition}
	\begin{proof}
		Since $\mu^\vee$ is strictly dominant the dimension of $W(n\mu^\vee)$ is polynomial in $n$ of degree $|R^+|$. Therefore, we can assume $e=1$. If the proposition does not hold then we can choose $\lambda^\vee$ with $C_{\lambda^\vee} \neq 0$ and $C_{\lambda^\vee_0} =  0$ whenever $\lambda_0^\vee > \lambda^\vee$.
		
		\begin{observation}
			Fix integers $\Phi,\Psi >0$ and let $S^\vee$ denote the set of simple roots. Then there exists a degree $|R_+^\vee|$ polynomial $Q(x) \in \mathbb{Q}[x]$ with positive leading term so that for $n>>0$ one has
			$$
			\sum_{\eta^\vee} \operatorname{dim} W(n\lambda^\vee)_{\eta^\vee} \geq Q(n)
			$$
			where the sum runs over $\eta^\vee = n\lambda^\vee - \sum_{\alpha^\vee \in S^\vee} l_{\alpha^\vee} \alpha^\vee$ with $0 \leq l_{\alpha^\vee} <\frac{n}{\Psi}-\Phi$.
		\end{observation}
		\begin{proof}[Proof of Observation]
			The Kostant multiplicity formula \cite[\S24.2]{Hum78} asserts that 
			$$
			\operatorname{dim} W(n\lambda^\vee)_{\eta^\vee} = \sum_{w \in W} (-1)^{l(w)} P(w(n\lambda^\vee +\rho^\vee) - (\eta^\vee +\rho^\vee))
			$$ 
			where $P(\mu^\vee)$ denotes the number of ways in which $\mu^\vee \in X^*(T)$ can be expressed as a $\mathbb{Z}_{\geq 0}$-linear combination of $\alpha^\vee \in R_+^\vee$. We claim if $\eta^\vee = n\lambda^\vee - \sum_{\alpha^\vee \in S^\vee} l_{\alpha^\vee} \alpha^\vee$ for $0 \leq l_{\alpha^\vee} <\frac{n}{\Psi}-\Phi$ then $P(w(n\lambda^\vee +\rho^\vee) - (\eta^\vee +\rho^\vee)) = 0$ for $w \neq 1$. Since
			$$
			w(n\lambda^\vee +\rho^\vee) - (\eta^\vee +\rho^\vee) = w(n\lambda^\vee+\rho^\vee) - (n\lambda^\vee + \rho^\vee) + \sum_{\alpha^\vee \in S^\vee} l_{\alpha^\vee} \alpha^\vee
			$$
			the claim follows if, when $w(n\lambda^\vee+\rho^\vee) - (n\lambda^\vee + \rho^\vee)$ is expressed as a $\mathbb{Z}$-linear combination of $\alpha^\vee \in S^\vee$, at least one coefficient is $\leq -n$. This is the case because both $\lambda^\vee$ and $\rho^\vee$ are strictly dominant and so $w(\lambda^\vee) - \lambda^\vee$ and $w(\rho^\vee) - \rho^\vee$ are both $< 0$ for $w \neq 1$ \cite[13.2.A]{Hum78}.
			Therefore, the observation is reduced to producing a polynomial lower bound on
			\begin{equation}\label{eq-Kostantpartionbound}
				\sum_{0 \leq l_{\alpha^\vee} < \frac{n}{\Psi}-\Phi} P(\sum_{\alpha^\vee \in S^\vee} l_{\alpha^\vee} \alpha^\vee)	
			\end{equation}
			of the correct degree. To do this we first claim that
			$$
			P(\sum_{\alpha^\vee \in S^\vee} l_{\alpha^\vee} \alpha^\vee) \geq \left( \frac{1}{K(|R_+^\vee| - |S^\vee|)}\operatorname{min} \lbrace l_{\alpha^\vee} \rbrace\right)^{|R_+^\vee| - |S^\vee|}
			$$
			where $K \geq 0$ is the largest coefficient appearing when any $\alpha^\vee \in R_+^\vee$ is expressed as a sum of simple roots. To see this note that if $0 \leq j_{\alpha^\vee} \leq \frac{1}{K(|R_+^\vee| - |S^\vee|)}\operatorname{min}\lbrace  l_{\alpha^\vee} \rbrace$ then $\sum_{\alpha^\vee \in R^\vee_+ \setminus S^\vee} j_{\alpha^\vee} \alpha^\vee$, when expressed as a sum of simple roots, will have $\alpha^\vee$-coefficient in $[0, \operatorname{min} \lbrace l_{\alpha^\vee}\rbrace]$ for each $\alpha^\vee \in S^\vee$. Consequently, for each choice of $j_\alpha$ there exists $i_{\alpha^\vee} \geq 0$ so that
			$$
			\sum_{\alpha^\vee \in S^\vee} l_{\alpha^\vee} \alpha^\vee = \sum_{\alpha^\vee \in R_+^\vee \setminus S^\vee} j_{\alpha^\vee} \alpha^\vee + \sum_{\alpha^\vee \in S^\vee} i_{\alpha^\vee} \alpha^\vee
			$$
			Thus, each tuple $(j_{\alpha^\vee})_{\alpha^\vee \in R^\vee_+ \setminus S^\vee}$ contributes one to $P(\sum_{\alpha^\vee \in S^\vee} l_{\alpha^\vee} \alpha^\vee)$ which gives the claimed lower bound. Therefore \eqref{eq-Kostantpartionbound} is 
			$$
			\geq \sum_{0 \leq l_\alpha < \frac{k}{\Psi}-\Phi} \left( \frac{1}{K(|R_+^\vee| - |S^\vee|)}\operatorname{min} \lbrace l_{\alpha^\vee} \rbrace\right)^{|R_+^\vee| - |S^\vee|}
			$$
			which is easily seen to be a polynomial in $n$ of degree $(|R_+^\vee| - |S^\vee|) + |S^\vee|$ with positive leading term. 
		\end{proof}
		We return to the proof of the proposition. Choose $\Psi,\Phi>0$ (we will be more specific later). If $e(\theta^\vee)$ appears in $C_{\lambda^\vee}$ with non-zero multiplicity then $e(\theta^\vee + n\lambda^\vee - \sum_{\alpha^\vee} l_{\alpha^\vee} \alpha^\vee)$ appears in $C_{\lambda^\vee} W(n\lambda^\vee)$ for any $n$ and any $0 \leq l_{\alpha^\vee} < \frac{n}{\Psi}- \Phi$. The observation above implies that for $n>>0$ at least one of these $e(\theta^\vee + n\lambda^\vee - \sum_{\alpha^\vee} l_{\alpha^\vee} \alpha^\vee)$ must cancel in $\sum_{\lambda_0^\vee} C_{\lambda_0^\vee} W(n\lambda_0^\vee)$, since otherwise we contradict that assumption that $\sum_{\lambda_0^\vee} C_{\lambda_0^\vee} W(n\lambda_0^\vee)$ is polynomial in $n$ of degree $<|R_+^\vee|$. Therefore, for each sufficiently large $n$ there exists $0 \leq l_{\alpha^\vee} < \frac{n}{\Psi} - \Phi$ and $e(\theta_0^\vee)$ appearing with non-zero multiplicity in $C_{\lambda_0^\vee}$ for $\lambda_0^\vee \neq \lambda^\vee$ so that $e(\theta^\vee + n\lambda^\vee - \sum_{\alpha^\vee} l_{\alpha^\vee} \alpha^\vee)$ appears in $e(\theta_0^\vee) W(n\lambda_0^\vee)$. This implies
		$$
		n\lambda^\vee - \sum_{\alpha^\vee} l_{\alpha^\vee} \alpha^\vee \leq \theta_0^\vee- \theta^\vee + n\lambda_0^\vee
		$$
		for large $n$. We are going to show this implies $\lambda_0 \geq \lambda$ which is a contradiction since $\lambda_0 \neq \lambda$ and $C_{\lambda_0} \neq 0$. To do this choose $\beta^\vee \in X^*(T)$ so that $\alpha^\vee \in S^\vee$ and the $\beta^\vee$ form a basis of $X^*(T) \otimes_{\mathbb{Z}} \mathbb{Q}$. If $\theta_0^\vee - \theta^\vee = \sum_{\alpha^\vee \in S^\vee} n_{\alpha^\vee} \alpha^\vee + \sum_{\beta^\vee} n_{\beta^\vee} \beta^\vee$ then
		$$
		n(\lambda_0^\vee-\lambda^\vee) = \sum_{\alpha^\vee \in S^\vee} m_{\alpha^\vee,n} \alpha^\vee +\sum_{\beta^\vee} n_{\beta^\vee} \beta^\vee
		$$
		with $m_{\alpha^\vee,n} \geq n_{\alpha^\vee} - l_{\alpha^\vee}$. Since the $n_{\beta^\vee}$ are bounded above independently of $n$ (as there are only finitely many possible $\theta_0^\vee$) it follows that each $n_{\beta^\vee} =0$. If $\Phi \geq -n_{\alpha^\vee}$ for every $\alpha^\vee$ then we also have $m_{\alpha^\vee,n} \geq - \frac{n}{\Psi}$. Therefore
		$$
		\lambda_0^\vee - \lambda^\vee = \sum_{\alpha^\vee \in S^\vee} m_{\alpha^\vee} \alpha^\vee
		$$ 
		with $m_{\alpha^\vee} \geq - \frac{1}{\Psi}$ for all $\Psi >0$. We conclude that each $m_{\alpha^\vee} \geq 0$ and so $\lambda_0^\vee \geq \lambda^\vee$. Since this contradicts the maximality of $\lambda^\vee$ we conclude $C_{\lambda^\vee}=0$ for every $\lambda^\vee$.
	\end{proof}
	
	\newpage
	\part{Cycle identities in moduli spaces of crystalline representations}

	\section{Notation}
	
	\begin{sub}
		For the second part of this paper we fix the following data:
		\begin{itemize}
			\item Let $K/\mathbb{Q}_p$ be a finite extension with residue field $k$ and ramification degree $e$ over $\mathbb{Q}_p$. Let $C$ denote a completed algebraic closure of $K$ with ring of integers $\mathcal{O}_C$ and fix a compatible system $\pi^{1/p^\infty}$ of $p$-th power roots of a uniformiser $\pi \in K$. 
			\item Fix another extension $E$ of $\mathbb{Q}_p$, with ring of integers $\mathcal{O}$ and residue field $\mathbb{F}$, and an embedding $k \hookrightarrow \mathbb{F}$ which we extend to an embedding $W(k) \rightarrow \mathcal{O}$.  Enlarging $E$ if necessary we assume that $E$ contains a Galois closure of $K$ so that $W(k) \hookrightarrow\mathcal{O}$ extends to $e$ distinct embeddings, which we index as $\kappa_1,\ldots,\kappa_e$.
			\item Let $G$ be a split reductive group over $W(k)$ with connected fibres and let $\operatorname{Res}_{W(k)/\mathbb{Z}_p}G$ be the group-scheme over $\mathbb{Z}_p$ defined by 
			$$
			\widetilde{G}(A) = G( W(k) \otimes_{\mathbb{Z}_p} A)
			$$
			for any $\mathcal{O}$-algebra $A$. Set $\widetilde{G}$ equal to the base-change of $\operatorname{Res}_{W(k)/\mathbb{Z}_p}G$ to $\mathcal{O}$. Since $W(k) \otimes_{\mathbb{Z}_p} \mathcal{O} = \prod_{i=1}^f W(k)\otimes_{W(k),\varphi^i} \mathcal{O}$, for $f = [k:\mathbb{F}_p]$ and $\varphi$ the lifting to $W(k)$ of the $p$-th power map on $k$, we can also write
			$$
			\widetilde{G}\cong \prod_{i=1}^f G \otimes_{W(k), \varphi^i} \mathcal{O}
			$$ 
			so that $\widetilde{G} $ is split reductive group over $\mathcal{O}$ with connected fibres. We apply the constructions from Definition~\ref{sub-Grdef} to $\widetilde{G}$ and with $\pi_i := \kappa_i(\pi)$ to obtain the ind-scheme $\operatorname{Gr}_{\widetilde{G}}$. Notice we also have:
			$$
			\operatorname{Gr}_{\widetilde{G}}  \cong \prod_{i=1}^f \operatorname{Gr}_{G \otimes_{W(k),\varphi^i} \mathcal{O}}
			$$
			Maintaining the notation from Part 1, we write $E(u) = \prod_{i=1}^e (u-\pi_i)$. Notice this coincides with the minimal polynomial of $\pi$ in $W(k)[u]$.
			\item For any $p$-adically complete $\mathcal{O}$-algebra $A$ we set $\mathfrak{S}_A := (W(k)\otimes_{\mathbb{Z}_p} A)[[u]]$ and equip this ring with the $A$-linear Frobenius $\varphi$ sending $u \mapsto u^p$ and lifting the $p$-th power map on $k$. We frequently identify 
			$$
			G(\mathfrak{S}_A) =\widetilde{G}(A[[u]]) = \prod_{i=1}^f G(A[[u]])
			$$
			and notice that the endomorphism of $G(\mathfrak{S}_A)$ induced by $\varphi$ on $\mathfrak{S}_A$ identifies with the automorphism of $ \prod_{i=1}^f G(A[[u]])$ given by $(g_i)_i \mapsto (\varphi'(g_i))_{i+1}$ where  the $i$ are viewed modulo $f$ and $\varphi'$ is the automorphism of $G(A[[u]])$ induced by the $A$-linear endomorphism of $A[[u]]$ given by $u \mapsto u^p$
			\item For any $p$-adically complete $\mathcal{O}$-algebra we also consider
			$$
			A_{\operatorname{inf},A} := \varprojlim_a \varprojlim_i( W(\mathcal{O}_{C^\flat})/p^a \otimes_{\mathbb{Z}_p} A)/u^i
			$$
			where $\mathcal{O}_{C^\flat} = \varprojlim_{x\mapsto x^p} \mathcal{O}_C/p$ and $u = [(\pi,\pi^{1/p},\pi^{1/p^2},\ldots)] \in W(\mathcal{O}_{C^\flat})$. We view $A_{\operatorname{inf},A}$ as an $\mathfrak{S}_A$-algebra via $u$ and note that the lift of Frobenius on $W(\mathcal{O}_{C^\flat})$ induces a Frobenius $\varphi$ on $A_{\operatorname{inf},A}$ which is compatible with that on $\mathfrak{S}_A$. The natural $G_K$-action on $\mathcal{O}_C$ also induces a continuous (for the $(u,p)$-adic topology) $G_K$-action on $A_{\operatorname{inf},A}$ commuting with $\varphi$. We also have 
			$$
			W(C^\flat)_A := \varprojlim_a A_{\operatorname{inf},A}[\tfrac{1}{u}]/p^a
			$$
			If $A$ is topologically of finite type (i.e. $A \otimes_{\mathbb{Z}_p} \mathbb{F}_p$ is of finite type) then $\mathfrak{S}_A \rightarrow A_{\operatorname{inf},A}$ is faithfully flat (in particular injective) \cite[2.2.13]{EG19}. We only consider the $A_{\operatorname{inf},A}$ and $W(C^\flat)_A$ for such $A$.
			\item Fix a compatible system $\epsilon = (\epsilon_1,\epsilon_2,\ldots)$ of primitive $p$-th power roots of unity in $C$. We view $\epsilon \in \mathcal{O}_{C^\flat}$ and we set $\mu = [\epsilon]-1 \in A_{\operatorname{inf}}$.
			\item  A Hodge type $\mu$ for $K$ is a tuple of conjugacy classes of cocharacters of $G$, indexed by the embeddings of $K$ into $\overline{\mathbb{Q}}_p$ (although we typically do not distinguish between such conjugacy classes and a given representative). Since every such embedding factors through $E$ we can (and frequently do) interpret a Hodge type as an $e$-tuple of cocharacters of $\widetilde{G}$. 
		\end{itemize}
	\end{sub}

	\section{Moduli of Breuil--Kisin modules}
	
	\begin{sub}
		For any $p$-adically complete $\mathcal{O}$-algebra $A$ a $G$-Breuil--Kisin module (or Breuil--Kisin module when the group $G$ is clear from context) over $A$ is a $G$-torsor $\mathfrak{M}$ on $\operatorname{Spec}\mathfrak{S}_A$ equipped with an isomorphism
		$$
		\varphi_{\mathfrak{M}}: \varphi^*\mathfrak{M}[\tfrac{1}{E(u)}] \xrightarrow{\sim}\mathfrak{M}[\tfrac{1}{E(u)}]
		$$
		We refer to $\varphi_{\mathfrak{M}}$ as the Frobenius on $\mathfrak{M}$ and frequently write $\varphi$ instead of $\varphi_{\mathfrak{M}}$ when there is no risk of confusion.
		\begin{itemize}
			\item Let $Z_G(A)$ be the category of Breuil--Kisin modules over $A$ whose morphisms are isomorphisms of $G$-torsors compatible with the Frobenii.
			\item Let $\widetilde{Z}_G(A)$ be the category of pairs $(\mathfrak{M},\iota)$ with $\mathfrak{M}$ a Breuil--Kisin module over $A$ and $\iota$ a trivialisation of $\mathfrak{M}$ over $\operatorname{Spec}\mathfrak{S}_A$. Morphisms are isomorphisms of $G$-torsors compatible with the Frobenii and commuting with the trivialisation.
		\end{itemize} 
		Any homomorphism of $p$-adically complete $\mathcal{O}$-algebras $A\rightarrow B$ induces a homomorphism $\mathfrak{S}_A \rightarrow \mathfrak{S}_B$ and pull back induces functors $Z_G(B) \rightarrow Z_G(A)$ and $\widetilde{Z}^N_G(B) \rightarrow \widetilde{Z}^N_G(A)$ making $Z_G$ and $\widetilde{Z}_G$ into categories fibred over $\operatorname{Spf}\mathcal{O}$. In the obvious way these constructions are functorial in $G$.
	\end{sub}

	\begin{construction}\label{con-PRdiagram}
		We have morphisms
		$$
		\begin{tikzcd}
			& \widetilde{Z}_G \ar[dr,"\Psi"]\ar[dl,"\Gamma"]& \\
			Z_G & & \operatorname{Gr}_{\widetilde{G}}
		\end{tikzcd}
		$$
		where $\Gamma$ forgets the choice of trivialisation and the map $\Psi$ is given on $A$-valued points by
		$$
		\Psi(\mathfrak{M},\iota) := (\mathfrak{M}, \varphi^*\iota \circ \varphi_{\mathfrak{M}}^{-1})
		$$
		Here we interpret $(\mathfrak{M}, \varphi^*\iota \circ \varphi_{\mathfrak{M}}^{-1})$, first as $\widetilde{G}$-torsor on $\operatorname{Spec}A[[u]]$ together with a trivialisation after inverting $E(u)$, and then as an $A$-valued point of $\operatorname{Gr}_{\widetilde{G}}$ via Lemma~\ref{lem-BL}. This crucially uses that $A$ is $p$-adically complete in order to identify $\mathfrak{S}_A$ with the $E(u)$-adic completion of $(W(k) \otimes_{\mathbb{Z}_p} A)[u]$.
	\end{construction}
	\begin{remark}\label{rem-Frob}
		If $(\mathfrak{M},\iota) \in \widetilde{Z}^N_G(A)$ then we obtain an element $ C_{\mathfrak{M},\iota} \in G(\mathfrak{S}_A[\frac{1}{E(u)}]) = \widetilde{G}(A[[u]][\frac{1}{E(u)}])$ giving the isomorphism
		$$
		\varphi^* \mathcal{E}^0 \xrightarrow{\varphi^* \iota^{-1} }\varphi^*\mathfrak{M}[\tfrac{1}{E(u)}] \xrightarrow{\varphi_{\mathfrak{M}}} \mathfrak{M}[\tfrac{1}{E(u)}] \xrightarrow{\iota} \mathcal{E}^0
		$$
		We say that $C_{\mathfrak{M},\iota}$ represents the Frobenius on $\mathfrak{M}$ relative to $\iota$.
	\end{remark}
	\begin{sub}
		For an alternative viewpoint on Construction~\ref{con-PRdiagram} let  $L^+G$ denote the group-scheme over $\mathcal{O}$ defined by $A \mapsto G(\mathfrak{S}_A)$and write $LG$ for the group ind-scheme over $\mathcal{O}$ given by $A\mapsto G(\mathfrak{S}_A[\frac{1}{E(u)}])$. Then $(\mathfrak{M},\iota) \mapsto C_{\mathfrak{M},\iota}$ gives an isomorphism $\widetilde{Z}_G \cong LG$ which identifies the diagram in Construction~\ref{con-PRdiagram} with
		$$
		\begin{tikzcd}
			& LG \ar[dr,"C \mapsto C^{-1}"]\ar[dl]& \\
			\left[LG/\prescript{}{\varphi}{L^+G}\right] & & \left[LG/L^+G\right] \cong \operatorname{Gr}_{\widetilde{G}}
		\end{tikzcd}
		$$
		where $LG/\prescript{}{\varphi}{L^+G}$ indicates the quotient by $\varphi$-conjugation and $LG/L^+G$ indicates the quotient by right multiplication. In particular, this shows that both $\Gamma$ and $\Psi$ from Construction~\ref{con-PRdiagram} are $L^+G$-torsors, with the actions given respectively by $\varphi$-conjugation and right multiplication.
	\end{sub}
	An issue with $\widetilde{Z}_G$ is that it is not of finite type over $\mathcal{O}$. To address this we will consider the certain quotients. These ideas go back to \cite[2.2]{PR09}. See also \cite[\S3.3]{Lin23}.
	
	\begin{definition}
		For $N \geq 1$ let $U_{G,N} \subset L^+G$ denote the subgroup with $A$-valued points 
		$$
		\operatorname{ker} \left( G(\mathfrak{S}_A) \rightarrow G(\mathfrak{S}_A/u^N) \right)
		$$
		and set $\mathcal{G}_{G,N} = L^+G /U_{G,N}$.
	\end{definition}
	
	\begin{proposition}\label{prop-torsor}
		Let $X \subset \operatorname{Gr}_{\widetilde{G}}$ be a closed subscheme on which $p$ is nilpotent. Then, for $N \geq N_0$ (with $N_0$ depending on $X$),
		$$
		[\widetilde{Z}_G \times_{\operatorname{Gr}_{\widetilde{G}}} X / \prescript{}{\varphi}{U_{G,N}}] \cong [\widetilde{Z}_G \times_{\operatorname{Gr}_{\widetilde{G}}} X / U_{G,N}]
		$$
		(the quotient on the left being by the $\varphi$-conjugation action and that on the right action by the right translation action). In particular, the map $\Psi$ induces a morphism
		$$
		\Psi_N: [\widetilde{Z}_G \times_{\operatorname{Gr}_{\widetilde{G}}} X / \prescript{}{\varphi}{U_{G,N}}] \rightarrow X
		$$
		which is a torsor for the group scheme $\mathcal{G}_{G,N}$.
	\end{proposition}
	\begin{proof}
		The isomorphism $[\widetilde{Z}_G \times_{\operatorname{Gr}_{\widetilde{G}}} X / \prescript{}{\varphi}{U_{G,N}}] \cong [\widetilde{Z}_G \times_{\operatorname{Gr}_{\widetilde{G}}} X / U_{G,N}]$ follows from the concrete assertion that there exists $N \geq 1$ so that for any $C \in LG(A) = G(\mathfrak{S}_A[\frac{1}{E(u)}])$ representing an $A$-valued point in $X$ one has:
		\begin{itemize}
			\item If $g_0 \in U_{G,N}(A)$ then $g_0^{-1} C \varphi(g_0) = gC$ for a unique $g \in U_{G,N}(A)$.
			\item If $g \in U_{G,N}(A)$ then there exists a unique $g_0 \in U_{G,N}(A)$ for which $g_0^{-1} C \varphi(g_0) = gC$.
		\end{itemize}
		When $G = \operatorname{GL}_n$ this is shown in \cite[9.6]{B21} (following arguments in \cite[2.2]{PR09}). For general $G$ one chooses a faithful representation into $\operatorname{GL}_n$. The first point then follows immediately from the statement for $\operatorname{GL}_n$. For the second point one recalls that, in the case of $\operatorname{GL}_n$, one constructs $g_0 \in U_{\operatorname{GL}_n,N}(A) = 1 + u^N\operatorname{Mat}(\mathfrak{S}_A)$ as the limit of a $u$-adically converging sequence of matrices in $U_{\operatorname{GL}_n,N}(A)$. If $g$ and $C$ are in $G(\mathfrak{S}_A[\frac{1}{E(u)}])$ then $g_0$ will be the limit of a convergent sequence in $G(\mathfrak{S}_A[\frac{1}{E(u)}]) \cap U_{\operatorname{GL}_n,N}(A) = U_{G,N}(A)$. Thus $g_0 \in U_{G,N}(A)$ also, and the proposition follows. 
	\end{proof}

	\begin{corollary}\label{cor-descend}
		Let $\mu$ be a Hodge type and assume that for each $\kappa_0:k \rightarrow \mathbb{F}$ 
		$$
		\sum_{i=1}^e\langle \alpha^\vee,\mu_{i} \rangle \leq p
		$$
		for all roots $\alpha^\vee$. Then there exists a closed subfunctor $Z_{G,\mu,\mathbb{F}}$ of $Z_G \otimes_{\mathcal{O}} \mathbb{F}$ represented by an algebraic stack, of finite type over $\operatorname{Spec}\mathbb{F}$, with the property that $\mathfrak{M} \in Z_{G,\mu,\mathbb{F}}(A)$ if and only if 
		\begin{itemize}
			\item For any $A$-algebra $A'$ and any trivialisation $\iota$ of $\mathfrak{M} \otimes_A A'$ one has $\Psi(\mathfrak{M} \otimes_A A',\iota) \in M_\mu \otimes_{\mathcal{O}} \mathbb{F}$ for $M_\mu \subset \operatorname{Gr}_{\widetilde{G}}$ defined as in Definition~\ref{def-mmu}.
		\end{itemize}
		Furthermore, $\operatorname{dim}Z_{G,\mu,\mathbb{F}} = \sum_{\kappa:K \rightarrow E} \operatorname{dim}\widetilde{G}/P_{\mu_\kappa}$.
	\end{corollary}
	\begin{proof}
		Applying Proposition~\ref{prop-torsor} with $X = M_\mu \otimes_{\mathcal{O}} \mathbb{F}$ shows that $[\widetilde{Z}_G \times_{\operatorname{Gr}_{\widetilde{G}}} X / \prescript{}{\varphi}{U_{G,N}}]$ is, for large enough $N$, a finite type $\mathbb{F}$-scheme of dimension 
		$$
		\operatorname{dim}\mathcal{G}_{G,N} + \sum_{\kappa:K \rightarrow E} \operatorname{dim}\widetilde{G}/P_{\mu_\kappa}
		$$
		To construct $Z_{G,\mu,\mathbb{F}}$ we descend this closed subscheme along the morphism $\Psi_N$. For this we need that $[\widetilde{Z}_G \times_{\operatorname{Gr}_{\widetilde{G}}} X / \prescript{}{\varphi}{U_{G,N}}]$ is stable under the $\varphi$-conjugation action $C_{\mathfrak{M},\iota} \mapsto g^{-1}C_{\mathfrak{M},\iota}\varphi(g)$ of $\mathcal{G}_{G,N}$. This is equivalent to asking that the $A$-valued points of each $M_\mu \otimes_{\mathcal{O}} \mathbb{F} \subset \operatorname{Gr}_{\widetilde{G}}$ are stable under the action of $\widetilde{G}(A[u^p])$.
		
		For this notice that if $\sum_{i=1}^e\langle \alpha^\vee,\mu_{i} \rangle \leq p$ then $g \in \widetilde{G}(A[u^p])$ acts on $Y_{\widetilde{G},\leq \mu} \otimes_{\mathcal{O}} \mathbb{F}$ as $g_0 := g$ modulo $u^p$ (this is clear from the definition). Since $M_\mu \otimes_{\mathcal{O}} \mathbb{F}\subset Y_{\widetilde{G},\leq \mu} \otimes_{\mathcal{O}} \mathbb{F}$ (see Proposition~\ref{prop-schubertcontain}) the claim reduces to the claim that $M_\mu \otimes_{\mathcal{O}} \mathbb{F}$ is stable under the action of $G$, and this is immediate.
	\end{proof}

	\begin{corollary}\label{cor-repble}
		Let $H \subset G$ be an embedding of reductive groups. Then the induced morphism
		$$
		Z_H \rightarrow Z_G
		$$
		is representable by schemes, and of finite type.
	\end{corollary}
	\begin{proof}
		The well-known fact that $\operatorname{Bun}_H \rightarrow \operatorname{Bun}_G$ is representable by schemes implies that, for any $A$-valued point of $Z_G$, $Z_H \times_{Z_G} \operatorname{Spec}A$ is representable by a closed subscheme of an $\mathfrak{S}_A$-scheme. To check this scheme is of finite type over $A$ we can assume that $A$ is a Noetherian $\mathcal{O}$-algebra $A$ on which $p$ is nilpotent. After replacing $A$ by an fppf-cover we can factor $\operatorname{Spec}A \rightarrow Z_G$ through $[\widetilde{Z}_G \times_{\operatorname{Gr}_{\widetilde{G}}} X / \prescript{}{\varphi}{U_{G,N}}]$ for sufficient large $N$ and some $X \subset \operatorname{Gr}_{\widetilde{G}}$. We can also assume $X$ is actually a closed subscheme of $\operatorname{Gr}_{\widetilde{H}}$. Writing $\widetilde{Z}^N_{G,X} := [\widetilde{Z}_G \times_{\operatorname{Gr}_{\widetilde{G}}} X / \prescript{}{\varphi}{U_{G,N}}]$ we then have a sequence of morphisms
		$$
		\widetilde{Z}^N_{H,X} \times_{\widetilde{Z}^N_{G,X}}  \times\operatorname{Spec}A \rightarrow Z_H \times_{Z_G} \operatorname{Spec}A \rightarrow \operatorname{Spec}A
		$$
		The composite is of finite type since  the same is true of $[\widetilde{Z}_H \times_{\operatorname{Gr}_{\widetilde{H}}} X / \prescript{}{\varphi}{U_{H,N}}] \rightarrow [\widetilde{Z}_G \times_{\operatorname{Gr}_{\widetilde{G}}} X / \prescript{}{\varphi}{U_{G,N}}])$ (in fact this is a closed immersion). As $A$ is Noetherian it follows that $Z_H \times_{Z_G} \operatorname{Spec}A$ is of finite type also.
	\end{proof}
	\begin{remark}
		For an embedding $H \subset G$ the analogous morphism between moduli spaces of shtuka's in representable by schemes, and additionally finite and unramified \cite{Breutmann,Yun22}. One expects that the same is true for $Z_H \rightarrow Z_G$, and it seems that the arguments of loc. cit. will go through largely unchanged. Since we do not need this additional level of control we do not try to give any details.
	\end{remark}
	
	\section{Crystalline Breuil--Kisin modules}

	Here we discuss the link between Breuil--Kisin modules and crystalline representations, by extending the discussion from \cite[\S10]{B21} from $\operatorname{GL}_n$ to $G$.
	\begin{definition}\label{def-crystGaloisactions}
		Let $A$ be a $p$-adically complete $\mathcal{O}$-algebra topologically of finite type and recall the $\mathfrak{S}_A$-algebra $A_{\operatorname{inf},A}$ which is equipped with a Frobenius extending that on $\mathfrak{S}_A$ and a continuous action of $G_K$ commuting with the Frobenius. By a crystalline action of $G_K$-action on $\mathfrak{M} \in Z_G(A)$ is one of the following two equivalent (in view of \cite[2.2.1]{BB20}) pieces of data:
		\begin{itemize}
			\item A compatible (with exact sequences and tensor products) collection of crystalline $G_K$-actions on $\mathfrak{M}^\chi$ for each $\mathcal{O}$-linear representation $\chi$ of $G$ in the sense of \cite[10.1]{B21}. In other words a continuous $\varphi$-equivariant $A_{\operatorname{inf},A}$-semilinear action of $G_K$ on $\mathfrak{M}^\chi \otimes_{\mathfrak{S}_A} A_{\operatorname{inf},A}$ for each $\chi$ satisfying
			$$
			(\sigma -1)(m) \in {\mathfrak {M}} \otimes _{{\mathfrak {S}}_A} [\pi ^\flat ]\varphi ^{-1}(\mu )A_{{\text {inf}},A}, \qquad (\sigma _{\infty } -1)(m) = 0 
			$$
			for every $m \in {\mathfrak {M}}$ and every $\sigma \in G_K, \sigma _{\infty } \in G_{K_\infty }$.
			\item A collection of isomorphisms of $G$-torsors $\mathfrak{M} \otimes_{\mathfrak{S}_A,\sigma} A_{\operatorname{inf},A} \xrightarrow{\sim} \mathfrak{M} \otimes_{\mathfrak{S}_A} A_{\operatorname{inf},A}$ for each $\sigma \in G_K$ so that, for any trivialisation of $\mathfrak{M}$ over an fppf cover $A' \rightarrow A$, the corresponding elements $A_\sigma \in G(A_{\operatorname{inf},A'})$ satisfy
			\begin{enumerate}
				\item $A_\sigma =1$ for $\sigma \in G_{K_\infty}$
				\item $A_\sigma \in \operatorname{ker}\left( G(A_{\operatorname{inf},A'}) \rightarrow G(A_{\operatorname{inf},A'}/u\varphi^{-1}(\mu)) \right)$ for $\sigma \in G_K$ and $\mu = [\epsilon]-1 \in A_{\operatorname{inf},A}$.
				\item $A_{\sigma\tau} = A_\sigma \tau(A_\sigma)$
				\item If $C \in G(\mathfrak{S}_{A'}[\frac{1}{E(u)}])$ is the matrix of Frobenius on $\mathfrak{M}$ corresponding to the trivialisation overr $A'$ then $C \varphi(A_\sigma) \sigma(C^{-1}) =A_\sigma$ for each $\sigma \in G_K$ and where $\varphi$ is the endomorphism of $G(A_{\operatorname{inf},A'})$ induced by $\varphi$ on $A_{\operatorname{inf},A'}$.
				\item The map $G_K \rightarrow G(A_{\operatorname{inf},A'})$ given by $\sigma  \mapsto A_\sigma$ is continuous for the topology induced by any choice of faithful embedding of $G$ into $\operatorname{GL}_n$.
			\end{enumerate}
		\end{itemize}
		Write $Y_G(A)$ for the groupoid consisting of $\mathfrak{M} \in Z_G(A)$ equipped with a crystalline $G_K$-action and $Y_G$ for the resulting limit preserving fppf stack over $\operatorname{Spf}\mathcal{O}$.
	\end{definition}

	We say that a continuous representation $\rho:G_K \rightarrow G(A)$ is crystalline if $\chi \circ \rho$ is crystalline in the sense of \cite{Fon94b} for every representation $\chi$ of $G$ (equivalently for a single faithful $\chi$).
	\begin{theorem}\label{thm-Kisin}
		Let $A$ be a finite flat $\mathcal{O}$-algebra. Then to each $(\mathfrak{M},x) \in Y_G(A)$ there is a uniquely determined (up to isomorphism) crystalline representation $\rho:G_K \rightarrow G(A)$ together with a $\varphi,G_K$-equivariant identification
		$$
		\mathfrak{M} \otimes_{\mathfrak{S}_A} W(C^\flat)_A \cong \rho \otimes_A W(C^\flat)_A
		$$
		of $G$-torsors. If $A$ is a discrete valuation ring then every crystalline $\rho$ arises in this way from a unique (up to isomorphism) such $(\mathfrak{M},x)$.
	\end{theorem}
	\begin{proof}
		The first part follows immediately from the assertion for $\operatorname{GL}_n$ (see \cite[2.1.12]{B19}). The second part does not immediately follow from the case of $\operatorname{GL}_n$ because the construction in \cite{Kis06} of $\mathfrak{M}$ from any $\rho:G_K \rightarrow \operatorname{GL}_n(A)$,  while functorial and tensor compatible, is not exact. In particular, the construction cannot, a priori, be used to associate the $G$-torsor $\mathfrak{M}$ to $\rho:G_K \rightarrow G(A)$. 
		
		Fortunately, this issue can easily be addressed because Kisin's construction actually produces an exact tensor functor associating to  any crystalline representation $\rho:G_K \rightarrow \operatorname{GL}_n$ a vector bundle $\mathfrak{M}(\rho)^*$ on $D^* = \operatorname{Spec}\mathfrak{S}_{\mathcal{O}} \setminus \lbrace u=p=0 \rbrace$ together with a Frobenius isomorphism after inverting $E(u)$. As explained in e.g. \cite[2.3.6]{Lev15}, such an $\mathfrak{M}(\rho)^*$ can be interpreted as a pair of projective $\varphi$-modules, respectively over $\mathfrak{S}_{\mathcal{O}}[\frac{1}{p}]$ and the $p$-adic completion $\mathcal{O}_{\mathcal{E},A}$ of $\mathfrak{S}_A[\frac{1}{u}]$, together with a comparison isomorphism over $\mathcal{O}_{\mathcal{E},A}[\frac{1}{p}]$. The former is constructed in \cite[1.3.15]{Kis06} (and we look into this construction in more detail in Section~\ref{sec-shape}) and the latter is the etale $\varphi$-module associated to $\rho$ as in e.g. \cite[\S 2.1]{Kis06}. Then $\mathfrak{M}(\rho)$ is obtained using the equivalence between vector bundles on $D^*$ and $\operatorname{Spec}\mathfrak{S}_\mathcal{O}$ \cite[1.2]{Ansch22}. This last step is where  exactness is lost.
		
		Since $\rho \mapsto \mathfrak{M}(\rho)^*$ is exact and tensor compatible applying the construction to a crystalline representation valued in $G(\mathcal{O})$ produces a $G$-torsor on $D^*$ equipped with a Frobenius after inverting $E(u)$. Then \cite[1.2]{Ansch22} can be applied to extend this $G$-torsor to a $G$-Breuil--Kisin module, producing the $\mathfrak{M}$ associated to $\rho$.
	\end{proof}
	
	We say that a crystalline representation $\rho:G_K \rightarrow G(A)$ has Hodge type $\mu$ if $\chi \circ \rho$ has Hodge type $\chi \circ \mu$ for any representation $\chi:G \rightarrow \operatorname{GL}_n$ of $G$. We emphasise that while it is enough to check $\rho$ is crystalline using a single faithful representation, this is not sufficient to determine the Hodge type. 
	
	\begin{proposition}
		For each Hodge type $\mu$ there exists a closed subfunctor $Y^\mu_G$ of $Y_G$ which is represented by an $\mathcal{O}$-flat $p$-adic algebraic formal stack (in the sense \cite[A7]{EG19}) $Y^\mu_G$ of topologically finite type over $\mathcal{O}$ and  is uniquely determined by the property that its groupoid of $A$-valued points, for any finite flat $\mathcal{O}$-algebra $A$, is canonically equivalent to the full subcategory $Y_G^\mu(A)$ of $Y_G(A)$ consisting of those $\mathfrak{M}$ for which the associated crystalline representation $\rho$ as in Theorem~\ref{thm-Kisin} has Hodge type $\mu$.
	\end{proposition}
	
	\begin{proof}
		When $G =\operatorname{GL}_n$ this is \cite[10.7]{B21}. For general $G$, first choose a faithful representation $\chi:G \rightarrow \operatorname{GL}_n$. Corollary~\ref{cor-repble} shows that the projection
		$$
		Y^{\chi \circ \mu}_{\operatorname{GL}_n} \times_{Z_{\operatorname{GL}_n}} Z_G \rightarrow Y^{\chi \circ \mu}_{\operatorname{GL}_n}
		$$
		is representable by finite type schemes, and so $Y^{\chi \circ \mu}_{\operatorname{GL}_n} \times_{Z_{\operatorname{GL}_n}} Z_G$ is a $p$-adic  formal algebraic stack of topologically finite type over $\mathcal{O}$. Since $\chi$ is faithful the $A$-points, for any $p$-adically complete and topologically finite type $\mathcal{O}$-algebra $A$, of $Y^{\chi \circ \mu}_{\operatorname{GL}_n} \times_{Z_{\operatorname{GL}_n}} Z_G$ can be interpreted as a full subcategory of $Y_G(A)$. Therefore, any representation  $\chi':G \rightarrow \operatorname{GL}_m$ gives a morphism $Y^{\chi \circ \mu}_{\operatorname{GL}_n} \times_{Z_{\operatorname{GL}_n}} Z_G \rightarrow Y_{\operatorname{GL}_m}$ and pullback along the closed subfunctor $Y^{\chi' \circ \mu}_{\operatorname{GL}_m}$ gives a closed substack of $Y^{\chi \circ \mu}_{\operatorname{GL}_n} \times_{Z_{\operatorname{GL}_n}} Z_G$ whose $A$-points, for any finite flat $\mathcal{O}$-algebra $A$, consist of $\mathfrak{M} \in Y_G(A)$ whose associated crystalline representation has Hodge type $\chi' \circ \mu$ when composed with $\chi'$. The intersection of these closed substacks over all representations $\chi'$ therefore has $A$-points, for any finite flat $\mathcal{O}$-algebra $A$, consisting of $\mathfrak{M} \in Y_G(A)$ with associated crystalline representation having Hodge type $\mu$. Finally, one sets $Y^\mu_G$ equal the  $\mathcal{O}$-flat closure of this closed substack (in the sense of \cite[p.230]{EG19}).
	\end{proof}
	
	We finish this section by explaining the relationship between $Y_G^\mu$ and $G$-crystalline deformation rings. Fix a continuous homomorphism $\overline{\rho}:G_K \rightarrow G(\mathbb{F})$ and let $R_{\overline{\rho}}^\square$ denote the corresponding framed deformation ring over $\mathcal{O}$, i.e. the unique complete local Noetherian $\mathcal{O}$-algebra with residue field $\mathbb{F}$ equipped with a continuous homomorphism $\rho^{\operatorname{univ}}:G_K \rightarrow G(R_{\overline{\rho}}^\square)$ satisfying $\rho^{\operatorname{univ}} \otimes_{R^\square_{\overline{\rho}}} \mathbb{F} = \overline{\rho}$ and universal amongst all such rings with this property.
	
	\begin{theorem}\label{thm-Kisindef}
		For each Hodge type $\mu$ there exists a unique $\mathcal{O}$-flat reduced quotient $R_{\overline{\rho}}^{\square,\operatorname{cr},\mu}$ of $R^\square_{\overline{\rho}}$ with the property that a homomorphism $R^\square_{\overline{\rho}} \rightarrow A$ with $A$ a finite flat $\mathcal{O}$-algebra factors through $R^{\square,\operatorname{cr},\mu}_{\overline{\rho}}$ if and only if the composite
		$$
		G_K \xrightarrow{\rho^{\operatorname{univ}}} G(R^{\square}_{\overline{\rho}}) \rightarrow G(A)
		$$
		is crystalline of Hodge type $\mu$. Furthermore,
		$$
		\operatorname{dim}_{\mathcal{O}} R^{\square,\operatorname{cr},\mu}_{\overline{\rho}} = \operatorname{dim}_{\mathcal{O}} G + \sum_{i=1}^e \operatorname{dim}\widehat{G}/P_{\mu_i}
		$$
	\end{theorem}
	\begin{proof}
		This is a special case of the main result of \cite{Kis08} when $G = \operatorname{GL}_n$ and of \cite{Bal12} for general $G$. See also \cite[Theorem A]{BG18}.
	\end{proof}
	
	\begin{construction}\label{con-Kisinresolution}
		Let $A$ be any complete local Noetherian $\mathcal{O}$-algebra with finite residue field and let $\rho:G_K \rightarrow G(A)$ be a continuous representation. Consider the functor which sends any $p$-adically complete $A$-algebra $B$ which is topologically of finite type over $\mathcal{O}$ to the set tuples $(\mathfrak{M},x,\alpha,\beta)$ for which $(\mathfrak{M},x) \in Y^\mu_G(B)$, $\alpha: A \rightarrow B$ is a continuous homomorphism, and $\beta$ is a $\varphi,G_{K}$-equivariant identification
		$$
		\mathfrak{M} \otimes_{\mathfrak{S}_{B}} W(C^\flat)_{B} \cong \rho \otimes_{A,\alpha} W(C^\flat)_{B}
		$$
		After choosing a faithful representation it follows from e.g. \cite[4.5.26]{EG19} that this functor is represented by the $\mathfrak{m}_A$-adic completion of a projective $A$-scheme  whose $\mathcal{O}$-flat closure we denote by $\mathcal{L}_\rho^\mu$. Then $\mathcal{L}_\rho^\mu$ has the property that the structure morphism $\mathcal{L}_\rho^\mu \rightarrow \operatorname{Spec}A$ becomes a closed immersion after inverting $p$. The scheme theoretic image of this morphism corresponds to a quotient $A^{\operatorname{cr,\mu}}$ of $A$ with the property that a homomorphism $A \rightarrow B$ into a finite flat $\mathcal{O}$-algebra $B$ factors through $\mathcal{L}_{\overline{\rho}}^\mu$ if and only if $G_K \xrightarrow{\rho} G(A) \rightarrow G(B)$ is crystalline of Hodge type $\mu$.
	\end{construction}
	
	For a given continuous $\overline{\rho}:G_K \rightarrow G(\mathbb{F})$ set $\mathcal{L}_{\overline{\rho}}^{\operatorname{cr},\mu} = \mathcal{L}_{\rho^{\operatorname{univ}}}$ for $\rho^{\operatorname{univ}}: G_K \rightarrow G(R_{\overline{\rho}}^\square)$ the universal deformation of $\rho$. Then Construction~\ref{con-Kisinresolution} shows that $R_{\overline{\rho}}^{\square,\operatorname{cr,\mu}}$ is the scheme theoretic image of $\mathcal{L}_{\overline{\rho}}^{\operatorname{cr},\mu}$.
	
	\begin{lemma}\label{lem-formalsmooth}
		For any continuous homomorphism $\overline{\rho}:G_K \rightarrow G(\mathbb{F})$ there is a formally smooth  morphism
		$$
		\mathcal{L}_{\overline{\rho}}^{\operatorname{cr},\mu} \otimes_{\mathcal{O}} \mathbb{F} \rightarrow Y_G^{\mu} \otimes_{\mathcal{O}} \mathbb{F}
		$$
		of relative dimension $\operatorname{dim}_{\mathcal{O}} G$ which induces, for any $p$-adically complete $\mathcal{O}$-algebra of topologically finite type over $\mathcal{O}$, the functor $(\mathfrak{M},x,\alpha,\beta) \mapsto (\mathfrak{M},x)$.
	\end{lemma}
	\begin{proof}
		The lifting property describing formal smoothness can be checked on $p$-adically complete $\mathcal{O}$-algebras factoring through $R_{\overline\rho}^{\square,\operatorname{cr,\mu}}$, and so we can assume the ring is a complete local Noetherian ring with finite residue field. The lifting can therefore be deduced from the main result of \cite{Dee}. This lifting is unique up to $G$-conjugation which shows that the relative dimension is as claimed.
	\end{proof}
	
	\begin{corollary}\label{cor-lifting}
		\begin{enumerate}
			\item Let $A$ be an Artin local $\mathcal{O}$-algebra and $(\mathfrak{M},x) \in Y_G^\mu(A)$. Then there exists a finite flat $\mathcal{O}$-algebra $A^\circ$, with a morphism $A^\circ \rightarrow A$,  and $(\mathfrak{M}^\circ,x^\circ) \in Y_G^\mu(A^\circ)$ whose image under $Y_G^\mu(A^\circ) \rightarrow Y_G^\mu(A)$ is $(\mathfrak{M},x)$.
			\item $Y_G^\mu \otimes_{\mathcal{O}} \mathbb{F}$ has dimension $\sum_{i=1}^e \operatorname{dim} \widetilde{G}/P_{\mu_i}$.
		\end{enumerate}
	\end{corollary}
	\begin{proof}
		Using Lemma~\ref{lem-formalsmooth} this follows from analogous assertions for $\mathcal{L}_\rho^\mu$, of which (1) is a consequence of \cite[4.1.2]{B19} and (2) is a consequence of the dimension formula for $R_{\overline{\rho}}^{\square,\operatorname{cr},\mu}$.
	\end{proof}
	\section{The shape of Frobenius}\label{sec-shape}
	
	For the results of this section it is necessary to assume that the compatible system $\pi^{1/p^\infty}$ is chosen so that $K_\infty \cap K(\mu_{p^\infty}) =K$ whenever $\mu^{p^\infty}$ is a compatible system of primitive $p$-th power roots of unity. When $p>2$ this is automatic and, while not automatic when $p=2$, it follows from \cite{Wang17} that $\pi$ can be chosen so that this is the case. 
	\begin{theorem}\label{thm-factorisation2}
		Assume that $\mu$ is a Hodge type satisfying
		$$
		\sum_{i=1}^e \langle \alpha^\vee, \mu_i \rangle \leq \frac{p-1}{\nu} +1, \qquad \nu = \operatorname{max}_{i\neq j}\lbrace v(\pi_i-\pi_j) \rbrace
		$$
		where $v$ denotes the valuation on $\mathcal{O}$ with $v(\pi_i) =1$ for one (equivalently all) $i$. Then the morphism $Y_G^\mu \otimes_{\mathcal{O}} \mathbb{F} \rightarrow Z_G$ which forgets the crystalline $G_K$-action factors through $Z_{G,\mu,\mathbb{F}}$.
	\end{theorem}

	\begin{remark}
		\begin{enumerate}
			\item 
			If $p$ does not divide $e$, i.e. if $K$ is tamely ramified over $\mathbb{Q}_p$, then $\nu =1$. To see this recall $E(u) = \prod_{i=1}^e (u-\pi_i)$ and so
			$$
			\frac{d}{du} E(u) = \sum_{i=1}^e \prod_{j \neq i} (u-\pi_j)
			$$
			Therefore $\frac{d}{du}E(u)|_{u=\pi_i} =\prod_{j \neq i}(\pi_i-\pi_j)$ has valuation $ \sum_{j \neq i} v(\pi_j-\pi_i)$. On the other hand, since $E(u) \equiv u^e$ modulo $p$ we have
			$$
			\frac{d}{du}E(u)|_{u=\pi_i} = e \pi_i^{e-1} \text{ modulo }p
			$$
			and so, if $e$ is prime to $p$, then $e-1 =  \sum_{j \neq i} v(\pi_j-\pi_i)$. As $v(\pi_j - \pi_i) \geq 1$ we must have each $v(\pi_j-\pi_i) = 1$.
			\item If each $\mu_i$ is strictly dominant then $\sum_{i=1}^e \langle \alpha^\vee, \mu_i \rangle \geq e$ and so for the bound in Theorem~\ref{thm-factorisation} to hold we must have
			$$
			e \leq \frac{p-1}{\nu} + 1
			$$
			In particular, (1) implies $\nu =1$. Therefore, the bound in Theorem~\ref{thm-factorisation} is equivalent to asking that
			$$
			\sum_{i=1}^e \langle \alpha^\vee, \mu_i \rangle \leq p
			$$
			for each root $\alpha^\vee$.
		\end{enumerate}
	\end{remark}
	
	In order to prove Theorem~\ref{thm-factorisation2} it suffices to show such a factorisation on the level of Artin local $\mathbb{F}$-algebras (see for example \cite[15.2]{B21}). Using the lifting result of Corollary~\ref{cor-lifting} we therefore reduce Theorem~\ref{thm-factorisation2} to the following:

	\begin{theorem}\label{thm-factorisation}
		Let $A$ be a finite flat $\mathcal{O}$-algebra and suppose that $(\mathfrak{M}(\rho),x) \in Y_G^\mu(A)$ corresponds as in Theorem~\ref{thm-Kisin} to a crystalline representation $\rho$ of Hodge type $\mu$. Assume that, for all roots $\alpha^\vee$ of $\widetilde{G}$,
		$$
		\sum_{i=1}^e \langle \alpha^\vee, \mu_i \rangle \leq \frac{p-1}{\nu} +1, \qquad \nu = \operatorname{max}_{i\neq j}\lbrace v(\pi_i-\pi_j) \rbrace
		$$
		where $v$ denotes the valuation on $\mathcal{O}$ with $v(\pi_i) =1$ for one (equivalently all) $i$. Then
		$$
		\Psi(\mathfrak{M}(\rho) ,\iota) \otimes_{\mathcal{O}} \mathbb{F} \in M_\mu (A \otimes_{\mathcal{O}} \mathbb{F})
		$$
		for any trivialisation $\iota$ of $\mathfrak{M}(\rho)$ 
	\end{theorem}
	
	The rest of this section will be devoted to the proof of the theorem. The first step is to realise the Hodge type $\mu$ in terms of $\mathfrak{M}(\rho)$. We will see that this is easy after inverting $p$.
	\begin{sub}\label{sub-rings1/p}
		We begin by introducing some power series rings in which $p$ had been inverted:
		\begin{itemize}
			\item Let $\widehat{\mathfrak{S}}_A$ denote the $E(u)$-adic completion of $\mathfrak{S}_A[\frac{1}{p}]$. Notice that since $E(u)$ generates the kernel of the surjection $\widehat{\mathfrak{S}}_A \rightarrow K\otimes_{\mathbb{Z}_p} A$ sending $u \mapsto \pi$ this surjection has a unique splitting, via which we view $\widehat{\mathfrak{S}}_A$ as a $K\otimes_{\mathbb{Z}_p} A$-module.
			\item Let $\widehat{\mathfrak{S}}_{A,i}$ denote the $(u-\pi_i)$-adic completion of $\mathfrak{S}_A[\frac{1}{p}]$ and identify this ring with $(K_0\otimes_{\mathbb{Z}_p} A)[[u-\pi_i]]$. As in~\ref{sub-productdecomp} we have an isomorphism $\widehat{\mathfrak{S}}_A \cong \prod_{i=1}^e \widehat{\mathfrak{S}}_{A,i}$ and this allows us to consider the Taylor expansion around $u=\pi_i$ of any $f \in \widehat{\mathfrak{S}}_A$: it is the power series 
			$$
			\sum_{n \geq 0} f_n (u-\pi_i)^n, \qquad f_n \in K_0 \otimes_{\mathbb{Z}_p} A
			$$
			corresponding to the image of $f$ in $\widehat{\mathfrak{S}}_{A,i}$.
			\item Let $\mathcal{O}^{\operatorname{rig}}$ denote the subring of $K_0[[u]]$ consisting of power series convergent on the open unit disk and set $\mathcal{O}^{\operatorname{rig}}_A = A \otimes_{\mathbb{Z}_p} \mathcal{O}^{\operatorname{rig}}$ whenever $A$ is a finite $\mathcal{O}$-algebra. We set $\lambda = \prod_{i=0}^\infty \frac{\varphi^n(E(u))}{E(0)} \in \mathcal{O}^{\operatorname{rig}}$ and we view $\mathcal{O}^{\operatorname{rig}}_A[\frac{1}{\varphi(\lambda)}]$ as an $\widehat{\mathfrak{S}}_A$-algebra by sending an element onto its Taylor series around $u=\pi$. We write $\varphi$ for the unique extension of $\varphi$ on $\mathfrak{S}_A$ to $\mathcal{O}^{\operatorname{rig}}_A$.
		\end{itemize}
		Notice that the composite $\mathcal{O}^{\operatorname{rig}}_{A}[\tfrac{1}{\varphi(\lambda)}] \rightarrow \widehat{\mathfrak{S}}_{A,i}$ is injective for each $i$ and so we frequently abuse notation by writing 
		$$
		f = \sum_{n\geq 0} f_n (u-\pi_i)^n
		$$
		whenever $f \in \mathcal{O}^{\operatorname{rig}}_{A}[\tfrac{1}{\varphi(\lambda)}]$.
	\end{sub}
	\begin{sub}\label{sub-Kisinconst}
		Next we recall some aspects of the construction of $\mathfrak{M}(\rho)[\frac{1}{p}]$ from \cite{Kis06} when $G = \operatorname{GL}_n$. First, the filtered $\varphi$-module $D(\rho)$ associated to $\rho[\frac{1}{p}]$ is used to construct a projective $\mathcal{O}^{\operatorname{rig}}_A = \mathcal{O}^{\operatorname{rig}} \otimes_{\mathbb{Z}_p} A$-module $\mathcal{M}(\rho)$ together with an isomorphism
		$$
		\varphi^*\mathcal{M}(\rho)[\tfrac{1}{\lambda}] \xrightarrow{\sim} \mathcal{M}(\rho)[\tfrac{1}{\lambda}]
		$$
		See \cite[\S 1.2]{Kis06}. There are two key consequences of this construction:
		\begin{itemize}
			\item There exists a $\varphi$-equivariant isomorphism
			$$
			\xi: \varphi^*\mathcal{M}(\rho)[\tfrac{1}{\varphi(\lambda)}] \cong D(\rho) \otimes_{K_0 \otimes_{\mathbb{Z}_p} A} \mathcal{O}^{\operatorname{rig}}_A[\tfrac{1}{\varphi(\lambda)}]
			$$
			See \cite[1.2.6]{Kis06}.
			\item After extending scalars to $\widehat{\mathfrak{S}}_A$ we obtain isomorphisms
			$\varphi^*\mathcal{M}(\rho) \otimes_{\mathcal{O}^{\operatorname{rig}}_A} \widehat{\mathfrak{S}}_A  \cong D(\rho) \otimes_{K_0 \otimes_{\mathbb{Z}_p} A} \widehat{\mathfrak{S}}_A \cong D(\rho)_K \otimes_{K \otimes_{\mathbb{Z}_p} A} \widehat{\mathfrak{S}}_A
			$
			under which 
			\begin{equation}\label{eq-fil-1}
				\begin{aligned}
					\mathcal{M}(\rho) \otimes_{\mathcal{O}^{\operatorname{rig}}_A} \widehat{\mathfrak{S}}_A = \sum_{i\in \mathbb{Z}} \operatorname{Fil}^i(D(\rho)_K) \otimes_{K \otimes_{\mathbb{Z}_p} A} E(u)^{-i} \widehat{\mathfrak{S}}_A
				\end{aligned}
			\end{equation}
			See \cite[1.2.1]{Kis06}.
		\end{itemize}
		Using that $D(\rho)$ comes from a crystalline representation (i.e. is an admissible filtered $\varphi$-module) Kisin then shows \cite[1.3.8]{Kis06} that $\mathcal{M}(\rho)$ descends uniquely to the projective $\mathfrak{S}_A[\frac{1}{p}]$-module $\mathfrak{M}(\rho)[\frac{1}{p}]$.
	\end{sub}

	All the constructions from~\ref{sub-Kisinconst} are functorial in $\rho$, and compatible with exact sequences and tensor products. Therefore, the Tannakian formalism ensures that the observations in~\ref{sub-Kisinconst} remain valid  (after interpreting \eqref{eq-fil-1} as in~\ref{sub-moduliinterpret}) when $\rho$ is valued in a general $G$. As a consequence we deduce:
	\begin{corollary}\label{cor-almostcontained}
		For any trivialisation $\beta$ of $D(\rho)$ over $\operatorname{Spec}K_0 \otimes_{\mathbb{Z}_p} A$, the pair
		$$
		(\mathfrak{M}(\rho) \otimes_{\mathfrak{S}_A} \widehat{\mathfrak{S}}_A,  \beta \circ \xi \circ \varphi_{\mathfrak{M}}^{-1})
		$$
		consisting of a $G$-torsor on $\operatorname{Spec}\widehat{\mathfrak{S}}_A$ and a trivialisation after inverting $E(u)$, defines an $A[\frac{1}{p}]$-valued point of $M_\mu \subset \operatorname{Gr}_{\widetilde{G}}$. Equivalently, if $X_{\xi,\beta}$ denotes the automorphism
		$$
		\mathcal{E}^0 \xrightarrow{\varphi^*\iota^{-1}} \varphi^*\mathfrak{M}(\rho) \otimes_{\mathfrak{S}_A} \mathcal{O}^{\operatorname{rig}}_A[\tfrac{1}{\varphi(\lambda)}] \xrightarrow{\xi} D(\rho) \otimes_{K_0 \otimes_{\mathbb{Z}_p} A} \mathcal{O}^{\operatorname{rig}}_A[\tfrac{1}{\varphi(\lambda)}] \xrightarrow{\beta} \mathcal{E}^0
		$$
		of the trivial $G$-torsor then 
		$$
		X_{\xi,\beta} \cdot \Psi(\mathfrak{M}(\rho),\iota)[\tfrac{1}{p}] \in M_\mu(A[\tfrac{1}{p}])
		$$
		This requires no bound on the Hodge type $\mu$.
	\end{corollary}
	
	In order to use Corollary~\ref{cor-almostcontained} to prove Theorem~\ref{thm-factorisation} we need to control the denominators appearing in $X_{\xi,\beta}$. Following ideas of \cite{GLS} we do this by first deriving some kind of intergrality of a differential operator associated to $X_{\xi,\beta}$. Set $S_{\operatorname{max}} = W(k)[[u,\frac{u^{e}}{p}]] \cap \mathcal{O}^{\operatorname{rig}}[\frac{1}{\lambda}]$ and $S_{\operatorname{max},A} = S_{\operatorname{max}} \otimes_{\mathbb{Z}_p} A$ for any finite $\mathcal{O}$-algebra $A$.
	\begin{proposition}\label{prop-Nnablamax}
		Assume $G = \operatorname{GL}_n$ and define a differential operator $N_\nabla$ on $\varphi^*\mathcal{M}(\rho)[\tfrac{1}{\varphi(\lambda)}] \cong D(\rho) \otimes_{K_0 \otimes_{\mathbb{Z}_p} A} \mathcal{O}^{\operatorname{rig}}_A[\frac{1}{\varphi(\lambda)}]$ over $\partial := u\frac{d}{du}$ by setting $N_\nabla(d) =0$ for $d \in D(\rho)$. The assumption that $K_\infty \cap K(\mu_{p^\infty}) = K$ ensures the matrix of $N_\nabla$ relative to the trivialisation $\varphi^*\iota$ of $\varphi^*\mathcal{M}(\rho)[\frac{1}{\varphi(\lambda)}]$ has entries in 
		$$
		u^p \varphi(S_{\operatorname{max},A})
		$$
		Again this requires no bound on the Hodge type $\mu$.
	\end{proposition} 
	\begin{proof}
		The proof will be given in Section~\ref{sec-filandmono} below. The essential idea is to relate $N_\nabla$ and the $G_K$-action on $\mathfrak{M} \otimes_{\mathfrak{S}_A} A_{\operatorname{inf},A}$ after basechanging to an appropriate period ring, and exploit the integrality of the $G_K$-action.
	\end{proof}
	\begin{remark}\label{rem-GLS}
		What is proved in \cite[4.7]{GLS} (when $p>2$) and \cite[4.1]{Wang17} (when $p=2$ and $\pi$ is chosen so that $K_\infty \cap K(\mu^{p^\infty}) = K$) is that the entries of the matrix representing $N_\nabla$ are contained in
		\begin{equation}\label{eq-GLSring}
			u^p\left( W(k)[[u^p,\tfrac{u^{ep}}{p}]][\tfrac{1}{p}] \cap S \right) \otimes_{\mathbb{Z}_p} A	
		\end{equation}
		where $S$ denotes the $p$-adic completion of the divided power envelope of $W(k)[u]$ with respect to the ideal generated by $E(u)$. This is slightly weaker than Proposition~\ref{prop-Nnablamax} (though for the purposes of this paper it makes no difference because the calculations in the first paragraph of Corollary~\ref{cor-denominatorbounds} below also go through using \eqref{eq-GLSring}, see \cite[2.3.9]{GLS15}). We have stated the stronger result here because it may be useful when considering Hodge types beyond the bounds imposed in this paper.
	\end{remark}

	\begin{corollary}\label{cor-denominatorbounds}
		Continue to assume $G = \operatorname{GL}_n$ and fix a trivialisation $\beta$ of $D(\rho)$ over $\operatorname{Spec}K_0\otimes_{\mathbb{Z}_p} A$. We then view the automorphism $X_{\xi,\beta}$ from Corollary~\ref{cor-almostcontained} as a matrix and, as in~\ref{sub-rings1/p}, write its Taylor expansion around $u=\pi_i$ as
		$$
		X_{\xi,\beta} = \sum_{n \geq 0} X_{i,n} (u-\pi_i)^n
		$$
		Then 
		$$
		X_{i,0}^{-1} X_{i,n}\in \operatorname{Mat}(\pi_i^{p-n}W(k) \otimes_{\mathbb{Z}_p} A)
		$$
		for $1 \leq n \leq p-1$.
	\end{corollary}
	\begin{proof}
		Let $\underline{e} = (e_1,\ldots,e_n)$ denote the standard basis of $\mathcal{E}^0$. Then $\xi(\varphi^*\iota^{-1}(\underline{e})) = \beta^{-1}(\underline{e}) X_{\xi,\beta}$. We can also write
		$$
		N_\nabla(\varphi^*\iota^{-1}(\underline{e})) = \varphi^*\iota^{-1}(\underline{e}) N
		$$
		and Proposition~\ref{prop-Nnablamax} ensures that the matrix $N$ has entries in $u^p\varphi(S_{\operatorname{max},A})$. Therefore,
		$$
		N = \sum_{n \geq 1} \frac{N_n'}{\pi_i^{n-1}} u^{pn}
		$$
		for matrices $N_m'$ with entries in $W(k) \otimes_{\mathbb{Z}_p} A$. If the Taylor expansion of $N$ around $u=\pi_i$ is $\sum_{m \geq 0} N_m (u-\pi_i)^m$ then
		$$
		N_m =\frac{1}{m!}\left( \frac{d}{du}\right)^m (N)|_{u=\pi_i} = \sum_{n \geq 1} \binom{pn}{m}N'_n \pi_i^{(p-1)n-m+1}
		$$
		for $m\geq 0$. Therefore $N_0 \in \pi_i^{p} \operatorname{Mat}(W(k)\otimes_{\mathbb{Z}_p} A)$ and $N_m \in \pi_i^{p+e-m}\operatorname{Mat}(W(k)\otimes_{\mathbb{Z}_p} A)$ for $m =1,\ldots,p-1$.
		
		By definition we have $N_\nabla( \xi^{-1} \circ \beta^{-1}(\underline{e})) = 0$ and so, recalling that $\partial =u\frac{d}{du}$, we have
		$$
		\begin{aligned}
			\varphi^*\iota^{-1}(\underline{e}) N &= N_\nabla( \varphi^*\iota^{-1}(\underline{e})) 
			\\ &= N_\nabla( \xi^{-1} \circ \beta^{-1}(\underline{e}) X_{\xi,\beta})
			\\
			&= \xi^{-1} \circ \beta^{-1}(\underline{e}) \partial(X_{\xi,\beta})\\
			&= \varphi^*\iota^{-1}(\underline{e}) X_{\xi,\beta}^{-1} \partial(X_{\xi,\beta})
		\end{aligned}
		$$
		In other words, $\partial(X_{\xi,\beta})  = X_{\xi,\beta}N$.
		In terms of Taylor expansions around $u=\pi_i$ this gives the recurrence $nX_{i,n} + \pi_i(n+1)X_{i,n+1} = \sum_{j=0}^n  X_{i,n-j}N_j$. Multiplying on the left by $X_{i,0}^{-1}$ gives
		$$
		nX_{i,0}^{-1} X_{i,n} + \pi_i(n+1)X_{i,0}^{-1} X_{i,n+1} = \sum_{j=0}^{n} X_{i,0}^{-1} X_{i,n-j}N_j
		$$
		for all $n \geq 0$. The corollary then follows by an easy induction, using the divisibility of the $N_j$ from the first paragraph.
	\end{proof}
	\begin{proof}[Proof of Theorem~\ref{thm-factorisation}]
		Recall that if $\beta$ is a trivialisation of $D(\rho)$ then
		$$
		X_{\xi,\beta} \cdot \Psi(\mathfrak{M}(\rho),\iota)[\tfrac{1}{p}] \in M_\mu(A[\tfrac{1}{p}])
		$$
		for $X_{\xi,\beta} \in G(\widehat{\mathfrak{S}}_A)$ as in Corollary~\ref{cor-almostcontained}. To prove the theorem we are going to show that, under the specified bounds on $\mu$, one can replace $X_{\xi,\beta}$ in this assertion with any $\widetilde{X}_{\operatorname{trun}} \in \operatorname{ker}\left( G(\mathfrak{S}_A) \rightarrow G(\mathfrak{S}_A \otimes_{\mathcal{O}} \mathbb{F}) \right)$.
		\subsubsection*{Step 1}
		Let $g \in G(K \otimes_{\mathbb{Z}_p} A)$ be the image of $X_{\xi,\beta}$ induced by the quotient $\widehat{\mathfrak{S}}_A \rightarrow K \otimes_{\mathbb{Z}} A$. Via the unique section of this surjection we can view $g \in G(\widehat{\mathfrak{S}}_A)$. We then set $\widetilde{X} = g^{-1} X_{\xi,\beta}$ and consider its image in $G(\widehat{\mathfrak{S}}_A/ \prod_{i=1}^e (u-\pi_i)^{n_i})$ for $n_i = \operatorname{max}\lbrace \langle \alpha^\vee, \mu_i \rangle \rbrace$, the maximum  being taken over all roots of $\widetilde{G}$.
		
		\subsubsection*{Step 2} We are going to show that 
		$$
		\widetilde{X}_{\operatorname{trun}} \cdot \Psi(\mathfrak{M}, \iota)[\tfrac{1}{p}] \in M_\mu(A[\tfrac{1}{p}])
		$$
		whenever $\widetilde{X}_{\operatorname{trun}} \in G(\widehat{\mathfrak{S}}_A)$ has image in $G(\widehat{\mathfrak{S}}_A/ \prod_{i=1}^e (u-\pi_i)^{n_i})$ equal to that of  $\widetilde{X}$. First note that, since $g$ is a constant matrix, it stabilises $M_\mu[\frac{1}{p}]$ and so $\widetilde{X} \cdot \Psi(\mathfrak{M},\iota)[\frac{1}{p}] \in M_\mu(A[\frac{1}{p}])$. Since any such $\widetilde{X}_{\operatorname{trun}} \in G(\widehat{\mathfrak{S}}_A)$ can be written as $Z_0 \widetilde{X}$ for 
		$$
		Z_0 \in \operatorname{ker} \left( G(\widehat{\mathfrak{S}}_A) \rightarrow G(\widehat{\mathfrak{S}}_A/ \prod_{i=1}^e (u-\pi_i)^{n_i})	 \right) 
		$$
		it suffices to show that any such $Z_0$ acts trivially on $M_\mu[\frac{1}{p}]$. Under the isomorphism $M_\mu[\frac{1}{p}] \cong \prod_{i=1}^e G/P_{\mu_i}$ the element $Z_0$ acts on on the $i$-th component by its image under the projection onto $G(\widehat{\mathfrak{S}}_{A,i})$. It therefore suffices to show that any element in the kernel of $G(\widehat{\mathfrak{S}}_{A,i}) \rightarrow G(\widehat{\mathfrak{S}}_{A,i}/(u-\pi_i)^{n_i})$ acts trivially on $G/P_{\mu_i}$. But this is clear from the choice of $n_i$.
		
		\subsubsection*{Step 3}
		
		We want to show $\widetilde{X}_{\operatorname{trun}}$ can be chosen as above so that it additionally lies in the image of $\operatorname{ker}\left( G(\mathfrak{S}_A) \rightarrow G(\mathfrak{S}_A \otimes_{\mathcal{O}} \mathbb{F}) \right)$. This will prove the theorem because if $\widetilde{X}_{\operatorname{trun}} \in G(\mathfrak{S}_A)$ then the containment $\widetilde{X}_{\operatorname{trun}} \cdot \Psi(\mathfrak{M},\iota) \in M_\mu(A[\frac{1}{p}])$ implies $\widetilde{X}_{\operatorname{trun}} \cdot \Psi(\mathfrak{M},\iota) \in M_\mu(A)$. The fact that $ \widetilde{X}_{\operatorname{trun}} \in \operatorname{ker}\left( G(\mathfrak{S}_A) \rightarrow G(\mathfrak{S}_A \otimes_{\mathcal{O}} \mathbb{F}) \right)$ further implies $\widetilde{X}_{\operatorname{trun}} \cdot \Psi(\mathfrak{M},\iota) \otimes_{\mathcal{O}} \mathbb{F} = \Psi(\mathfrak{M},\iota) \otimes_{\mathcal{O}} \mathbb{F}$. Therefore, $\Psi(\mathfrak{M},\iota) \otimes_{\mathcal{O}} \mathbb{F} \in M_{\mu} \otimes_{\mathcal{O}} \mathbb{F}$ as required.
		
		We claim that, after choosing a faithful representation of $G$, this question reduces to the case of $\operatorname{GL}_n$. To see this first note that $\widetilde{X}$ is functorial in the group $G$ since $\widetilde{X}_{\xi,\beta}$ is functorial in $G$ and the construction of $g \in G(K\otimes_{\mathbb{Z}_p} A)$ also commutes with any change of group. As a consequence, if we know the assertion of the previous paragraph when $G = \operatorname{GL}_n$ then, for any representation $G \rightarrow \operatorname{GL}_n$, the composite 
		$$
		G(\widehat{\mathfrak{S}}_A) \rightarrow G(\widehat{\mathfrak{S}}_A/ \prod_{i=1}^e (u-\pi_i)^{n_i}) \rightarrow \operatorname{GL}_n(\widehat{\mathfrak{S}}_A/ \prod_{i=1}^e (u-\pi_i)^{n_i})
		$$
		maps $\widetilde{X}$ into the image of $\operatorname{ker}\left( \operatorname{GL}_n(\mathfrak{S}_A) \rightarrow \operatorname{GL}_n(\mathfrak{S}_A \otimes_{\mathcal{O}} \mathbb{F}) \right)$. We want to choose the representation so this implies 
		$$
		G(\widehat{\mathfrak{S}}_A) \rightarrow G(\widehat{\mathfrak{S}}_A/ \prod_{i=1}^e (u-\pi_i)^{n_i})
		$$
		maps $\widetilde{X}$ into the image of $\operatorname{ker}\left( G(\mathfrak{S}_A) \rightarrow G(\mathfrak{S}_A \otimes_{\mathcal{O}} \mathbb{F}) \right)$. Since $\mathfrak{S}_A/\prod_{i=1}^e (u-\pi_i)^{n_i}$ is a subring of $\widehat{\mathfrak{S}}_A/ \prod_{i=1}^e (u-\pi_i)^{n_i}$ this will be the case whenever $G \rightarrow \operatorname{GL}_n$ is a faithful representation whose quotient $\operatorname{GL}_n/G$ is a scheme. 
		\subsubsection*{Step 4}
		
		To conclude we may therefore assume $G = \operatorname{GL}_n$. We can then view $X_{\xi,\beta}$ as a matrix and, for each $1\leq i \leq e$, consider its Taylor expansions 
		$$
		X_{\xi,\beta} = \sum_{n \geq 0} X_{i,n} (u-\pi_i)^n, \qquad X_{i,n} \in \operatorname{Mat}(K_0\otimes_{\mathbb{Z}_p} A)
		$$
		around $u =\pi_i$ as in Corollary~\ref{cor-denominatorbounds}. Under the isomorphism $G(K \otimes_{\mathbb{Z}_p} A) \cong \prod_{i=1}^e G(K_0\otimes_{\mathbb{Z}_p} A)$ we have $g =(X_{i,0})_i$ (where $g$ is the element from Step 1). Therefore $\widetilde{X} = g^{-1} X_{\xi,\beta}$ has Taylor expansions
		$$
		\widetilde{X}= \sum_{n \geq 0} X_{i,0}^{-1}X_{i,n} (u-\pi_i)^n
		$$
		around $u = \pi_i$.
		Corollary~\ref{cor-denominatorbounds} ensures $X_{i,0}^{-1}X_{i,n} \in \operatorname{Mat}(\pi_i^{p-n}W(k)\otimes_{\mathbb{Z}_p} A)$ for $n=1,\ldots,p-1$. Applying Lemma~\ref{lem-congruence} to the entries of $\widetilde{X}-1$ produces a matrix with entries in $ \pi_i \mathfrak{S}_A$ (for any choice of $i$) and whose image in $\operatorname{GL}_n(\mathfrak{S}_A/\prod_{i=1}^e (u-\pi_i)^{n_i})$ equals that of $\widetilde{X}-1$. This finishes the proof.
	\end{proof}
	\begin{lemma}\label{lem-congruence}
		Let $A$ be a finite flat $\mathcal{O}$-algebra and suppose $n_i \geq 0$ are such that
		$$
		\sum_{i=1}^e n_i \leq \frac{p-1}{\nu} + 1
		$$
		for $\nu$ as in Theorem~\ref{thm-factorisation}. Suppose that
		$$
		\frac{\widehat{\mathfrak{S}}_A}{\prod_{i=1}^{e} (u-\pi_i)^{n_i}} \cong \prod_{i=1}^e \frac{\widehat{\mathfrak{S}}_{A,i}}{(u-\pi_i)^{n_i}}
		$$
		maps an element $f$ onto $(f_i)_i$ with $f_i = \sum_{n=0}^{n_i-1} f_{i,n} (u-\pi_i)^n$ for $f_{i,n} \in \pi^{p -n}_iW(k) \otimes_{\mathbb{Z}_p} A$. Then $f$ is represented by an element in $\pi_i\mathfrak{S}_A$ (for any $i$).
	\end{lemma}
	\begin{proof}
		By linearity we can fix $1 \leq i \leq e$ and assume $f_j =0$ for $i \neq j$. We need to describe the inverse of the above isomorphism and so express $f$ in terms of $f_i$. For this we use the formal identity $\frac{1}{(1-y)^n} = \sum_{m \geq 0} \binom{n+m - 1}{n-1} y^m$. Setting $y = \frac{u- \pi_i}{\pi_i- \pi_j}$ shows that
		$$
		(u-\pi_j)^{n_j} \sum_{m=0}^{n_i -1} \binom{n_j + m -1}{n_j-1}\frac{(u-\pi_i)^{m}}{(\pi_i-\pi_j)^{m+n_j}} \equiv 1 \text{ modulo } (u-\pi_i)^{n_i}
		$$
		Therefore $f$ is represented by
		$$
		F := f_i \prod_{j \neq i} (u-\pi_j)^{n_j} \sum_{m=0}^{n_i -1} \binom{n_j + m -1}{n_j-1}\frac{(u-\pi_i)^{m}}{(\pi_i-\pi_j)^{m+n_j}} \text{ modulo } (u-\pi_i)^{n_i}
		$$
		(indeed,$ F \equiv f_i$ modulo $(u-\pi_i)^{n_i}$ and $F \equiv 0$ modulo $(u-\pi_j)^{n_j}$ for $j \neq i$). By hypothesis the coefficient of $(u-\pi_i)^n$ in $f_i$ has coefficient with $\pi_i$-adic valuation $\geq p-n \geq p-\nu n$ (since $\nu \geq 1$). On the other hand, the coefficient of $(u-\pi_i)^n$ in $\sum_{m=0}^{n_i -1} \binom{n_j + m -1}{n_j-1}\frac{(u-\pi_i)^{m}}{(\pi_i-\pi_j)^{m+n_j}}$ has $\pi_i$-adic valuation $\geq -(n+n_j)\nu$ whenever $i \neq j$. Therefore, the coefficient of $(u-\pi_i)^n$ in $F$ has $\pi_i$-adic valuation
		$$
		\geq p -\nu(n +\sum_{i \neq j} n_j)
		$$
		and so we will be done if $p -\nu(n +\sum_{i \neq j} n_j) \geq 1$ for all $n=0,\ldots,n_i-1$. In other words, if $p -\nu(-1 + \sum_{j=1}^e n_j) \geq 1$ or equivalently
		$$
		\sum_{j=1}^e n_j \leq \frac{p-1}{\nu} + 1
		$$
		which finishes the proof.
	\end{proof}
	\section{Constructing Galois actions}\label{sec-galois}
	
	In this section we equip $\mathfrak{M} \in Z_{G,\mu,\mathbb{F}}$ with a unique crystalline $G_K$-action, under a bound on $\mu$. The necessary bound is very slightly stronger than asking that $\sum_{i=1}^e \langle \alpha^\vee, \mu_i \rangle \leq p+e-1$ for each root $\alpha^\vee$. In order to formulate it recall that if a Hodge type $\mu$ corresponds to an $e$-tuple of cocharacters $(\mu_1,\ldots,\mu_e)$ of $\widetilde{G}$ then, since
	$$
	\widetilde{G} \cong \prod_{j=1}^f G \otimes_{W(k),\varphi^i} W(k)
	$$
	for $f$ the degree of $k/\mathbb{F}_p$, we can also view $\mu$ as a tuple $\mu_{ij}$ of cocharacters of $G$ for $1\leq j \leq e$ and $1 \leq i \leq f$.
	\begin{proposition}\label{prop-constGalois}
		Let $\mathfrak{M} \in Z_{G,\mu,\mathbb{F}}(A)$ with $A$ any finite type $\mathbb{F}$-algebra and  $\mu$ a Hodge type satisfying 
		$$
		\sum_{j=1}^e \langle \alpha^\vee, \mu_{ij} \rangle \leq p+e-1
		$$
		for each root $\alpha^\vee$ of $G$ and for each $1 \leq i \leq f$. Assume there is an $1 \leq i \leq f$  with the inequality strict for every $\alpha^\vee$. Then $\mathfrak{M}$ admits a unique crystalline $G_K$-action.
	\end{proposition}
	
	Before giving the proof we explain the propositions significance for us:
	\begin{corollary}\label{cor-closed immerson}
		Assume $\mu$ is as in Proposition~\ref{prop-constGalois}. Then the factorisation $Y_G^\mu \otimes_{\mathcal{O}} \mathbb{F} \rightarrow Z_{G,\mu,\mathbb{F}}$ from Theorem~\ref{thm-factorisation2} is a closed immersion.
	\end{corollary}
	\begin{proof}
		Proposition~\ref{prop-constGalois} implies that the morphism $Y_G \times_{Z_G} Z_{\mu,G,\mathbb{F}} \rightarrow Z_{\mu,G,\mathbb{F}}$ is an isomorphism. Since $Y_G^\mu$ is a closed subfunctor of $Y_G$ it follows that
		$$
		Y_G^\mu \times_{Z_G} Z_{G,\mu,\mathbb{F}} \rightarrow Z_{G,\mu,\mathbb{F}}
		$$
		is a closed immersion. But Theorem~\ref{thm-factorisation2} implies $Y_G^\mu \times_{Z_G} Z_{G,\mu,\mathbb{F}} = Y_G^\mu \otimes_{\mathcal{O}} \mathbb{F}$ so the corollary follows.
	\end{proof}
	
	We emulate the proof of  \cite[11.3]{B21}, which proves Proposition~\ref{prop-constGalois} in the case $G = \operatorname{GL}_n$.
	\begin{proof}[Proof of Proposition~\ref{prop-constGalois}]
		The claimed uniqueness of the $G_K$-action means it commutes with any descent datum on $\mathfrak{M}$. It therefore suffices to prove the proposition after pulling $\mathfrak{M}$ back along an fppf cover of $A$. This allows us to assume that $\mathfrak{M}$ admits a trivialisation $\iota$ over $\operatorname{Spec}\mathfrak{S}_A$. Let $C \in G(\mathfrak{S}_A[\frac{1}{E(u)}])$ be the corresponding matrix of Frobenius.
		
		For any $\mathcal{O}$-algebra $R$ and any $\varpi \in R$ we set $K_G(R,\varpi) = \operatorname{ker}\left( G(R) \rightarrow G(R/ \varpi) \right)$. Then the existence and uniqueness of such a $G_K$-action can be deduced from the following two assertions:
		\begin{enumerate}
			\item $C\sigma(C^{-1}) \in K_G(A_{\operatorname{inf,A}},u\varphi^{-1}\mu)$ for all $\sigma \in G_K$
			\item There is a sequence $\varpi_n \in A_{\operatorname{inf,A}}$ with $\varpi_0=1$ and converging $u$-adically to zero so that $x \mapsto C \varphi(x) C^{-1}$ induces a map 
			$$
			K_G(A_{\operatorname{inf,A}},\varpi_n u \varphi^{-1}(\mu)) \rightarrow K_G(A_{\operatorname{inf,A}}, \varpi_{n+1}u \varphi^{-1}(\mu))
			$$
			for each $n\geq 0$.
		\end{enumerate}
		To see that (1) and (2) prove the proposition we consider, for each $\sigma \in G_K$, the limit
		$$
		\begin{aligned}
			A_\sigma &:= \lim_{j \rightarrow \infty} \underbrace{C \varphi(C) \ldots \varphi^j(C) \varphi^j(\sigma(C^{-1})) \ldots \varphi(\sigma(C^{-1})) \sigma(C^{-1})}_{:= A_j} \\
			&= \prod_{j=1}^\infty \underbrace{C \varphi(C) \ldots \varphi^{j-1}(C)\varphi^j(C\sigma(C^{-1})) \varphi^{j-1}(C^{-1}) \ldots \varphi(C^{-1}) C^{-1}}_{= A_j A_{j-1}^{-1}}
		\end{aligned}
		$$
		Together (1) and (2) imply that the terms in the product converge $u$-adically to $1$, and so the entire product converges. Since each term in the product is in $K_G(A_{\operatorname{inf,A}},u\varphi^{-1}(\mu))$ so is $A_\sigma$. As $\sigma$ varies through $G_K$ the resulting $A_\sigma$ define a crystalline $G_K$-action as in Definition~\ref{def-crystGaloisactions}. This proves existence. For uniqueness, if $A_\sigma'$ was another such $G_K$-action then $x= A_\sigma (A_{\sigma}')^{-1} \in K_G(A_{\operatorname{inf,A}},u\varphi^{-1}(\mu))$ and satisfies $C\varphi(x)C^{-1}  = x$; consequently (2) implies $x = 1$.
		\subsubsection*{Verifying condition (2)} 
		This is almost identical to \cite[Lemma 6.3.5]{BB20}. Let $I \subset \mathcal{O}_G$ denote the kernel of the counit. Conjugation by $C$ defines an automorphism $\Phi_C$ of $\mathcal{O}_G \otimes_{W(k)} \mathfrak{S}_A[\frac{1}{E(u)}]$ which maps $I \otimes_{W(k)} \mathfrak{S}_A[\frac{1}{E(u)}]$ onto itself. We claim that if $x \in I$ then
		\begin{equation}\label{eq-inductiveident}
			\Phi_C(x) \in \sum_{j \geq 1} I^j \otimes_{W(k)} (u^{-jh^{(i)}}) \mathfrak{S}_A	
		\end{equation}
		where $(u^{-h^{(i)}} )_i  \in \mathfrak{S}_A$ denotes the corresponding element under the isomorphism $\mathfrak{S}_A \cong \prod_{i=1}^f A[[u]]$ and $h^{(i)} = \operatorname{max}_{\alpha^\vee} \sum_{j= 1}^{e} \langle \alpha^\vee, \mu_{ij} \rangle $.
		It suffices to check \eqref{eq-inductiveident} modulo $I^n$ for all $n \geq 1$. This can be done inductively once we know that, for any $x \in I$, 
		\begin{equation}\label{eq-adjoint}
			\Phi_C(x) \in I \otimes_{W(k)} (u^{-(n-1)h^{(i)}}) \mathfrak{S}_A \text{ modulo } I^n \otimes_{W(k)} \mathfrak{S}_A[\tfrac{1}{E(u)}]	
		\end{equation} 
		Indeed, \eqref{eq-adjoint} combined with the inductive hypothesis ensures that $\Phi_C(x)$ modulo $ I^n \otimes_{W(k)} \mathfrak{S}_A[\tfrac{1}{E(u)}]	$ is contained inside of 
		$$
		\left( \sum_{j=1}^{n-2} I^j \otimes_{W(k)} (u^{-jh^{(i)}}) \mathfrak{S}_A \right)  +  I \otimes_{W(k)} (u^{-(n-1)h^{(i)}}) \mathfrak{S}_A = \sum_{j=1}^{n-1} I^j \otimes_{W(k)} (u^{-jh^{(i)}}) \mathfrak{S}_A
		$$
		moduli $ I^n \otimes_{W(k)} \mathfrak{S}_A[\tfrac{1}{E(u)}]$.
		
		To deduce \eqref{eq-adjoint} we use that $\Psi(\mathfrak{M},\iota) \in Y_{\widetilde{G},\leq \mu} \otimes_{\mathcal{O}} \mathbb{F}$ (as defined in Definition~\ref{def-schubert}). More precisely, the map $x \mapsto \Phi_C(x)$ describes the action of $C$ on the $G$-representation $I$ and, since $I^{n-1}/I^{n} \cong \operatorname{Sym}^{n-1}(\mathfrak{g}^*)$, the maximal weights of $I/I^n$ are of the form $-(n-1)w_0 \alpha^\vee_{\operatorname{max}}$ for $\alpha^\vee_{\operatorname{max}}$ the maximal roots of $G$.  Recalling what it means to have $\Psi(\mathfrak{M},\iota) \in Y_{\widetilde{G},\leq \mu} \otimes_{\mathcal{O}} \mathbb{F}$ then gives \eqref{eq-adjoint}.
		
		Now any $g \in K_G(A_{\operatorname{inf},A}, \varpi_n u \varphi^{-1}(\mu))$ corresponds to a $W(k)$-algebra homomorphism $f: \mathcal{O}_G \rightarrow A_{\operatorname{inf},A}$ with $f(I) \subset \varpi_n u \varphi^{-1}(\mu) A_{\operatorname{inf},A}$. Then $\varphi(g)$ corresponds to the homomorphism $\varphi \circ f$ and $C \varphi(g) C^{-1}$ corresponds to $\varphi \circ f \circ \Phi_C$. We therefore have to find $\varpi_n$ so that 
		$$
		f(I) \subset \varpi_n u \varphi^{-1}(\mu) A_{\operatorname{inf},A} \Rightarrow \varphi \circ f \circ \Phi_C(I)\subset \varpi_{n+1} u \varphi^{-1}(\mu) A_{\operatorname{inf},A}
		$$
		with $\varpi_n \rightarrow 0$. We will show one can take $\varpi_{n+1} = \varphi(\varpi_n) ( u^{p+e-1 - h^{(i)}})_i$ for $h^{(i)}$ as above. The $\varpi_n$ converge to $0$ as $n \rightarrow \infty$ due to the assumption that $h^{(i)} \leq e+p-1$ with the inequality strict for at least one $i$. For the required implication we note that, as $f$ is a homomorphism of rings, \eqref{eq-inductiveident} ensures
		$$
		\begin{aligned}
			\varphi \circ f \circ \Phi_C(x) &\in \sum_{j \geq 1} \varphi( \varpi_n u \varphi^{-1}(\mu))^j \left( (u^{-h^{(i)}})_i \right)^j  A_{\operatorname{inf},A} \\
			&\subset\varphi(\varpi_n) u \varphi^{-1}(\mu)   (u^{p+e-1 - h^{(i)}})_i  A_{\operatorname{inf},A} \\
			&\subset \varpi_{n+1} u \varphi^{-1}(\mu) A_{\operatorname{inf},A} 
		\end{aligned}
		$$
		with the second inclusion using that, since $A$ is an $\mathbb{F}$-algebra, $\varphi(u\varphi^{-1}(\mu)) = u^p \mu$ is divisible by $\varphi^{-1}(\mu) u^{p +e}$.

		\subsubsection*{Verifying condition (1)}
		Condition (1) will be a consequence of the fact that $\Psi(\mathfrak{M},\iota) \in M_{\mu}$, and holds without any assumption on $\mu$. In fact, $M_\mu$ is contained inside a closed subscheme $\operatorname{Gr}_{\widetilde{G}}^{\nabla_\sigma} \subset \operatorname{Gr}_{\widetilde{G}}$ whose $A$-valued points, for any $p$-adically complete $\mathcal{O}$-algebra $A$ of topologically finite type, consists of those $(\mathcal{E},\iota) \in \operatorname{Gr}_{\widetilde{G}}^{\nabla_\sigma}(A)$ for which there exists an fpcq cover $A' \rightarrow A$ trivialising $\mathcal{E}$ so that
		$$
		(\mathcal{E},\iota) \otimes_A A' = (\mathcal{E}^0,C) \Leftrightarrow C \sigma(C)^{-1} \in U_{\operatorname{inf},A'}
		$$
		for every $\sigma \in G_K$. That this condition is closed, and that $M_\mu \subset \operatorname{Gr}_{\widetilde{G}}^{\nabla_\sigma}$ easily reduce, after choosing a faithful representation, to the case of $\operatorname{GL}_n$, where they are proved in \cite[7.4]{B21} and \cite[7.6]{B21}.
		
	\end{proof}

	\section{Cycle inequalities}
	
	Now we can prove the main theorem:
	
	\begin{theorem}\label{thm-cycleinequality}
		Assume that $G$ admits a twisting element $\rho$ and let $\overline{\rho}:G_K \rightarrow G(\mathbb{F})$ be a continuous homomorphism. Let $\mu$ be a Hodge type with each $\mu_i$ strictly dominant. Suppose also that
		$$
		\sum_{i=1}^e \langle \alpha^\vee,\mu_i \rangle  \leq p
		$$
		for each root $\alpha^\vee$ of $\widetilde{G}$. Then, as $\sum_{i=1}^e \operatorname{dim} \widetilde{G}/P_{\mu_i}$-dimensional cycles inside of $\operatorname{Spec}R_{\overline{\rho}}^{\square} \otimes_{\mathcal{O}} \mathbb{F}$, one has
		$$
		[\operatorname{Spec}R_{\overline{\rho}}^{\square,\operatorname{cr},\mu} \otimes_{\mathcal{O}} \mathbb{F}] \leq  \sum_{\lambda} m_\lambda [\operatorname{Spec}R_{\overline{\rho}}^{\square,\operatorname{cr},\widetilde{\lambda}} \otimes_{\mathcal{O}} \mathbb{F}]
		$$
		where
		\begin{itemize}
			\item The $\leq$ indicates that the difference is an effective cycle, i.e. a $\mathbb{Z}_{\geq 0}$-linear combination of integral closed subschemes.
			\item The sum runs over dominant cocharacters $\lambda$ of $\widetilde{G}$.
			\item $\widetilde{\lambda}$ denotes the Hodge type given by the $e$-tuple $(\lambda+\rho,\rho,\ldots,\rho)$.
			\item $m_\lambda$ denotes the multiplicity of $W(\lambda)$ inside $\bigotimes_{i=1}^e W(\mu_i- \rho)$. It follows from \cite[5.6]{Janbook}  and \cite[3.10]{Her09} that, due to the bound on $\mu$, $m_\lambda$ can equivalently be defined as the multiplicity of the representation of $\widehat{G}(\mathbb{F})$ obtained from the $\mathbb{F}$-valued points of $W(\lambda)$ inside that induced from the $\mathbb{F}$-valued points in $\bigotimes_{i=1}^e W(\mu_i- \rho)$.
		\end{itemize}
	\end{theorem}
	
	In the proof we use the standard functoriality of groups of cycles, namely the existence of a pullback homomorphism along flat morphisms and the pushforward along proper morphisms. See for example \cite[02R3, 02RA]{stacks-project}.
	
	\begin{proof}
		First, we can assume $e>1$ since when $e=1$ the theorem is vacuous. As a consequence the inequality $\sum_{i=1}^e \langle \alpha^\vee,\mu_i \rangle  \leq p$ ensures that Corollary~\ref{cor-closed immerson} is applicable.
		
		Theorem~\ref{thm-cycles} gives an identity of cycles $[M_\mu \otimes_{\mathcal{O}} \mathbb{F}] = \sum_{\lambda} m_\lambda [M_{\widetilde{\lambda}} \otimes_{\mathcal{O}} \mathbb{F}]$. For sufficiently large $N$ we can pull this identity back along the smooth morphism $\Psi_N$ by \cite[02R8]{stacks-project}, giving an equality 
		$$
		[\widetilde{Z}_G \times_{\operatorname{Gr}_{\widetilde{G}}} (M_{\mu} \otimes_{\mathcal{O}} \mathbb{F}) / \prescript{}{\varphi}{U_{G,N}}] = \sum_\lambda m_\lambda  [\widetilde{Z}_G \times_{\operatorname{Gr}_{\widetilde{G}}} (M_{\widetilde{\lambda}} \otimes_{\mathcal{O}} \mathbb{F}) / \prescript{}{\varphi}{U_{G,N}}]
		$$
		of $\operatorname{dim} \mathcal{G}_{G,N} + \sum_{i=1}^e \operatorname{dim} \widetilde{G}/P_{\mu_i}$-dimensional cycles. This identity then descends to an identity
		$$
		[Z_{\mu,G,\mathbb{F}}] = \sum_{\lambda} m_\lambda [Z_{\widetilde{\lambda},G,\mathbb{F}}]
		$$
		of $\sum_{i=1}^e \operatorname{dim}\widetilde{G}/P_{\mu_i}$-dimensional cycles. Note that here we are discussing cycles inside an algebraic stack, as opposed to a scheme. In this case a cycle is again a formal linear combination of integral closed substacks, with the notion of multiplicity as discussed in \cite[0DR4]{stacks-project}. We also observe that, since $\mathcal{G}_{G,N}$ is smooth and irreducible, the irreducibility and generic reducedness of $M_{\widetilde{\lambda}} \otimes_{\mathcal{O}} \mathbb{F}$ from Theorem~\ref{thm-irred} is shared by $Z_{\widetilde{\lambda},G,\mathbb{F}}$.
		
		Now Corollary~\ref{cor-closed immerson} implies that $[Y_{G}^\mu \otimes_{\mathcal{O}} \mathbb{F}] \leq [Z_{G,\mu,\mathbb{F}}]$. The irreducibility and generic reducedness of $Z_{G,\widetilde{\lambda},\mathbb{F}}$, together with the fact that $\operatorname{dim}Y_{G}^\mu \otimes_{\mathcal{O}} \mathbb{F} = \operatorname{dim} Z_{G,\mu,\mathbb{F}}$ implies that this is an equality when $\mu = \widetilde{\lambda}$. Therefore, we have
		$$
		[Y_G^\mu \otimes_{\mathcal{O}} \mathbb{F}] \leq \sum_{\lambda} m_\lambda [Y_G^{\widetilde{\lambda}} \otimes_{\mathcal{O}} \mathbb{F}]
		$$
		Pulling this identity back along the formally smooth morphism from Lemma~\ref{lem-formalsmooth} gives
		$$
		[\mathcal{L}_{\overline{\rho}}^{\operatorname{cr},\mu} \otimes_{\mathcal{O}} \mathbb{F}] \leq \sum_{\lambda} m_\lambda [\mathcal{L}_{\overline{\rho}}^{\operatorname{cr},\widetilde{\lambda}} \otimes_{\mathcal{O}} \mathbb{F}]
		$$ 
		Finally, pushing this identity forward along the proper morphism $\mathcal{L}_{\overline{\rho}}^{\operatorname{cr},\mu} \otimes_{\mathcal{O}} \mathbb{F} \rightarrow \operatorname{Spec}R_{\overline{\rho}}^\square \otimes_{\mathcal{O}} \mathbb{F}$ and using \cite[3.3]{B21} to equate the pushforward of $[\mathcal{L}_{\overline{\rho}}^{\operatorname{cr},\mu} \otimes_{\mathcal{O}} \mathbb{F}]$ with $[\operatorname{Spec}R_{\overline{\rho}}^{\square,\operatorname{cr},\mu} \otimes_{\mathcal{O}} \mathbb{F}]$ proves the theorem.
	\end{proof}
	\section{Monodromy and Galois}\label{sec-filandmono}
	
	Here we give a proof of the following equivalent formulation of Proposition~\ref{prop-Nnablamax}. For simplicity we work with $\mathbb{Z}_p$ coefficients but the extension to any coefficient ring which is finite and flat over $\mathbb{Z}_p$ is immediate.
	\begin{proposition}\label{prop-Nnablamax2}
		Let $\mathfrak{M}$ denote the Breuil--Kisin module associated to a crystalline representation $\rho:G_K \rightarrow \operatorname{GL}_n(\mathbb{Z}_p)$ and let $N_\nabla$ be the operator over $\partial=u\frac{d}{du}$ on $\mathfrak{M} \otimes_{\mathfrak{S}} \mathcal{O}^{\operatorname{rig}}[\frac{1}{\lambda}]$ induced from the $\varphi$-equivariant identification
		$$
		\mathfrak{M} \otimes_{\mathfrak{S}} \mathcal{O}^{\operatorname{rig}}[\tfrac{1}{\lambda}] \cong D(\rho)\otimes_{\mathfrak{S}} \mathcal{O}^{\operatorname{rig}}[\tfrac{1}{\lambda}]
		$$
		described in~\ref{sub-Kisinconst}. If $N_\nabla(\iota) = \iota N$ for an $\mathfrak{S}$-basis $\iota$ of $\mathfrak{M}$ then $N \in \frac{u}{p}\operatorname{Mat}(S_{\operatorname{max}})$ for $S_{\operatorname{max}} := W(k)[[u,\frac{u^e}{p}]]$.
	\end{proposition}
	
	As explained in Remark~\ref{rem-GLS}, the results of this section are not strictly speaking necessary for this paper but we still think they may be useful to help orient the reader.
	
	\begin{sub}The ideas go back to \cite{GLS}, with the new insight being that improved bounds can be achieved by replacing Fontaine's crystalline period ring $B_{\operatorname{crys}}$ with a better behaved period ring $B_{\operatorname{max}}$. This ring is defined in \cite[\S III]{Col98} by considering the subring $A_{\operatorname{max}}$ of $B_{\operatorname{dR}}^+$ consisting of elements which can be expressed as 	
		$$
		\sum_{n \geq 0} x_n \left( \frac{\nu}{p} \right)^n
		$$
		for $\nu$ any element generating the kernel of usual surjection $\Theta:A_{\operatorname{inf}} \rightarrow \mathcal{O}_C$ and $x_n \in A_{\operatorname{inf}}$ a sequence converging to zero for the $(p,u)$-adic topology on $A_{\operatorname{inf}}$. Note that $E(u)$ is one such generator of this kernel. Then $B_{\operatorname{max}}^+ = A_{\operatorname{max}}[\frac{1}{p}]$ and $B_{\operatorname{max}} = B_{\operatorname{max}}^+[\frac{1}{t}]$ for $t:= \operatorname{log}([\epsilon]) = \sum_{n\geq 1} (-1)^{n+1}\frac{([\epsilon]-1)^n}{n}$. The essential property that we will need is:
	\end{sub}
	\begin{lemma}\label{lem-intersectionisSmax}
		Recall that $\mathcal{O}^{\operatorname{rig}}$ denotes the ring of power series in $K_0[[u]]$ converging on the open unit disk, and $\lambda := \prod_{n \geq 0} \varphi^n(\frac{E(u)}{E(0)})^n$. Then the inclusion of $\mathfrak{S} \rightarrow A_{\operatorname{inf}}$ extends to an embedding of $\mathcal{O}^{\operatorname{rig}}[\frac{1}{\lambda}] \rightarrow B_{\operatorname{max}}$ so that 
		$$
		\mathcal{O}^{\operatorname{rig}}[\tfrac{1}{\lambda}] \cap A_{\operatorname{max}} \subset S_{\operatorname{max}}
		$$
	\end{lemma}
	\begin{proof}
		An easy computation shows that each $\frac{\varphi^n(E(u))}{p}$ is invertible in $S_{\operatorname{max}}$, and so it suffices to show $\mathcal{O}^{\operatorname{rig}}\cap A_{\operatorname{max}} \subset S_{\operatorname{max}}$. Any $f \in \mathcal{O}^{\operatorname{rig}}$ can be expressed uniquely as 
		$$
		f = \sum (\tfrac{E(u)}{p})^n q_n
		$$
		with $q_n \in K_0[u]$ polynomials of degree $<e$ converging $p$-adically to zero. We claim that $f \in A_{\operatorname{max}}$ if and only if each $q_n \in W(k)[u]$. This will prove the proposition because it will imply $f \in S_{\operatorname{max}}$. To see this we use a result of Colmez. Recall that $\Theta$ extends to a surjection $\Theta: B_{\operatorname{dR}}^+ \rightarrow C$ and, following \cite[\S V.3]{Col98}, we call an element $x \in B_{\operatorname{dR}}^+$ flat if $\theta(x) \neq 0$ and if $x \in p^{w(x)}A_{\operatorname{inf}}$ where $w(x)$ denotes the integer part of $v_p(\Theta(x))$. We also say zero is flat. If $q_n = \sum_{i=0}^{e-1} a_i u^i$ is non-zero then $\Theta(q_n) =\sum_{i=0}^{e-1} a_i \pi^i$ is non-zero and $w(\Theta(q_n)) = \operatorname{min} v_p(a_i)$. Thus, $q_n \in p^{w(\Theta(q_n))}A_{\operatorname{inf}}$ and so each $q_n$ is flat. Colmez shows in \cite[Lemme V.3.1]{Col98} that if $x \in B_{\operatorname{dR}}^+$ can be expressed as a sum $\sum_{n \geq 0} y_n (\frac{\nu}{p})^n$ with $\nu \in A_{\operatorname{inf}} \cap \operatorname{ker}\Theta$ a generator and $y_n \in B_{\operatorname{dR}}^+$ flat, then $x \in A_{\operatorname{max}}$ if and only if $w(y_n) \geq 0$ and converges to $\infty$. Since $E(u)$ is one possible generator of $\operatorname{ker}\Theta$ this gives the result.
	\end{proof}
	
	Combining Lemma~\ref{lem-intersectionisSmax} with the following gives Proposition~\ref{prop-Nnablamax2}.
	\begin{proposition}\label{prop-NnablaAmax}
		With notation as in Proposition~\ref{prop-Nnablamax2} one has $N \in \frac{u}{p\lambda} \operatorname{Mat}( A_{\operatorname{max}})$.
	\end{proposition}
	
	\begin{sub}
		To prove this first observe that $A_{\operatorname{max}}$ has a natural Frobenius $\varphi$ and a $\varphi$-equivariant $G_K$-action extending that on $A_{\operatorname{inf}}$. Furthermore, one has $\varphi(B_{\operatorname{max}}) \subset B_{\operatorname{crys}} \subset B_{\operatorname{max}}$ which shows that $B_{\operatorname{max}}$ can be used as a replacement for the crystalline period ring $B_{\operatorname{crys}}$. This means that there are $\varphi$-equivariant identifications
		\begin{equation}\label{eq-Bmaxcomparison}
			\rho \otimes_{\mathbb{Z}_p} B_{\operatorname{max}} \cong D(\rho) \otimes_{K_0} B_{\operatorname{max}} \cong \mathfrak{M} \otimes_{\mathfrak{S}} B_{\operatorname{max}}	
		\end{equation}
		with the first being $G_K$-equivariant for the trivial action of $G_K$ on $D(\rho)$. Here the second isomorphism is the base-change of the $\varphi$-equivariant isomorphism
		\begin{equation}\label{eq-Origcomparison}
			\mathfrak{M} \otimes_{\mathfrak{S}} \mathcal{O}^{\operatorname{rig}}[\tfrac{1}{\lambda}] \cong D(\rho) \otimes_{K_0} \mathcal{O}^{\operatorname{rig}}[\tfrac{1}{\lambda}]	
		\end{equation}
		described in~\ref{sub-Kisinconst}, while  the composite is obtained from the identification
		\begin{equation}\label{eq-Ainfcomparison}
			\mathfrak{M} \otimes_{\mathfrak{S}}W(C^\flat)\cong \rho \otimes_{\mathbb{Z}_p} W(C^\flat)
		\end{equation}
		in Theorem~\ref{thm-Kisin}, after applying \cite[4.26]{BMS} to descend this isomorphism to  $A_{\operatorname{inf}}[\frac{1}{\mu}]$, and then base-changing to $B_{\operatorname{max}}$.
	\end{sub}
	
	The key to proving Proposition~\ref{prop-NnablaAmax} is to relate, inside of \eqref{eq-Bmaxcomparison}, the $G_K$-action coming from $\rho$ with the operator $N_\nabla$ coming from $D(\rho)$:
	
	\begin{lemma}\label{lem-Galandmono}
		For any $\sigma \in G_K$ and $m \in \mathfrak{M} \otimes_{\mathfrak{S}} \mathcal{O}^{\operatorname{rig}}[\frac{1}{\lambda}]$ one has
		$$
		(\sigma - 1)(m) = \sum_{n \geq 1}N_\nabla^n(m) \otimes \frac{  \operatorname{log}([\epsilon(\sigma)])^n }{n!}
		$$
		Conversely, if $\sigma \in G_K$ acts trivially on $\mathbb{Z}_p(1) = \varprojlim \mu_{p^n}(C)$ (i.e. if the cyclotomic character $\chi_{\operatorname{cyc}}$ is trivial on $\sigma$) then
		$$
		N_\nabla(m) = \frac{1}{\operatorname{log}([\epsilon(\sigma)])}\sum_{n=1}^\infty (-1)^{n+1} \frac{(\sigma-1)^n}{n}(m)
		$$
		where $\epsilon(\sigma) \in \mathbb{Z}_p(1) = \sigma(u)/u$.
	\end{lemma}
	
	Here convergence of the sums is taken with respect to the topology on $B_{\operatorname{max}}^+$ with basis of open neighbourhoods of $0$ given by $p^nA_{\operatorname{max}}$. Since $A_{\operatorname{max}}$ is $p$-adically complete so is $B_{\operatorname{max}}^+$ for this topology. 
	\begin{proof}
		It suffices to check these identities for $m = d \otimes f$ for $d \in D(\rho)$ and $f \in \mathcal{O}^{\operatorname{rig}}[\frac{1}{\lambda}]$. Since $(\sigma-1)(d) = N_\nabla(d) = 0$ the lemma reduces to the claim that when $\chi_{\operatorname{cyc}}(\sigma) =1$
		$$
		\frac{1}{\operatorname{log}([\epsilon(\sigma)])}\sum_{n=1}^\infty (-1)^{n+1} \frac{(\sigma-1)^n}{n}(f)
		$$
		converges in $B_{\operatorname{max}}^+$ to $\partial(f)$ and, for any $\sigma \in G_K$,
		$$
		\sum_{n \geq 1}\partial^n(f) \otimes \frac{\left(  \operatorname{log}([\epsilon(\sigma)]) \right)^n }{n!}
		$$
		converges in $B_{\operatorname{max}}^+$ to 	$(\sigma-1)(f)$. It suffices to check either claim when $f = u^i$. For the first note that if $\chi_{\operatorname{cyc}}(\sigma)=1$ then $(\sigma-1)^n(u^i) = u^i([\epsilon(\sigma)]^i - 1)^n$ for all $n$. Therefore, the claimed convergence follows from the observation that $[\epsilon(\sigma)] -1 \in pA_{\operatorname{max}}$ when $p>2$ and that, when $p=2$, instead $[\epsilon(\sigma)] - 1 = ([\epsilon(\sigma)]^{1/2}-1)([\epsilon(\sigma)]^{1/2}+1) \in 4 A_{\operatorname{max}}$. For the second claim, we note that $\partial^n(f) = u^i$ and so $\sum_{n \geq 1}\partial^n(f) \otimes \frac{  \operatorname{log}([\epsilon(\sigma)])^n }{n!} = u^i\operatorname{exp}( \operatorname{log}([\epsilon(\sigma)]^i))$, which converges to $u^i [\epsilon(\sigma)]^i = \sigma(f)$.
	\end{proof}
	Next we prove the divisibility of the $G_K$-action asserted in Theorem~\ref{thm-Kisin}. Actually, we need something a little stronger:
	\begin{proposition}\label{prop-Galoisdivisibility}
		If $m \in \mathfrak{M}$ and $\sigma \in G_K$ then
		$$
		(\sigma-1)^n(m) \in \mathfrak{M} \otimes_{\mathfrak{S}} u\varphi^{-1}(\mu)^nA_{\operatorname{inf}}
		$$
		for $n=1$. If additionally $\chi_{\operatorname{cyc}}(\sigma) =1$ then this is true for all $n \geq 1$.
	\end{proposition}
	
	Here we will use that  the topology on $B_{\operatorname{max}}^+$ (in contrast to that on $B_{\operatorname{crys}}^+$) is well behaved. More precisely, one has \cite[Proposition III.2.1]{Col98} which implies that any principal ideal in $B_{\operatorname{max}}^+$ is closed.
	\begin{proof}
		We show the equivalent assertion that $(\sigma-1)^n(\varphi(m)) \in \mathfrak{M}^\varphi \otimes_{\mathfrak{S}} u^p \mu^n A_{\operatorname{inf}}$ for $\mathfrak{M}^\varphi$ the image of $\varphi^*\mathfrak{M}$ in $\mathfrak{M}[\frac{1}{E(u)}]$ under the Frobenius. Iterating the formula in Lemma~\ref{lem-Galandmono} shows that $(\sigma-1)^n(\varphi(m))$ can be expressed as
		\begin{equation}\label{eq-sigma-1}
			\sum_{j=n}^\infty \left(  \sum_{j_1+\ldots+j_n =j,j_i \geq 1} N_\nabla^j(m) \otimes \frac{\operatorname{log}([\epsilon(\sigma)])^j}{j_1!\ldots j_n!} \right)	
		\end{equation}
		for $n=1$ and any $\sigma \in G_K$ and, if $\chi_{\operatorname{cyc}}(\sigma) =1$, for all $n \geq 1$. As explained in~\ref{sub-Kisinconst}, \eqref{eq-Origcomparison} arises from an identification $\mathfrak{M}^\varphi \otimes_{\mathfrak{S}} \mathcal{O}^{\operatorname{rig}}[\frac{1}{\varphi(\lambda)}] \cong D(\rho) \otimes_{K_0} \mathcal{O}^{\operatorname{rig}}[\frac{1}{\varphi(\lambda)}]$. This means that $N_\nabla^j(\varphi(m)) \in \mathfrak{M}^\varphi \otimes_{\mathfrak{S}} u^p\mathcal{O}^{\operatorname{rig}}[\frac{1}{\varphi(\lambda)}]$ for each $j \geq 1$ and so  each term of \eqref{eq-sigma-1} is contained in $u^p t^j B_{\operatorname{max}}^+$. Since all principal ideals in $B_{\operatorname{max}}^+$ are closed it follows that the entire sum is contained in $\mathfrak{M}^\varphi \otimes_{\mathfrak{S}} u^p t^n B_{\operatorname{max}}^+$ also.
		
		On the other hand, since \eqref{eq-Ainfcomparison} descends to an isomorphism over $A_{\operatorname{inf}}[\frac{1}{\mu}]$, we also know that $(\sigma-1)(\varphi(m)) \in \mathfrak{M}^\varphi \otimes_{\mathfrak{S}} A_{\operatorname{inf}}[\frac{1}{\mu}]$. The proposition will therefore follow from the assertion that
		$$
		A_{\operatorname{inf}}[\tfrac{1}{\mu}] \cap u^p t^n B_{\operatorname{max}}^+ = u^p \mu^n A_{\operatorname{inf}}
		$$
		To prove this first note that $\frac{t}{\mu}$ is a unit in $A_{\operatorname{max}}$ by \cite[Lemme III.3.9]{Col98}. Therefore, we need to show that if $a \in A_{\operatorname{inf}} \cap \mu^n B_{\operatorname{max}}^+$ then $a \in \mu^n A_{\operatorname{inf}}$ and if $a \in A_{\operatorname{inf}} \cap u^n B_{\operatorname{max}}$ then $a \in u^nA_{\operatorname{inf}}$. The first claim follows from the fact \cite[Proposition 5.1.3]{Fon94} that $\mu$ generates the ideal consisting of those $x \in A_{\operatorname{inf}}$ with $\varphi^n(x) \in \operatorname{ker}\theta$ for all $n \geq 0$.  For the second claim we use \cite[Lemma 3.2.2]{Liu10b}, which shows $u^nB_{\operatorname{crys}}^+ \cap A_{\operatorname{inf}} = u^n A_{\operatorname{inf}}$. Since $\varphi(B_{\operatorname{max}}^+) \subset B_{\operatorname{crys}} \subset B_{\operatorname{max}}$, if $b \in u^nB_{\operatorname{max}}^+ \cap A_{\operatorname{inf}}$ then $\varphi(b) \in u^{pn}B_{\operatorname{crys}}^+ \cap A_{\operatorname{inf}} = u^{pn}A_{\operatorname{inf}}$. Thus $b \in u^nA_{\operatorname{inf}}$, as required.
	\end{proof}
	
	Finally we can prove:
	
	\begin{proof}[Proof of Proposition~\ref{prop-NnablaAmax}]
		The assumption that $K_\infty \cap K(\mu_{p^\infty}) = K$ ensures that $\sigma  \in G_K$ can be found with $\epsilon(\sigma)$ equal the fixed generator $\epsilon \in \mathbb{Z}_p(1)$ and $\chi_{\operatorname{cyc}}(\sigma)=1$. For such a $\sigma$ we have
		$$
		N_\nabla(m) = \frac{1}{t}\sum_{n=1}^\infty (-1)^{n+1} \frac{(\sigma-1)^n}{n}(m)
		$$
		for any $m \in \mathfrak{M}$. By Proposition~\ref{prop-Galoisdivisibility} we know $(\sigma-1)^n(m) \in \mathfrak{M} \otimes_{\mathfrak{S}} u \varphi^{-1}(\mu)^nA_{\operatorname{inf}}$ for each $n \geq 1$. We are going to show that each term in the above sum, and hence the sum itself, is contained $\mathfrak{M} \otimes_{\mathfrak{S}} \frac{u\varphi^{-1}(\mu)}{t} A_{\operatorname{max}}$.
		
		For this claim it suffices to show that $\varphi^{-1}(\mu)^{n-1} \in n A_{\operatorname{max}}$. Since $\varphi^{-1}(\mu)^p \equiv \mu$ modulo $pA_{\operatorname{inf}}$ it follows that $\alpha := \frac{\varphi^{-1}(\mu)^p}{p} - \frac{\mu}{p} \in A_{\operatorname{inf}}$. Since $\varphi^n(\alpha) \in \operatorname{ker}\theta$ for all $n \geq 1$ we know that $\varphi(\alpha)$ is divisible by $\mu$ in $A_{\operatorname{inf}}$. Hence $\frac{\alpha}{\varphi^{-1}(\mu)} = \frac{\varphi^{-1}(\mu)^{p-1}}{p} - \frac{\mu}{\varphi^{-1}(\mu)p} \in A_{\operatorname{inf}}$. Since $\frac{\mu}{\varphi^{-1}(\mu)}$ generates the kernel of $\Theta$ it follows that $\frac{\varphi^{-1}(\mu)^{p-1}}{p} \in A_{\operatorname{max}}$, and so the claim holds when $n=p$. For general $n$, we write $n = p^s m$ for $m$ coprime to $p$. Since $p^s-1 = (p-1)(1+p+\ldots+p^{s-1})$ we have $n-1 \geq p^s-1 \geq (p-1)s$ and so $\frac{\varphi^{-1}(\mu)^{n-1}}{p^s} \in A_{\operatorname{max}}$ which proves the claim.
		
		It remains only to show that $\frac{\varphi^{-1}(\mu) }{t}A_{\operatorname{max}} = \frac{1}{p\lambda} A_{\operatorname{max}}$. We showed above that $\mu$ generates the same ideal of $A_{\operatorname{max}}$ as $t$ so this is equivalent to showing that $\frac{\mu}{\varphi^{-1}(\mu)p}$ and $\lambda$ generate the same ideal. As $\frac{\mu}{\varphi^{-1}(\mu)}$ and $E(u)$ generate the same ideal in $A_{\operatorname{inf}}$ this is equivalent to asking that $\varphi(\lambda)$ is a unit in $A_{\operatorname{max}}$. But this is clear because $\frac{\varphi^n(E(u))}{E(0)} - 1$ is topologically nilpotent in $A_{\operatorname{max}}$ for $n\geq 1$.
	\end{proof}
	
	\bibliography{/Users/robin.bartlett/Library/CloudStorage/Dropbox/Maths/biblio.bib}
\end{document}